\documentclass[openright]{book}
\usepackage{amsopn,amstext,amsbsy,amsmath,amscd,amsthm,amsfonts}
\usepackage{color}
\usepackage{url}
\usepackage[english]{babel}
\usepackage{enumitem}
\usepackage{booktabs}
\usepackage{hyperref}
\usepackage{multirow}
\usepackage{afterpage}
\usepackage{float}
\usepackage{wrapfig}
\usepackage{tikz}
\usetikzlibrary{matrix,arrows.meta}
\usepackage{pdfpages}
\usepackage[export]{adjustbox}
\usepackage[ma,mnf,master]{mccover} % MC framside generator
\usepackage{amsfonts, amsmath, amssymb}
\usepackage{geometry} % packet för att ändra margialbredd
\usepackage{indentfirst}

\theoremstyle{plain}
\newtheorem{theorem}                 {\bf Theorem}      [chapter]
\newtheorem{proposition}  [theorem]  {\bf Proposition}
\newtheorem{corollary}    [theorem]  {\bf Corollary}
\newtheorem{lemma}        [theorem]  {\bf Lemma}

\theoremstyle{definition}
\newtheorem{example}      [theorem]  {\bf Example}
\newtheorem{definition}   [theorem]  {\bf Definition}

\newtheorem{observation}  [theorem]  {\bf Observation}
\newtheorem{remark}       [theorem]  {\bf Remark}

\newtheorem{conjecture}   [theorem]  {\bf Conjecture}

\allowdisplaybreaks

\def\nab#1#2#3{\nabla^{\hbox{$\scriptstyle{#1}$}}_{\hbox{$\scriptstyle{#2}$}}{\hbox{$#3$}}}

\def\snab#1#2#3{\hbox{$\nabla$\kern-.1em\raise 1.2 ex\hbox{$\scriptstyle{#1}$}\kern-.5em\lower 0.8 ex\hbox{$#2$}\kern-.0em{$#3$}}}

\def\nsnab#1{\hbox{$\nabla$\kern-.1em\raise 1.0 ex\hbox{$\scriptstyle{#1}$}}}

\def\rn{\mathbb R}

\def\g{\mathfrak{g}}

\def\nab#1#2{\hbox{$\nabla$\kern -.3em\lower 1.0 ex
		\hbox{$#1$}\kern -.1 em {$#2$}}}

\def\snab#1#2#3{\hbox{$\nabla$\kern-.1em\raise 1.2 ex\hbox{$\scriptstyle{#1}$}\kern-.5em\lower 0.8 ex\hbox{$#2$}\kern-.0em{$#3$}}}

\def\nsnab#1{\hbox{$\nabla$\kern-.1em\raise 1.0 ex\hbox{$\scriptstyle{#1}$}}}

\def \cn{{\mathbb C}}
\def \C{{\mathbb C}}
\def \hn{{\mathbb H}}
\def \H{{\mathbb H}}

\def \rn{{\mathbb R}}
\def \R{{\mathbb R}}

\def \H{\mathcal H}

\def \cn{{\mathbb C}}
\def \C{{\mathbb C}}
\def \hn{{\mathbb H}}
\def \H{{\mathbb H}}

\def \rn{{\mathbb R}}
\def \R{{\mathbb R}}

\def\nab#1#2{\hbox{$\nabla$\kern -.3em\lower 1.0 ex
		\hbox{$#1$}\kern -.1 em {$#2$}}}

\def \Re{\mathfrak R\mathfrak e}
\def \Im{\mathfrak I\mathfrak m}

\def \a{\mathfrak{a}}
\def \b{\mathfrak{b}}

\def \g{\mathfrak{g}}

\def \k{\mathfrak{k}}

\def \m{\mathfrak{m}}

\def \o{\mathfrak{o}}

\DeclareMathOperator{\ad}{ad}
\DeclareMathOperator{\Ad}{Ad}
\DeclareMathOperator{\End}{End}

\DeclareMathOperator{\diff}{\mathrm{d}}

\def \GLR#1{\mathbf{GL}_{#1}(\rn)}

\def \glr#1{\mathfrak{gl}_{#1}(\rn)}
\def \GLC#1{\mathbf{GL}_{#1}(\cn)}
\def \glc#1{\mathfrak{gl}_{#1}(\cn)}
\def \GLH#1{\mathbf{GL}_{#1}(\hn)}

\def \SLR#1{\mathbf{SL}_{#1}(\rn)}
\def \SL2{\widetilde{\text{\bf SL}}_{2}(\rn)}
\def \slr#1{\mathfrak{sl}_{#1}(\rn)}
\def \SLC#1{\mathbf{SL}_{#1}(\cn)}
\def \slc#1{\mathfrak{sl}_{#1}(\cn)}

\def \O#1{\mathbf{O}(#1)}
\def \SO#1{\mathbf{SO}(#1)}
\def \so#1{\mathfrak{so}(#1)}

\def \Spin#1{\text{\bf Spin}(#1)}

\def \U#1{\text{\bf U}(#1)}
\def \u#1{\mathfrak{u}(#1)}

\def \SU#1{\text{\bf SU}(#1)}
\def \su#1{\mathfrak{su}(#1)}

\def \Sp#1{\text{\bf Sp}(#1)}
\def \sp#1{\mathfrak{sp}(#1)}

\DeclareMathOperator{\Aut}{Aut}

\DeclareMathOperator{\Div}{div} 
\DeclareMathOperator{\grad}{grad}

\DeclareMathOperator{\trace}{trace}
\DeclareMathOperator{\rank}{rank}
\DeclareMathOperator{\Id}{Id}
\DeclareMathOperator{\pf}{pf}
\DeclareMathOperator{\GL}{GL}

\allowdisplaybreaks

\def\rn{\mathbb R}

\def\g{\mathfrak{g}}

\def \GLR#1{\mathbf{GL}_{#1}(\rn)}

\def \glr#1{\mathfrak{gl}_{#1}(\rn)}
\def \GLC#1{\mathbf{GL}_{#1}(\cn)}
\def \glc#1{\mathfrak{gl}_{#1}(\cn)}
\def \GLH#1{\mathbf{GL}_{#1}(\hn)}

\def \SLR#1{\mathbf{SL}_{#1}(\rn)}
\def \SL2{\widetilde{\text{\bf SL}}_{2}(\rn)}
\def \slr#1{\mathfrak{sl}_{#1}(\rn)}
\def \SLC#1{\mathbf{SL}_{#1}(\cn)}
\def \slc#1{\mathfrak{sl}_{#1}(\cn)}

\def \O#1{\mathbf{O}(#1)}
\def \SO#1{\mathbf{SO}(#1)}
\def \so#1{\mathfrak{so}(#1)}

\def \Spin#1{\text{\bf Spin}(#1)}

\def \U#1{\text{\bf U}(#1)}
\def \u#1{\mathfrak{u}(#1)}

\def \SU#1{\text{\bf SU}(#1)}
\def \su#1{\mathfrak{su}(#1)}

\def \Sp#1{\text{\bf Sp}(#1)}
\def \sp#1{\mathfrak{sp}(#1)}

\begin{document}
\newpage
\thispagestyle{empty}
\maintitle{Minimal Submanifolds \\ of the Classical Compact \\ Riemannian Symmetric Spaces}
\authors{Johanna Marie Gegenfurtner}

\issnum{2024}{17}
\sernum{LUNFMA}{3150}{2024}

\frontcover
%\includepdf[pages=1]{Cover.pdf}
%\includepdf[pages=1]{2024 Gegenfurtner E17.pdf}
%\centering
%\includegraphics[width=\textwidth]{Cover.pdf}

%\newpage
%\thispagestyle{empty}
%\phantom{m}
%\newpage
\large 
%%%%%%%%%%%%%%%%%%%% abstract
\thispagestyle{empty}
\centerline {\bf\Large Abstract}\vskip1cm
Minimal submanifolds constitute a central area within the realm of differential geometry, due to their many applications in various branches of physics. An illustrative example is found in the behavior of soap films, which in striving to balance pressure, naturally assume surfaces characterised by zero mean curvature.

In this thesis we will employ a recent result of S. Gudmundsson and T.J. Munn to construct minimal submanifolds of the classical compact Riemannian symmetric spaces using eigenfunctions. In some of those, we obtain families of compact minimal submanifolds. 

Chapter 1 serves as an introduction to our work. It is followed by four chapters which aim to familiarise the reader with different aspects of the theoretical background. 
In Chapter 2 we will discuss some basic facts on harmonic morphisms. We then proceed with Lie groups and Lie algebras in Chapter 3, with a special focus on the classical matrix Lie groups. Further, we give a brief overview of Riemannian symmetric spaces in Chapter 4. In Chapter 5 we take a closer look at the above mentioned result of Gudmundsson and Munn, which we aim to employ.\smallskip

In the following part we apply their theorem to ten different cases of symmetric spaces. In Chapters 6-8, we look at the compact Lie groups $$\SO{n},\ \SU{n}, \ \Sp{n}.$$ In Chapters 7-12, we study the symmetric spaces $$\SU n/\SO n,\  \Sp n/\U n,\ \SO {2n}/\U n ,\  \SU{2n}/\Sp n, $$ respectively. In all those spaces we construct families of compact minimal submanifolds.

In Chapters 13-15, we investigate minimal submanifolds of the real, complex and quaternionic Grassmannians $$ \SO{m+n}/\SO{m}\times\SO{n}, \ \U{m+n}/\U{m}\times\U{n},$$ $$\Sp{m+n}/\Sp{m}\times\Sp{n}.$$
In the Appendix we clarify our notation and prove some useful lemmas.
\vskip1cm
{\it Throughout this work it has been my firm intention to give reference to the stated results and credit to the work of others. All theorems, propositions, lemmas and examples left unmarked are either assumed to be well known, or  are the fruits of my own efforts.}
%%%%%%%%%%%%%%%%%%%

\newpage % tom sida
\thispagestyle{empty}
\phantom{m}

\newpage % tacksida
\thispagestyle{empty}
\centerline {\bf\Large Acknowledgments}
\vskip 1cm
First of all, I would like to thank my supervisor Sigmundur Gudmundsson. He once passed on to me the wisdom of James Eells, that mathematics is to be taken seriously, but also has to be fun. True to this conviction, he chose an interesting and challenging topic for my thesis, which I greatly enjoyed working on. I also immensely appreciate how supportive and kind he has been to me as a junior mathematician.\smallskip\newline
I would also like to thank Thomas Jack Munn and Adam Lindström of Lund's compact but positively oriented Differential Geometry group for the many helpful discussions, their brilliant comments and suggestions and their camaraderie.\smallskip\newline
Thanks are also owed to my dear friends, in particular Alex and Emily. They have been a great help in all kinds of mathematical and non-mathematical matters.  \smallskip\newline
Lastly, I want to express my deepest gratitude to my family, especially my parents and my older brother (who keeps setting the bar so damn high) for their unwavering support and encouragement. \smallskip\newline
I humbly dedicate this thesis to my late grandmothers Marianne and Rosemarie Jeanne Sophie, who have been great inspirations to me. 
\vskip1pc
\hskip9cm Johanna Marie Gegenfurtner
\phantom{m}

\newpage % tom sida
\thispagestyle{empty}

\newpage % innehåll
 % sätt noll för chapter 1 sida 3
\tableofcontents
\thispagestyle{empty}
\newpage
\thispagestyle{empty}
\phantom{m}

%%%%%%%%%%%%%%%%%%%%%%%%%%%Ändras här%%%%%%%%%%%%%%%%%%%%%%%%%%%%%%%%%%

%\maketitle
\setcounter{page}{0}

\chapter{Introduction}
%\begin{wrapfigure}{r}{0.3\textwidth}
%  \centering
%   \includegraphics[width=0.3\textwidth,trim={0 3.5cm 0 5cm},clip]{dsc_1540.eps}

 %       \caption{A catenoid. Photo courtesy of {\it soapbubble.dk}}
%\end{wrapfigure}
In the three-dimensional Euclidean space, a surface is minimal if, given certain constraints, it minimises the area. Due to the many applications in physics and engineering, generations of mathematicians have studied minimal surfaces. Among them, Lagrange characterised minimal surfaces with the {\it Lagrange equation} (\cite{Lag}, 1760). He showed that if a surface, locally parameterised by $$z=z(x,y),$$ is minimal, it satisfies $$(1+z_y^2)z_{xx}-2z_xz_yz_{xy}+(1+z_x^2)z_{yy}=0.$$
Later, this was generalised to minimal submanifolds of other ambient spaces. %\smallskip
\begin{definition}{\rm \cite{Gud-Rie}}
Let $(N,h)$ be a Riemannian manifold and $M$ a submanifold with the induced metric. Let $\{X_1,\dots,X_m\}$ be a local orthonormal frame for the tangent bundle $TM.$ The submanifold $M$ is said to be {\it minimal} if its second fundamental form
$$B:C^\infty(TM)\otimes C^\infty(TM)\rightarrow C^\infty(NM)$$ is traceless,
i.e. $$\trace B=\sum_{k=1}^m B(X_k,X_k)=0.$$
\end{definition}
More than two hundred years after Lagrange, 
Eells and Sampson provided the following result (1964).
\begin{theorem}{\rm \cite{Bai-Woo-book}}
    A weakly conformal map from a Riemannian manifold of dimension two is harmonic if and only if its image is minimal at regular points.
\end{theorem}
This establishes an important connection to harmonic maps from Riemannian manifolds. Recall that a $C^2$-regular function $f:U\rightarrow\mathbb{R}$, where $U\subset\mathbb{R}^n,$ is harmonic if it satisfies the Laplace equation $$\Delta f=\frac{\partial^2 f}{\partial x_1^2}+\dots+\frac{\partial^2 f}{\partial x_n^2}=0.$$ A harmonic morphism $\phi:M\rightarrow N$ on the other hand is a smooth mapping between Riemannian manifolds such that
for every harmonic function $f:V\rightarrow\mathbb{R}$ defined on an open subset $V$ of $N$ with $\phi^{-1}(V)$ nonempty, the composition $f\circ\phi$ is harmonic on $\phi^{-1}(V).$ 
In Chapter 1 we will give a more detailed introduction to harmonic morphisms, mainly based on Baird and Wood's classic textbook \cite{Bai-Woo-book}.

Chapter 2 serves as an introduction to Lie groups and Lie algebras. A particular focus is put on matrix Lie groups, which will play an important role in this thesis.

In Chapter 3 we will define Riemannian symmetric spaces, which are the main site of action of this thesis. 

For $C^2$-regular functions $\phi,\psi:M\rightarrow\mathbb{C}$ on a Riemannian manifold, we define the Laplace-Beltrami operator $\tau$ by $$\tau(\phi)=\Div\nabla\phi,$$ 
and the conformality operator $\kappa$ by $$\kappa(\phi,\psi)=g(\nabla\phi,\nabla\psi).$$ We see that a harmonic morphism $\phi:M\rightarrow\cn$  satisfies $$\tau(\phi)=0\ \ {\rm and }\ \ \kappa(\phi,\phi)=0.$$
This motivates the definition of eigenfunctions.
A function $\phi:M\rightarrow\mathbb{C}$ on a Riemannian manifold $(M,g)$ is an eigenfunction if there exist $\lambda,\mu\in\mathbb{C}$ such that $$\tau(\phi)=\lambda\cdot\phi\  {\rm and } \ \kappa(\phi,\phi)=\mu\cdot\phi^2.$$
We will discuss eigenfunctions in more detail in Chapter 4.
In their paper \cite{Gud-Sak-1} from 2008, Gudmundsson and Sakovich constructed eigenfunctions on $\SO{n}, \SU{n},$ and $ \Sp{n}.$ Gudmundsson, Siffert and Sobak found eigenfunctions on the Riemannian symmetric spaces $$\SU n/\SO n,\ \ \Sp n/\U n, \ \ \SO {2n}/\U n \ \ \textrm{and} \ \ \SU{2n}/\Sp n$$ in their work \cite{Gud-Sif-Sob-2} from 2022. In their articles \cite{Gha-Gud-4} and \cite{Gha-Gud-5} from 2023, Ghandour and Gudmundsson constructed eigenfunctions from the real, complex and quaternionic Grassmannians. 

As shown by Gudmundsson and Munn in their work \cite{Gud-Mun-1} from 2023, eigenfunctions can be used to construct minimal submanifolds: 
\begin{theorem}{\rm \cite{Gud-Mun-1}}
    Let $\phi:(M,g)\rightarrow\mathbb{C}$ be a complex-valued eigenfunction on a Riemannian manifold, such that $0\in\phi(M)$ is a regular value for $\phi.$ Then the fibre $\phi^{-1}(\{0\})$ is a minimal submanifold of $M$ of codimension two.
\end{theorem} 
The aim of this thesis is to apply the previous result to the ten classical families of compact Riemannian symmetric spaces: 
the compact Lie groups $$\SO{n},\ \ \SU{n},\ \ \Sp{n},$$ the symmetric spaces
$$\SU n/\SO n,\ \ \Sp n/\U n,\ \ \SO {2n}/\U n,\ \ \SU{2n}/\Sp n, $$ and the real, complex and quaternionic Grassmannians
$$ \SO{m+n}/\SO{m}\times\SO{n},\ \U{m+n}/\U{m}\times\U{n},\ \Sp{m+n}/\Sp{m}\times\Sp{n}.$$ 
We have obtained the following results.
On the symmetric spaces $\SU n/\SO n$, $\Sp n/\U n$,   $\SO {2n}/\U n$ and $\SU{2n}/\Sp{n}$ we have constructed families of compact minimal submanifolds. 
In their paper \cite{Gud-Mun-1}, Gudmundsson and Munn have given some easy examples of the application of Theorem \ref{GM} on the compact Lie groups $\SO{n},$ $\SU{n}$ and $\Sp{n}$. We generalise those examples and construct families of compact minimal submanifolds.

On the complex Grassmannians $\U{m+n}/\U{m}\times\U{n},$ we were able to construct new eigenfunctions. However, we also showed that none of the known eigenfunctions on $$ \SO{m+n}/\SO{m}\times\SO{n},\ \U{m+n}/\U{m}\times\U{n},\  \Sp{m+n}/\Sp{m}\times\Sp{n}$$ are regular over $0\in\cn.$ 

\begin{figure}[htbp]\label{catenoid}
        \centering    \includegraphics[width=0.4\textwidth,trim={0 1,5cm 0 2,4cm},clip]{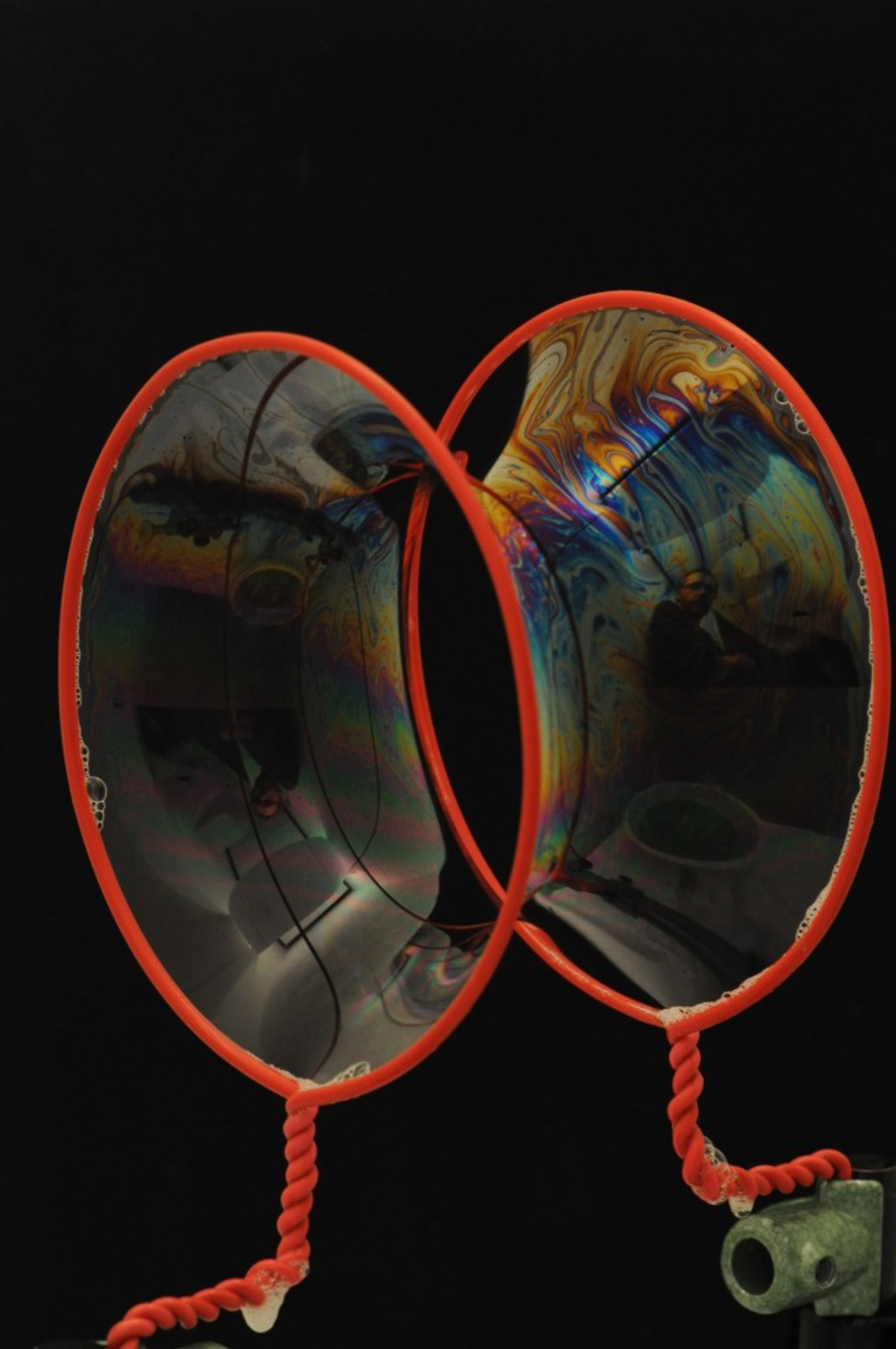}
        \includegraphics[width=0.4\textwidth]{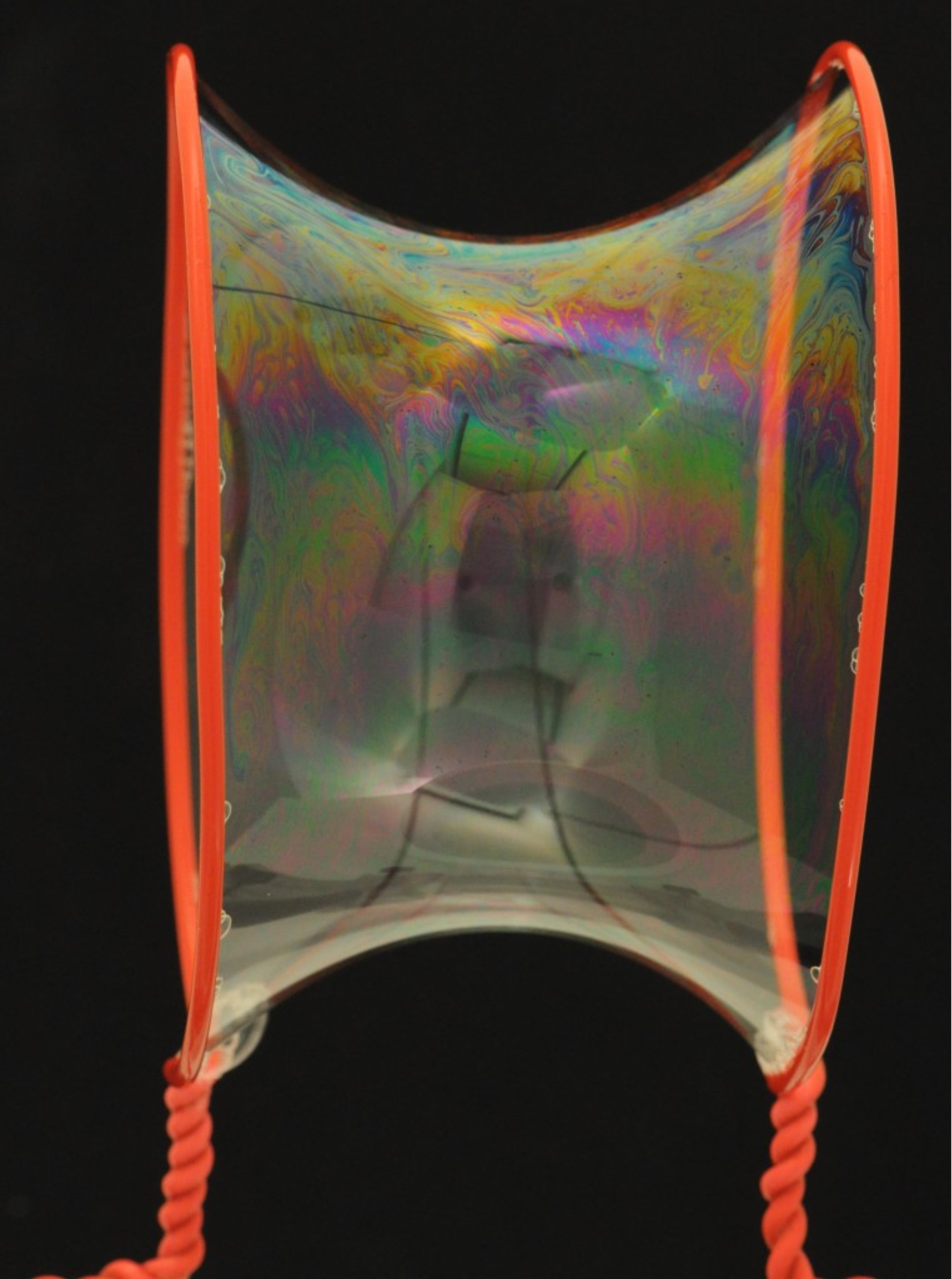}
        \caption{A catenoid, one of the first known non-trivial minimal surfaces. It is well known that soap film spanning wire frames takes the shape of a minimal surface. Photo courtesy of {\it soapbubble.dk}}
\end{figure}

%%%%%%%%%%%%%%%%%%%%%%%%%%%%%%%%%%%%%%%%%%%
\chapter{Harmonic Morphisms}

In this chapter we give an introduction to harmonic morphisms. We will start off with some basic definitions the reader might be familiar with.
We will define the Laplace-Beltrami and conformality operators. This will lead us to some useful results, for example by Fuglede and Ishihara (see Theorem \ref{fuglede}). A standard reference here is the textbook \cite{Bai-Woo-book} by Baird and Wood. For further reading, we recommend the frequently updated online Bibliography of Harmonic Morphisms \cite{Gud-bib}.

We begin with the definition of harmonic functions from $\rn^n$ into $\rn$, which the reader has probably seen before in a course on Complex Analysis. 
\begin{definition}{\rm \cite{Bai-Woo-book}}
Let $U$ be an open subset of $\rn^n$ and consider a $C^2$-regular function
$f:U\rightarrow\mathbb{R}$. Then $f$ is {\it harmonic} if it satisfies the {\it Laplace equation} $$\Delta f=\Div\grad f=0.$$ Explicitly,  the {\it Laplacian} of $f$ is given by $$\Delta f=\frac{\partial^2 f}{\partial x_1^2}+\dots+\frac{\partial^2 f}{\partial x_n^2}.$$
\end{definition}

\begin{example}
The real and the imaginary part of a holomorphic function are both harmonic. 
\end{example}
We now aim to generalise this notion to functions on Riemannian manifolds. Recall the definition of the gradient of a real-valued function from a Riemannian manifold.

\begin{definition}{\rm \cite{O'Neill}}
Let $(M,g)$ be a Riemannian manifold and $\phi:M\rightarrow\mathbb{R}$ a continuously differentiable function on $M.$ Then the {\it gradient} $\nabla \phi$ of $\phi$ is the unique vector field satisfying $$\langle\nabla \phi,X\rangle =\diff\phi(X)=X(\phi)$$ for all $X\in C^{\infty}(TM).$ With respect to local coordinates, $$\nabla \phi=\sum_{i,j}g^{i,j}\frac{\partial \phi}{\partial x_i}\frac{\partial}{\partial x_j}.$$
\end{definition}
We wish to extend the definition of the gradient to complex-valued functions.
\begin{definition}
    Let $(M,g)$ be a Riemannian manifold of dimension $m.$ We extend the metric $g$ to a complex bilinear form on the complexification $T^{\mathbb{C}}M$ of the tangent bundle $TM$ of $M.$ Further, let $\phi:(M,g)\rightarrow\mathbb{C}$ be a complex-valued $C^2$-regular function on $M.$ We may decompose $\phi$ into its real and complex part, i.e. $$\phi=u+iv.$$ Then the gradient of $\phi$ is the smooth vector field given by $$\nabla\phi=\nabla u+i\nabla v.$$
    In particular, if $\mathcal{B}$ is a local orthonormal frame of $TM,$ then the gradient of $\phi$ is the element of the complexified tangent bundle $T^{\mathbb{C}}M$ given by  $$\nabla\phi=\sum_{X\in\mathcal{B}}X(\phi)\cdot X.$$
\end{definition}

\begin{definition}{\rm \cite{Bai-Woo-book}}
Let $X\in C^{\infty}(T^{\mathbb{C}}M)$ be a smooth vector field and $\nabla$ denote the Levi-Civita connection on $(M,g)$. Then the {\it divergence} of $X$ is given by $$\Div X=\trace\nabla X.$$
\end{definition}
We are now ready to define the Laplace-Beltrami operator, which we will frequently encounter throughout the thesis. 
\begin{definition}\label{LB}{\rm \cite{Bai-Woo-book}}
Let $\phi:(M,g)\rightarrow\cn$ be a $C^2$-regular function defined on an open set $U\subset M.$ The {\it Laplace-Beltrami operator} (alternatively: {\it tension field}) $\tau$ on $(M,g)$ is defined by $$\tau(\phi)=\Div\nabla \phi=\trace\nabla\phi.$$ With respect to local coordinates, $\tau$ is given by $$\tau(\phi)=\sum_{i,j=1}^m \frac{1}{\sqrt{|g|}}\frac{\partial}{\partial x_j}\left( g^{ij}\sqrt{|g|}\frac{\partial\phi}{\partial x_i}\right).$$ 
We say that $\phi:(M,g)\rightarrow\rn$ is {\it harmonic} if $\tau(\phi)=0.$
\end{definition} 
This allows us to define the iterated Laplace-Beltrami operator and to introduce p-harmonic functions. 
\begin{definition}{\rm \cite{Gud-Sob-1}}
    Let $(M,g)$ be a Riemannian manifold. We recursively define the {\it iterated Laplace-Beltrami operator} $\tau^{p}$ for a non-negative integer $p$ as follows:
    $$\tau^{0}(\phi)=\phi, \quad \tau^{p}(\phi)=\tau(\tau^{p-1}(\phi)).$$
    A function $\phi:(M,g)\rightarrow\cn$ is said to be \begin{enumerate}
        \item $p$-harmonic if $\tau^p(\phi)=0,$ 
        \item proper $p$-harmonic if $\tau^p(\phi)=0,$ but $\tau^{p-1}(\phi)$ is not identically zero.
    \end{enumerate}
\end{definition}
We now define harmonic morphisms. 
\begin{definition}{\rm \cite{Bai-Woo-book}}\label{harmonicmorphism}
    Let $\phi:M\rightarrow N$ be a smooth mapping between Riemannian manifolds. The mapping $\phi$ is called a {\it harmonic morphism} if for every harmonic function $f:V\rightarrow\mathbb{R}$ defined on an open subset $V$ of $N$ with $\phi^{-1}(V)$ nonempty, the composition $f\circ\phi$ is harmonic on $\phi^{-1}(V).$
\end{definition}
In other words, as explained by Baird and Wood in \cite{Bai-Woo-book}, a harmonic morphism is a smooth map pulling back harmonic functions to harmonic functions. 

We have the following composition rule for harmonic morphisms.
\begin{proposition}{\rm \cite{Bai-Woo-book}}
    Let $\phi:M\rightarrow N$ be a harmonic morphism with image dense in $N$ and let $\psi:N\rightarrow P$ be a smooth map to another Riemannian manifold. Then
    \begin{enumerate}
        \item $\psi$ is a harmonic map if and only if the composition $\psi\circ\phi$ is a harmonic map,
        \item $\psi$ is a harmonic morphism if and only if the composition $\psi\circ\phi$ is a harmonic morphism.
        
    \end{enumerate}
    
\end{proposition}

Next, we will introduce the notion of horizontally weakly conformal maps. As we will see in Theorem \ref{fuglede}, this is closely linked to harmonic morphisms.
\begin{definition}{\rm \cite{Bai-Woo-book}}\label{horizontallyconformal}
Let $\phi:(M^m,g)\rightarrow(N^n,h)$ be a smooth map between Riemannian manifolds, and let $x\in M.$ The vertical and horizontal space of $\phi$ at $x$ are given by $$\mathcal{V}_x=\ker \diff\phi_x,\quad \mathcal{H}_x=\mathcal{V}_x^\bot.$$
Then $\phi$ is said to be {\it horizontally weakly conformal} or {\it semiconformal} at $x$ if either
\begin{enumerate}

    \item $\diff\phi_x=0,$ or 
    \item $\diff\phi_x$ is surjective and
there is a number $\Lambda(x)\neq0$ such that for all $X,Y\in\mathcal{H}_x$ $$h(\diff\phi_x(X),\diff\phi_x(Y))=\Lambda(x)g(X,Y).$$\end{enumerate}
The map $\phi$ is called {\it horizontally weakly conformal} or {\it semiconformal} on $M$ if it is horizontally weakly conformal at every point $x$ of $M.$
\end{definition}

A different characterisation of harmonic morphisms is due to Fuglede and Ishihara (1978 and 1979).
\begin{theorem}{\rm \cite{Bai-Woo-book}}\label{fuglede}
    A smooth map $\phi:M\rightarrow N$ between Riemannian manifolds is a harmonic morphism if and only if $\phi$ is both harmonic and horizontally weakly conformal. 
\end{theorem}

This motivates the definition of the conformality operator, which we will encounter frequently in the rest of this thesis. 

\begin{definition}\label{conformalityop}
    Let $\phi,\psi:(M^m,g)\rightarrow\cn$ be two complex-valued functions on a Riemannian manifold $M.$ The {\it conformality operator} is the complex-valued bilinear form given by $$\kappa(\phi,\psi)=g(\nabla\phi,\nabla\psi).$$ Locally it satisfies $$\kappa(\phi,\psi)=\sum_{i,j=1}^m g^{ij}\frac{\partial \phi}{\partial x_i}\frac{\partial\psi}{\partial x_j}.$$
\end{definition}

\begin{remark}
    A smooth map $\phi:(M,g)\rightarrow\cn$ is horizontally conformal if and only if $$g(\nabla\phi,\nabla\phi)=0.$$ If we decompose $\phi=u+i\cdot v$ into its real and imaginary parts, then $$\kappa(\phi,\phi)=g(\nabla\phi,\nabla\phi)=|\nabla u|^2-|\nabla v|^2+2i\cdot g(\nabla u,\nabla v).$$
    Thus, $\phi$ is horizontally conformal if $\nabla u$ and $\nabla v$ are orthogonal and of the same length at every point $p$ of $M.$
\end{remark}
%\begin{definition}{\rm \cite{Bai-Woo-book}}
%    A Riemannian or {\it isometric immersion} is a conformal map with conformality factor identically equal to $1.$
%\end{definition}

We will now state the well-known product rules for $\tau$ and $\kappa.$
\begin{proposition}{\rm \cite{Gud-Sif-1}}\label{taukappa}
    The Laplace-Beltrami and conformality operators $\tau,\kappa$ satisfy the following product rules.
    \begin{eqnarray*}
\tau(\phi\cdot\psi)&=&\phi\cdot\tau(\psi)+2\cdot\kappa(\phi,\psi)+\psi\cdot\tau(\phi),\\
\kappa(\phi_1\cdot\phi_2,\psi_1\cdot\psi_2)&=&\phi_1\cdot\psi_1\cdot\kappa(\phi_2,\psi_2)+\phi_1\cdot\psi_2\cdot\kappa(\phi_2,\psi_2)\\
&+&\phi_2\cdot\psi_1\cdot\kappa(\phi_1,\psi_2)+\phi_2\cdot\psi_2\cdot\kappa(\phi_1,\psi_1),
    \end{eqnarray*}
    for all smooth complex-valued functions $\phi,\psi,\phi_1,\phi_2,\psi_1,\psi_2$ on a manifold $M.$
\end{proposition}
We will conclude this chapter with a recent result. For the reader who is not familiar with symmetric spaces, we recommend to take a look at Chapter 4. In the nineties, Gudmundsson conjectured the following about the existence of harmonic morphisms from irreducible Riemannian symmetric spaces.  
\begin{conjecture}{\rm \cite{Gud-Sve-1}}\label{Gudconjecture}
Let $(M^m,g)$ be an irreducible Riemannian symmetric space of dimension $m\geq2.$ For each point $p\in M,$ there exists a complex-valued harmonic morphism $\phi:U\rightarrow\cn$ defined on an open neighbourhood $U$ of $p.$ If the space $(M,g)$ is of non-compact type, then the domain $U$ can be chosen to be the whole of $M.$
\end{conjecture}
We will explain in more detail what we mean by duality in Section \ref{duality}. For now, we point out that the irreducible Riemannian symmetric spaces come in pairs of compact and non-compact spaces. Gudmundsson and Svensson have shown in \cite{Gud-Sve-1} that a harmonic morphism on a compact space induces a harmonic morphism on the non-compact dual space, and vice-versa. 

Gudmundsson and Svensson were able to prove the statement of Conjecture \ref{Gudconjecture} in their papers \cite{Gud-Sve-1},\cite{Gud-Sve-2},\cite{Gud-Sve-3} and \cite{Gud-Sve-4} for all such spaces different from $G_2/\SO{4}$ and its non-compact dual.
\begin{theorem}{\rm \cite{Gud-Sve-7}}
Let $(M,g)$ be an irreducible Riemannian symmetric space which is neither $G_2/\SO{4}$ nor its non-compact dual. Then $M$ carries local complex-valued harmonic morphisms. If $M$ is of non-compact type, then global solutions exist.
\end{theorem}

In 2024, Burstall confirmed Conjecture \ref{Gudconjecture}.
\begin{theorem}{\rm \cite{Bur}}
    Let $M$ be a Riemannian symmetric space of non-compact type. Then there is a non-constant, globally defined harmonic morphism from $M$ to $\cn.$
\end{theorem}
By the duality principle, we note that there must also exist a harmonic morphism defined locally on the compact $G_2/\SO{4}.$

\chapter{Lie Groups and Lie Algebras}
This chapter summarises some important facts about Lie groups and Lie algebras. We provide several examples and will state some basic properties. We will then talk about representations of Lie groups and Lie algebras, including the adjoint representation. In the last section of this chapter, we introduce the classical matrix Lie groups, as well as their Lie algebras.

The chapter is mainly based on the very accessible yet highly informative book \cite{Arv} by Arvanitoyeorgos. For further reading, we recommend the books \cite{Hall} by Hall, the classic textbook \cite{Helg} by Helgason and \cite{Kna} by Knapp.

\section{Lie Groups}
In this section we take a first look at Lie groups and Lie group homomorphisms. We will give some concrete examples as well.
\begin{definition}\label{defliegroup}
A smooth manifold $G$ is called a {\it Lie group} if it is also a group $(G,\cdot),$ such that the group operation $$\cdot:G\times G\rightarrow G, \ (x,y)\mapsto x\cdot y$$ and taking inverses $$G\rightarrow G, \ x\mapsto x^{-1}$$ are smooth functions.
\end{definition}

\begin{remark}
    Alternatively, we may combine the two conditions given in Definition \ref{defliegroup} into one condition. In particular, $G$ is a Lie group if it is a smooth manifold, a group and the map $$G\times G\rightarrow G, \quad (x,y)\mapsto xy^{-1}$$ is smooth. This is used for example in Helgason's book \cite{Helg}.
\end{remark}
We state some well-known examples. They appear in many introductory texts on Lie groups, for example \cite{Arv}.
\begin{example}\label{vectorspaceliegroup}
    For $n\in\mathbb{N},$ the sets $\mathbb{R}^n,\mathbb{C}^n,\mathbb{H}^n$ with the standard vector addition as group operation are Lie groups.
\end{example}

%\begin{definition}
%Let $\H$ denote the set of quaternions $$\H=\{(z,w)\in\cn^2\}.$$ We define addition, multiplication and conjugation as follows:\begin{eqnarray*}
%    (z_1,w_1)+(z_2,w_2)&=&(z_1+z_2,w_1+w_2),\\
 %   (z_1,w_1)\cdot(z_2,w_2)&=&(z_1z_2-\Bar{w}_1w_2,w_1z_2+\Bar{z}_1w_2),\\
%    \overline{(z,w)}&=&(\Bar{z},-w).
%\end{eqnarray*}\end{definition}

\begin{example}
    Let $G$ be a Lie group with the group operation $\cdot:G\times G\rightarrow G.$ Then if $H$ is both a submanifold and a subgroup of $G,$ $(H,\cdot)$ is also a Lie group.
\end{example}

\begin{example}\label{rstarliegroup}
    The sets $\mathbb{R}^*,\mathbb{C}^*,\mathbb{H}^*$ are Lie groups with the usual multiplication. Thus, the unit circle $$S^1=\{z\ \rvert \ z\in\cn,\ |z|=1\}\subset\mathbb{C}^*$$ is also a Lie group.
    Similarly, $$S^3=\{z+jw \ \rvert \ z,w\in\cn, \ |z|^2+|w|^2=1\}\subset\H^* $$ is a Lie group.
\end{example}
Next, we turn our attention to some important maps on Lie groups.
\begin{definition}{\rm \cite{Arv}}
    Let G be a Lie group and $a\in G.$ The the {\it left} and {\it right translations} are the maps \begin{eqnarray*}
        L_a:G\rightarrow G, &&L_a(g)=ag\\
        R_a:G\rightarrow G, &&R_a(g)=ga
    \end{eqnarray*}
    respectively. 
\end{definition}

\begin{definition}{\rm \cite{Lee}}
    Let $G$ and $H$ be Lie groups. A map $\phi:G\rightarrow H$ is a {\it Lie group homomorphism} if $\phi$ is both a group homomorphism and a diffeomorphism. 
\end{definition}

\begin{example}{\rm \cite{Lee}}
    We have seen in Example \ref{vectorspaceliegroup} that $\rn$ is a Lie group under addition, and in Example \ref{rstarliegroup} that $\rn^*$ is a Lie group under multiplication. Then the map $$\exp:(\rn,+)\rightarrow(\rn^*,\cdot),\quad x\mapsto \exp(x)$$ is smooth with inverse $\log:\rn^+\rightarrow \rn.$ Further, the exponential map is a Lie group homomorphism since $$\exp(x+y)=\exp(x)\cdot\exp(y).$$
\end{example}

\begin{proposition}\label{innerauto}
    Let $G$ be a Lie group and $x\in G.$ Then the map $$I_x:G\rightarrow G, \ g\mapsto xgx^{-1}$$ is an automorphism. We will call it the {\it inner automorphism}.
\end{proposition}

\begin{proof}
    The map $I_x$ is a homomorphism since $$I_x(gh)=xghx^{-1}=xgx^{-1}xhx^{-1}=I_x(g)I_x(h).$$ Further, it is clear that $$I_x=R_{x^{-1}}\circ L_{x}.$$ For any $x\in G,$ the maps $L_x,R_x$ are smooth by definition of a Lie group. Since there exists an inverse $x^{-1}$ of $x$ in $G,$ the inverse maps $L_{x^{-1}}, R_{x^{-1}}$ are smooth too. Consequently, the left and right translations are diffeomorphisms, and so is the composition $I_x$. 
\end{proof}

\section{Lie Algebras}
In this section we review Lie algebras. Most importantly, we recall that the left-invariant vector fields of a Lie group together with the so-called Lie bracket form a Lie algebra.

\begin{definition}
    Let the field $F$ be either $\rn$ or $\cn.$ A {\it Lie algebra} is a vector space $\g$ over $F,$ equipped with a bracket operation $$[\cdot,\cdot]:\g\times\g\rightarrow\g, \quad (X,Y)\mapsto[X,Y],$$ satisfying the following properties for all $X,Y,Z\in\g,$ and all $a,b\in F:$
    \begin{enumerate}
        \item bilinearity: 
            $$[aX+bY,Z]=a[X,Z]+b[Y,Z],$$
          $$[Z,aX+bY]=a[Z,X]+b[Z,Y.]$$
        \item antisymmetry: $$[X,Y]=-[Y,X].$$
        \item the Jacobi identity:
        $$[X,[Y,Z]]+[Y,[Z,X]]+[Z,[X,Y]]=0.$$
    \end{enumerate}
\end{definition}

\begin{example}{\rm \cite{Arv}}
    The vector space $\rn^3$ together with the standard cross product $\times$ is a Lie algebra. 
\end{example}

We will now define the Lie bracket for vector fields on a manifold. As we will see, this yields a Lie algebra.
\begin{definition}{\rm \cite{Gud-Rie}}
    Let $M$ be a differentiable manifold. Then the {\it Lie bracket } $$[\cdot,\cdot]:C^{\infty}(TM)\times C^{\infty}(TM)\mapsto C^{\infty}(TM),$$ is given by $$[X,Y]_p(f)=X_p(Y(f))-Y_p(X(f)),$$ for $f\in C^{\infty}(M),p\in M.$ 
\end{definition}
\begin{theorem}{\rm \cite{Gud-Rie}}
Let $M$ be a smooth manifold. Then the vector space $C^{\infty}(TM)$ of smooth vector fields on $M$ together with the Lie bracket $[\cdot,\cdot]$ is a Lie algebra.
\end{theorem}
We now introduce a special subset of vector fields on a Lie group.

\begin{definition}{\rm \cite{Arv}}
    A vector field $X$ on a Lie group $G$ is {\it left invariant} if for all $a\in G,$
    $$X\circ L_a=\diff L_a(X).$$
    More explicitly, for all $a,g\in G,$
    $$X_{ag}=(\diff L_a)_g(X_g).$$
\end{definition}

\begin{remark}
    Left-invariant vector fields are entirely determined by their value at the origin $e$, which is evident from the following identity. $$X_a=(\diff L_a)_e(X_e)$$
\end{remark}

%\begin{definition}{\rm \cite{Gud-Rie}}
%    If $\phi:M\rightarrow N$ is a surjective map between differentiable manifolds, then two vector fields $X\in C^{\infty}(TM)$ and $\Bar{X}\in C^{\infty}(TN)$ are said to be $\phi$-{\it related} if $$\diff\phi_p(X_p)=\Bar{X}_{\phi(p)}.$$ In this case, we write $\diff\phi(X)=\Bar{X}.$
%\end{definition}

%\begin{remark}
%    Let $G$ be a Lie group and consider the left translation $L_g:G\rightarrow G$ for some $g\in G.$ Then a left invariant vector field $X\in\g$ is $L_g$-related to itself. 
%\end{remark}

\begin{definition}{\rm \cite{Helg}}
    Let $\b$ be a vector subspace of a Lie algebra $\a.$ Then $\b$ is a {\it subalgebra} of $\a$ if  $$[\b,\b]\subset\b$$ and an {\it ideal} if $$[\b,\a]\subset\b.$$ If $\b$ is an ideal, then the factor space $\a/\b$ is a Lie algebra with the bracket operation inherited from $\a.$
\end{definition}
It can be shown that the left invariant vector fields of a Lie group $G$ form a subalgebra of $C^\infty(TG).$
\begin{theorem}{\rm \cite{Gud-Rie}}
    Let $G$ be a Lie group. Then the set of left invariant vector fields $\g$ together with the Lie bracket forms a Lie algebra. 
\end{theorem}
The following shows that the Lie algebra $\g$ of a Lie group $G$ is isomorphic to the tangent space of $G$ at the neutral element. As a consequence, $\g$ and $G$ have the same dimension. 
\begin{proposition}{\rm \cite{Arv}}\label{liealgisototangentate}
The function $X\mapsto X_e$ defines a linear isomorphism between the vector spaces $\g$ and $T_eG.$\end{proposition}
\begin{proof}
    The function clearly is linear. It is injective since $X_e=0$ implies $X_g=\diff L_g(X_e)=0$ for all $g\in G.$ Finally, the function is surjective. For any $v\in T_eG$ define the vector field $X^{v}$ by $X^{v}_g=(\diff L_g)_e(v)$ for all $g\in G.$ Then $X^{v}$ is left-invariant and by construction and $X_e^{v}=v.$
\end{proof}
    
\begin{definition}{\rm \cite{Helg}}
    Let $\a,\b$ be two Lie algebras over the same field $F$ of characteristic zero, and $\sigma:\a\rightarrow\b$ a linear mapping. Then $\sigma$ is called a {\it homomorphism} if $$\sigma([X,Y])=[\sigma X,\sigma Y]$$ for all $X,Y\in\a.$ If $\ker\sigma=\{0\},$ then $\sigma$ is called an {\it isomorphism}.
\end{definition}

\begin{theorem}{\rm \cite{Helg}}
    Let $G$ be a Lie group. If $K$ is a Lie subgroup of $G,$ then the Lie algebra $\k$ of $K$ is a Lie subalgebra of the Lie algebra $\g$ of $G.$ Each subalgebra of $\g$ is the Lie algebra of exactly one connected Lie subgroup of $G.$
\end{theorem}

\section{Matrix Lie Groups}
We now introduce matrix Lie groups, which we will reencounter in the later chapters. This sections follows the approach Hall takes in his book \cite{Hall}. We first consider matrix Lie groups as closed subgroups of the complex general linear group $\GLC n$, as stated in Definition \ref{matrixliegroup}. We then show that any such matrix Lie group is a Lie group in the sense of Definition \ref{defliegroup}. This can also be shown in concrete cases using the Implicit Function Theorem.
In the remainder of the section, we give several examples and discuss the corresponding Lie algebras. 

\begin{definition}
    The {\it complex general linear group} $$\GLC n=\{x\in\cn^n \ \rvert \ \det z\neq0\}$$ is the group of invertible $n\times n$-matrices with complex entries. The group operation is the usual matrix multiplication. 
\end{definition}
For the remainder of the thesis, we will equip $\GLC n$ with the following inner product.

\begin{definition}\label{standardmetric}
    The standard metric on $\GLC{n}$ is given by $$g(Z,W)=\Re\trace(\Bar{Z}^t\cdot W).$$
\end{definition}
A matrix Lie group is a closed subgroup $G$ of $\GLC n.$ Formally, we define it as follows.

\begin{definition}{\rm \cite{Hall}}\label{matrixliegroup}
    A {\it matrix Lie group} is a subgroup $G$ of $\GLC n$ with the following property: If $\{z_k\}$ is any sequence of matrices in $G,$ and $\{z_k\}$ converges (entrywise) to some matrix $z,$ then either $z\in G$ or $z$ is not invertible.
\end{definition}
Theorem \ref{matrixliegroupeqliegroup} shows that a matrix Lie group $G$ satisfies the conditions of Definition \ref{defliegroup}. We first need to introduce the corresponding Lie algebras of matrix Lie groups.

\begin{definition}
    Let $X\in M_n(\cn).$ Then we define the exponential function as follows: $$\exp(X)=\sum_{k=0}^\infty \frac{X^k}{k!}.$$ Here, define $X^0=I_n.$
\end{definition}

\begin{definition}{\rm \cite{Hall}}\label{liealg}
Let $G$ be a matrix Lie group. The Lie algebra $\g$ of $G$ is the set of all matrices $X$ such that $\exp({sX})\in G$ for all $s\in\rn.$ The bracket operation is given by $$[X,Y]=XY-YX$$ for $X,Y\in\g.$
\end{definition}
It is shown in sections 3.7 and 3.8 of \cite{Hall} that subgroups of $\GLC n$ that are closed in the sense of Definition \ref{matrixliegroup} are in fact Lie groups in the sense of Definition \ref{defliegroup}. We state the following Theorem.

\begin{theorem}{\rm \cite{Hall}}\label{matrixliegroupeqliegroup}
    Let $G\subseteq\GLC{n}$ be a matrix Lie group with Lie algebra $\g$ and let $m$ be the dimension of $\g$ as a real vector space. Then $G$ is a smooth embedded submanifold of $M_n(\cn)$ of dimension $m$ and hence a Lie group.
\end{theorem}
We resume by stating some important examples.
\begin{definition}
    The {\it real general linear group} $$\GLR n=\{x\in\mathbb{R}^n | \det x\neq0\} $$ is the group of invertible $n\times n$-matrices with real entries. The group operation is the usual matrix multiplication. 
    The {\it orthogonal} and {\it real special linear group} are the following subgroups of $\GLR n.$
\begin{eqnarray*}
   \SLR{n}&=&\{x\in\R^{n\times n} \ \rvert \ \det x=1\},\\
   \O n&=&\{x\in\R^{n\times n}\ \rvert \ x\cdot x^t=I_n\}. 
\end{eqnarray*}
The {\it special orthogonal group} $\SO{n}$ is the intersection of the orthogonal and the real special linear groups.
\begin{eqnarray*}
    \SO{n}&=& \O{n}\cap\SLR n\\
    &=&\{x\in\R^{n\times n} \, | \, x\cdot x^t=I_n, \ \det x=1\}. 
\end{eqnarray*}
\end{definition}

\begin{proposition}\label{propmatrixliegroups}{\rm \cite{Hall}}
    The subgroups $\GLR n, \O n, \SLR n, \SO{n}$ of $\GLC n$ are matrix Lie groups.
\end{proposition}
\begin{proof}
    Let $\{x_k\}$ be a sequence of matrices in $\GLR n$ converging to a matrix $x.$ Clearly the entries of $x$ must be real. Now either $x$ is invertible, in which case $x\in\GLR n,$ or $x$ is not invertible. We conclude that $\GLR n$ is a matrix Lie group.

    If $\{x_k\}$ is a sequence of matrices in $\O n,$ then each matrix $x_k$ satisfies $x_kx_k^t=I_n.$ Since this relation is preserved when taking the limit, we see that $\O n$ is a matrix Lie group.

    Finally, $\SLR n $ is a matrix Lie group since the function $$x\mapsto\det x$$ is continuous. Indeed, if the sequence $\{x_k\}$ of matrices in $\SLR n$ converges to some matrix $x$, then $\det x=1$ and thus $x\in\SLR n$.
\end{proof}
We define groups of matrices with complex-valued entries. 
\begin{definition}
   % The {\it complex general linear group} $$\GLC n=\{x\in\cn^n \ \rvert \ \det z\neq0\}$$ is the group of invertible $n\times n$-matrices with complex entries. The group operation is the usual matrix multiplication. 
    The {\it unitary} and {\it complex special linear group} are the subgroups of $\GLC n$ given by
\begin{eqnarray*}
    \U n&=&\{z\in\cn^{n\times n}\ \rvert \ z\cdot z^*=I_n\},  \\
    \SLC{n}&=&\{z\in\cn^{n\times n} \ \rvert \ \det z=1\}.
\end{eqnarray*}
The {\it special unitary group} is the intersection of the unitary and the complex special linear group.
\begin{eqnarray*}
    \SU{n}&=& \U{n}\cap\SLC n\\
    &=&\{z\in\cn^{n\times n} \ \rvert \ z\cdot z^*=I_n, \ \det z=1\}. 
\end{eqnarray*}
\end{definition}

\begin{proposition}\label{unmatrixlie}
    The subgroups $\U n, \SLC n, \SU{n}$ of $\GLC n$ are matrix Lie groups. 
\end{proposition}
\begin{proof}
  The proof is similar to the real case, see Proposition \ref{propmatrixliegroups}.  
\end{proof}

%As we have suggested previously, it can also be shown explicitly that the matrix Lie groups are Lie groups. We will demonstrate this is the following example
%\begin{example}
 %   It is clear that $\U{n}$ is a group with the operation of standard matrix multiplication. It is left to show that $\U{n}$ is a smooth manifold. 
    
%    The map $$\det: \cn^{n\times n}\rightarrow\cn,\quad z\mapsto \det z,$$ mapping each matrix to its determinant is a continuous map. The one-point set $\{0\}$ is closed in $\cn,$ thus the preimage of $\{0\}$ under $\det$ is closed in in $\cn^{n\times n}.$ Hence its complement $\GLC{n}$ must be open. Since $\GLC{n}$ is an open subset of an Euclidean space, it must be a smooth manifold. 
%\end{example}
Some of the matrix Lie groups are compact, as the next proposition shows.
\begin{proposition}{\rm \cite{Hall}}
The Lie groups $\O{n},$ $\SO{n},$$\U{n},$ and $\SU{n}$ are compact.
\end{proposition}
\begin{proof}
    We will prove the statement in the complex setting. The corresponding proof in the real case is similar.

To see that $\U{n}$ is closed, we consider the continous map $$\phi: M_n(\cn)\rightarrow M_n(\cn), \quad z\mapsto z\Bar{z}^t$$ and simply note that $\U{n}$ is the inverse image of the unit matrix $I_n.$

    The Lie group $\U{n}$ is bounded, since for every $z\in\U{n}$ $$|z|^2=\langle z,z\rangle_F=\trace(\Bar{z}^t\cdot z)=\trace(I_n)=n.$$ Here, we denote the standard Frobenius inner product by $\langle\cdot,\cdot\rangle_F.$ Alternatively, it is also clear that if $z\in\U{n},$ then for each entry $|z_{ij}|\leq1,$ since the columns of $z$ are unit vectors in $\cn^n$. 

By the Heine-Borel Theorem, compactness follows. 
\end{proof}
We now dive a description of the Lie algebras of the previously discussed matrix Lie groups.

\begin{proposition}{\rm \cite{Hall}}
    The Lie algebra $\glc{n}$ of $\GLC{n}$ is given by $M_n(\cn),$ which is the space of all $n\times n$ matrices with complex entries. The Lie algebra $\glr{n}$ of $\GLR{n}$ is given by $M_n(\rn),$ the space of all $n\times n$ matrices with real entries. 
\end{proposition}
\begin{proof}
Definition \ref{liealg} states that the Lie algebra $\g$ of a matrix Lie group $G$ consists of all matrices $X$ such that $e^{sX}\in G$ for all $s\in\rn.$ 
    Let $X\in M_n(\cn).$ Then $e^{sX}$ is invertible, thus $X\in\glc{n}.$ If $X$ has only real entries, then clearly $e^{sX}$ is also real and $X\in\glr{n}.$ On the other hand if $e^{sX}$ is real, then $X$ must also be real.
\end{proof}  

\begin{proposition}{\rm \cite{Gud-Rie}}\label{comprulesexp}
    For $X\in M_n(\cn),$ the exponential map satisfies the following.
    \begin{enumerate}
        \item $e^{X^t}=(e^X)^t,$
        \item $e^{\Bar{X}}=\overline{e^X},$
        \item $\det(e^X)=e^{\trace(X)}.$
    \end{enumerate}
\end{proposition}

\begin{proposition}{\rm \cite{Hall}}
    The Lie algebra $\slr{n}$ of the special linear group $\SLR{n}$ is given by $$\slr n =\{ X\in\glr{n} \ \rvert \ \trace(X)=0\}.$$
    The Lie algebra $\slc{n}$ of the special linear group $\SLC{n}$ is given by $$\slc n =\{ X\in\glc{n} \ \rvert \ \trace(X)=0\}.$$
\end{proposition}
\begin{proof}
We prove the claim for the complex special linear group, as the proof is the real setting is done analogously. 
    Assume $X\in\slc{n}.$ Then $e^{sX}\in\SLC{n}$ for all $s\in\rn.$ By Proposition \ref{comprulesexp}, $$\det(e^{sX})=e^{s\trace(X)}=1.$$ It follows that $$\trace(X)=\frac{d}{ds}(e^{s\trace(X)})\rvert_{s=0}=\frac{d}{ds}(1)\rvert_{s=0}=0.$$ 
    
    Conversely, if $X\in\glc{n}$ satisfies $\trace X=0,$ then clearly for all $s\in\rn$ it holds that $$\det(e^{sX})=e^{s\trace(X)}=e^0=1$$ and thus $e^{sX}\in\slc{n}.$ 
\end{proof}

\begin{proposition}\label{liealgebraonun}{\rm \cite{Hall}}
    The Lie algebra $\o{(n)}$ of the orthogonal group $\O{n}$ is given by $$\o{(n)}=\{X\in\glr{n} \ \rvert \ X^t+X=0\}.$$
    The Lie algebra $\u{n}$ of the unitary group $\U{n}$ is given by $$\u{n}=\{X\in\glc{n} \ \rvert \ \Bar{X}^t+X=0\}.$$
\end{proposition}
\begin{proof}
We will prove the statement for the orthogonal group. Assume $X\in \o{(n)}.$ Then for all $s\in\rn,$ we have that $e^{sX}\in\O{n}$ and thus $$(e^{sX})^t\cdot e^{sX}=I_n.$$ Differentiating on both sides yields $$0=\frac{d}{ds}((e^{sX})^t\cdot e^{sX})\rvert_{s=0}=X^t+X.$$ Here we again used Proposition \ref{comprulesexp}.

    Conversely, if $X\in\glr{n}$ satisfies $X^t+X=0,$ then $$(e^{sX})^t\cdot e^{sX}=e^{sX^t}\cdot e^{sX}=e^{s(X^t+X)}=e^0=I_n$$ for all $s\in\rn.$ It follows that $X\in\o{(n)}.$ The proof for $\u{n}$ is done analogously. 
\end{proof}
The last two propositions we draw the following conclusion.
\begin{corollary}{\rm \cite{Hall}}
    The Lie algebra $\su{n}$ of $\SU{n}$ consists of all skew-herminitian $n\times n$ matrices, and the Lie algebra $\so{n}$ of $\SO{n}$ is the same as that of $\O{n}.$
    \end{corollary}
We now wish to find bases or generators for the Lie algebras given above. For this, we define our “building blocks" of matrices. This is the standard notation used in the research papers cited in this thesis, for example \cite{Gud-Sak-1} and \cite{Gud-Mun-1}.
\begin{definition}
    For $1\leq r,s\leq n,$ let $E_{rs}$ be the $n\times n$ matrix with entries $$(E_{rs})_{i,j}=\delta_{ir}\delta_{js}.$$ For $1\leq r<s\leq n,$ we define $$X_{rs}=\frac{1}{\sqrt{2}}\cdot(E_{rs}+E_{sr}),\quad Y_{rs}=\frac{1}{\sqrt{2}}\cdot(E_{rs}-E_{sr}),\quad D_{rs}=\frac{1}{\sqrt{2}}\cdot(E_{rr}-E_{ss}).$$
\end{definition}

\begin{remark}{\rm \cite{Gud-Rie}}
    Recall that a metric $g$ on a Lie group $G$ is said to be left-invariant if for all $p\in G$ the left-translation $L_p$ is an isometry. 
    In this case, $g$ is entirely determined by the scalar product $g_e:T_eG\times T_eG\rightarrow\rn$ at the origin $e.$ Indeed,
    if $X,Y\in\g$ are left-invariant vector fields on $G,$
    $$g_p(X_p,Y_p)=g_p((\diff L_p)_e(X_e),(\diff L_e)_p(Y_e))=g_e(X_e,Y_e).$$
\end{remark}
\begin{remark}
    We have seen in Proposition \ref{liealgisototangentate} that the Lie algebra $\g$ of a Lie group $G$ is isomorphic to $T_eG.$ Thus, it suffices to compute the bracket at the neutral element by standard matrix multiplication. As shown in \cite{Gud-Rie} the corresponding bracket operation $$[\cdot,\cdot]:T_eG\times T_eG\rightarrow T_eG$$ at the neutral element $e$
    is given by $$[A_e,B_e]=A_e\cdot B_e-B_e\cdot A_e.$$
\end{remark}

\begin{remark}
    The metric $g$ as given in Definition \ref{standardmetric} defines a left-invariant metric on the groups $\O{n},$ $ \SO{n},$ $ \U{n},$ $ \SU{n},$ turning them into {\it Riemannian Lie groups}.
\end{remark}

\begin{proposition}{\rm \cite{Gud-Rie}}\label{basisson}
    The Lie algebra $\so n$ of $\SO n $ is given by the  real skew-symmetric $n\times n$ matrices, 
$$\so n=\{ X\in\R^{n\times n} \ \rvert \ X+X^t=0\}.$$
The set $$\{Y_{rs} \ \rvert \ 1\leq r<s\leq n \}$$ is an orthonormal basis of the Lie algebra $\so n.$
\end{proposition}

\begin{proposition}{\rm \cite{Gud-Rie}}\label{basissun}
       The Lie algebra $\su n$ of $\SU{n}$ is given by $$\su{n}=\so{n}\oplus i\m,$$
where $$\m=\{X\in\rn^{n\times n} \ \rvert \ X=X^t \  \textrm{and} \ \trace(X)=0\}.$$
Further, $\m$ is generated by $$\mathcal{S}_\m=\{D_{rs},X_{rs}\  | \ 1\leq r<s\leq n\}.$$
\end{proposition}
 Here, the second statement is a direct consequence of the fact that all elements in $\m$ are symmetric and traceless.

Lastly, we will introduce the quaternionic unitary group, as well as its Lie algebra.
\begin{definition}
Let $\GLH n$ be the group of invertible $n\times n$ quaternionic matrices. Then the quaternionic unitary group $\Sp{n}$ is given by
$$\Sp{n}=\{q\in\GLH n \ \rvert \ q^*=q^{-1}\}.$$
The standard complex representation of $\Sp{n}$ on $\cn^{2n}$ is given by 
$$q=z+jw\mapsto\begin{pmatrix}
    z&w\\
    -\Bar{w}&\Bar{z}
\end{pmatrix}.$$
\end{definition}

\begin{proposition}{\rm \cite{Gud-Sak-1}}\label{liealgebraspn}
      The Lie algebra $\sp{n}$ of $\Sp{n}$ satisfies $$\sp{n}=\left\{ \begin{pmatrix}
        Z&W\\
        -\Bar{W}&\Bar{Z}
    \end{pmatrix}\in\cn^{2n\times2n}\ \rvert \ Z^*+Z=0\ \textrm{and} \ W^t-W=0\right\} .$$
An orthonormal basis is given by \begin{eqnarray*}
    \mathcal{B}_{\sp{n}}&=&\left\{Y_{rs}^{a}=\frac{1}{\sqrt{2}}\begin{pmatrix}
        Y_{rs}&0\\
        0&Y_{rs}
    \end{pmatrix}, \ X_{rs}^{a}=\frac{1}{\sqrt{2}}\begin{pmatrix}
        iX_{rs}&0\\
        0&-iX_{rs}
    \end{pmatrix},\right. \\
    && X_{rs}^{b}=\frac{1}{\sqrt{2}}\begin{pmatrix}
        0&iX_{rs}\\
        iX_{rs}&0
    \end{pmatrix},\ X_{rs}^{c}=\frac{1}{\sqrt{2}}\begin{pmatrix}
        0&X_{rs}\\
        -X_{rs}&0
    \end{pmatrix},\\
    &&D_{t}^{a}=\frac{1}{\sqrt{2}}\begin{pmatrix}
        iD_{t}&0\\
        0&-iD_t
    \end{pmatrix}, \ D_{t}^{b}=\frac{1}{\sqrt{2}}\begin{pmatrix}
        0&iD_{t}\\
        iD_t&0
    \end{pmatrix},\\
    &&\left. D_{t}^{c}=\frac{1}{\sqrt{2}}\begin{pmatrix}
        0&D_{t}\\
        -D_t&0
    \end{pmatrix} \ \rvert \ 1\leq r<s\leq n, \ 1\leq t\leq n\right\}.
\end{eqnarray*}

\end{proposition}

\section{Representations}
Many of the Lie groups we have seen so far are subgroups of the real and complex general linear groups $\GLR n$ and $\GLC n$. A natural question, as suggested by Lee in \cite{Lee}, would be if all Lie groups are of this form. This motivates the study of representations, which are Lie group homomorphisms from Lie groups into the real or complex general linear group.

In this section we will define representations of Lie groups and Lie algebras. We will also define the so-called adjoint representation.
\begin{definition}{\rm \cite{Hall}, \cite{Lee}}\label{representation}
    Let the field $F$ the $\rn$ or $\cn$. For a finite-dimensional real or complex vector space $V$ over $F$ we denote $\GL(V)$ the group of invertible linear transformations of $V.$ Let $G$ be a Lie group. A {\it finite-dimensional representation } of $G$ on $V$ is a Lie group homomorphism $$\rho:G\rightarrow\GL(V),\quad g\mapsto\rho_g$$ for some $V.$
\end{definition}

\begin{example}{\rm \cite{Gud-Rie}}
    Consider the representation $$\rho:\cn^*\rightarrow\Aut(\rn^2)$$ given by $$z=(a+ib)\mapsto\rho_z=\begin{pmatrix}
        a&-b\\
        b&a
    \end{pmatrix}.$$
    Due to basic properties of matrix multiplication, this map $\rho_z$ is linear for each complex-valued $z.$ A simple computation shows that $\rho$ is a group homomorphism:
    \begin{eqnarray*}
        \rho((a+ib)\cdot(c+id))&=&\rho(ac-bd+i(ad+bc))\\
        &=&\begin{pmatrix}
            ac-bd&-(ad+bc)\\
            ad+bc&ac-bd
        \end{pmatrix}\\
        &=&\begin{pmatrix}
            a&-b\\
        b&a
        \end{pmatrix}\cdot \begin{pmatrix}
            c&-d\\
        d&c
        \end{pmatrix}\\
        &=&\rho_{a+ib}\circ\rho_{c+id}\\
        &=&\rho(a+ib)\circ\rho(c+id).
    \end{eqnarray*}
    Since $z$ is assumed to be non-zero, it follows that $\rho_z$ is bijective.
\end{example}

\begin{definition}{\rm \cite{Lee}}
    A representation is said to be {\it faithful} if it is injective.
\end{definition}
It follows that a Lie group $G$ is a subgroup of $\GLR n$ or $\GLC n$ if there exists a faithful representation of $G.$

\begin{definition}{\rm \cite{Hall}}
    Let $\rho$ be a finite dimensional real or complex representation of a Lie group $G,$ acting on a vector space $V.$ A subspace $W$ of $V$ is called {\it invariant} if $$\rho_g(w)\in W$$ for all $g\in G, w\in W.$ An invariant subspace is called {\it nontrivial} if $W\neq\{0\},$ and $ W\neq V.$ A representation with no nontrivial invariant subspaces is called {\it irreducible}.
\end{definition}
Recall the inner automorphism, which we defined in Proposition \ref{innerauto}. We will now define the adjoint representation of a Lie group. It maps elements of a Lie group to the automorphisms of its Lie algebra.
\begin{definition}{\rm \cite{Arv}}
Let $I_g:G\rightarrow G$ be the inner automorphism. The {\it adjoint representation} of a Lie group $G$ is the homomorphism $$\Ad:G\rightarrow\Aut(\g)$$ given by $$\Ad(g)=(\diff I_g)_e.$$
\end{definition}

\begin{proposition}{\rm \cite{Arv}}\label{adjointrepmatrix}
    If $G$ is a matrix Lie group, then for all $g\in G$ and all $X\in\g,$ $$\Ad(g)(X)=gXg^{-1}.$$
\end{proposition}
\begin{proof}
    This fact follows from an easy computation. \begin{eqnarray*}
        \Ad(g)(X)&=&\frac{d}{ds}(I_g(\exp{(sX)}))\rvert_{s=0}\\
        &=&\frac{d}{ds}(g\cdot\exp{(sX)}\cdot g^{-1})\rvert_{s=0}\\
        &=&gXg^{-1}.
    \end{eqnarray*}
\end{proof}
Similarly to representations of Lie groups, we will now define representations of Lie algebras.
\begin{definition}
    Let $F$ be a field an $V$ a finite dimensional vector space over $F$. By $\End(V)$ we denote the $F$-linear endomorphisms of $V.$
\end{definition}

\begin{definition}{\rm \cite{Hall}}
    Let $\g$ be a real or complex Lie algebra and $V$ be a vector space over $\cn$. A finite dimensional complex {\it representation} of $\g$ is a Lie algebra homomorphism $$\rho:\g\rightarrow \End(V), \ g\mapsto\rho_g.$$ 
\end{definition}

\begin{definition}{\rm \cite{Lee}}
    For a Lie algebra $\g$, the {\it adjoint representation }$$\ad:\g\rightarrow\End(\g)$$ is given by $$X\mapsto\ad_X=(\diff\Ad)_e(X).$$
\end{definition}
For a matrix Lie group $G$ with Lie algebra $\g$, the adjoint representations of $G$ and $\g$ respectively satisfy the following property. 
\begin{proposition}{\rm \cite{Hall}}
Let $G$ be a matrix Lie group and $\g$ be its Lie algebra. Let $\Ad:G\rightarrow\GL(\g)$
 be the adjoint representation of $G,$ and $\ad:\g\rightarrow\End(\g)$ the adjoint representation of $\g.$ Then for all $X,Y\in\g$ the adjoint representation of $\g$ satisfies $$\ad_X(Y)=[X,Y].$$
\end{proposition}
\begin{proof}
    In Proposition \ref{adjointrepmatrix} we have seen that for a matrix Lie group $G,$ it holds that $$\Ad(g)(X)=gXg^{-1}$$ for all $g\in G$ and $X\in\g.$ Thus,
    \begin{eqnarray*}
        \ad_X(Y)&=&\frac{d}{ds}(\Ad_{\exp(sX)}(Y))\rvert_{s=0}\\
        &=&\frac{d}{ds}(\exp(sX)\cdot Y\cdot\exp(-sX))\rvert_{s=0}\\
        &=&XY-YX\\
        &=&[X,Y].
    \end{eqnarray*}
\end{proof}

%%%%%%%%%%%%%%%%%%%%%%%%%%%%%%%%%%%%%%%%
\chapter{Riemannian Symmetric Spaces}
%\jg{intro: introduced by Cartan in 1925, Lie theory and classification, Berger and holonomy, physics, representation theory
%references: \cite{Arv}, O'neill}

The main goal of this chapter is to introduce Riemannian symmetric spaces. In particular, we will study how Lie groups act on smooth manifolds. In preparation, we clarify what we mean by a group action on an abstract set in Section \ref{groupactions}. In Section \ref{homogeneousspaces}, we will define Riemannian homogeneous spaces, and will then move on to symmetric spaces in Section \ref{symmetricspaces}. In Section \ref{duality}, we will briefly talk about the notion of duality. We cover similar contents as the book \cite{Arv} by Arvanitoyeorgos. For a more detailed discussion of quotient manifolds we highly recommend Chapter 21 of \cite{Lee} by Lee. For many of the theorems and proofs we also used the standard textbooks \cite{O'Neill} by O'Neill and \cite{Helg} by Helgason.

\section{Group Actions}\label{groupactions}
In order to define homogeneous spaces, we first need to introduce group actions. We will give a handful of examples and will study some important properties.

\begin{definition}\label{defgroupaction}{\rm \cite{Lee}}
    Let $G$ be a group and $X$ a set. A (left) {\it group action} of $G$ on $X$ is a map $$\alpha: G\times X\rightarrow X,  \quad (g,x)\mapsto g\cdot x,$$
    satisfying the following conditions:
    \begin{enumerate}
        \item unitality: for all $x\in X,$ it holds that $e\cdot x=x,$ 
        \item associativity: for all $g,h\in G, $ and $x\in X$ it holds that $ g\cdot(h\cdot x)=(gh)\cdot x. $
    \end{enumerate}
\end{definition}
If it is clear which action of a group $G$ on a set $X$ is meant, we will often use the notations $\alpha(g,x)$ and $g\cdot x$ interchangeably. We will now look at some easy examples.
\begin{example}
    Let $(G,*)$ be any group. Then $G$ acts on itself via the usual group operation $*.$ Explicitly, the action is defined as follows: $$\alpha:G\times G\rightarrow G,\quad \alpha(g,h)=g*h.$$
    Since $G$ is closed under $*,$ the map $\alpha$ is well defined. Using the other group axioms of associativity and unitality, we see that $\alpha$ is indeed an action.
\end{example}

\begin{example}\label{cosetaction}
    Let $G$ be a group and $K$ a subgroup of $G.$ Then $G$ acts on the set of left cosets $$G/K=\{gK\ \rvert \ g\in G\} $$ by left translation, i.e. $$a\cdot gK=agK.$$
    This action clearly satisfies the conditions of a group action. The action of $G$ on $G/K$ is called the {\it natural action}.
\end{example}

\begin{example}
    The group $\SO{n}$ of $n\times n$ rotation matrices acts on $\mathbb{R}^n$ by the usual multiplication of matrices and vectors. Hence the map $$\alpha:\SO{n}\times\mathbb{R}^n\rightarrow\mathbb{R}^n,\quad (X,v)\mapsto X\cdot v$$ is well defined.
    It is clear that for all $x\in\mathbb{R}^n$ $$I_n\cdot x=x.$$ Further, the standard matrix multiplication is associative, hence for all $A,B\in\SO{n}$ and all $x\in\mathbb{R}^n,$ $$A\cdot(B\cdot x)=(AB)\cdot x.$$
\end{example}
In Definition \ref{representation} we have defined representations of Lie groups. As it turns out, they can be thought of as linear actions of a Lie group on a vector space.
\begin{example}
    Let $\rho:G\rightarrow GL(V), \ g\mapsto\rho_g$ be a representation of  a Lie group $G.$ Then the group homomorphism $\rho$ induces a group action  $$\alpha:G\times V\rightarrow V:
    \alpha(g,v)=\rho_g(v).$$ Indeed, for all $v\in V,$ and all $g_1,g_2\in G$ we have that \begin{eqnarray*}
        \alpha(e,v)&=&\rho_e(v)=1_V(v)=v,\\
        \alpha(g_1g_2,v)&=&\rho_{g_1g_2}(v)=\rho_{g_1}(\rho_{g_2}(v))=\alpha(g_1,\alpha(g_2,v)).
    \end{eqnarray*}
    This shows that the conditions of unitality and associativity are satisfied. 
    Thus, we often say that a representation $\rho$ of $G$ acts on a vector space $V.$
\end{example}

We now state a few more important definitions and properties of group actions. 
\begin{definition}
Let $\alpha: G\times X\rightarrow X$ define a group action of $G$ on $X.$ For an element $x\in X,$ the {\it orbit} of $x$ is the set $$G\cdot x=\{g\cdot x\ \rvert \ g\in G\}\ \subseteq X.$$\end{definition}

\begin{definition}
    A group action $\alpha:G\times X\rightarrow X$ of $G$ on $X$ is called {\it transitive} if for every pair $x_1,x_2\in X$ there exists $g\in G$ such that $$\alpha(g,x_1)=x_2.$$ This means that $X$ has a single orbit and for any $x\in X,$ $$G\cdot x=X.$$
\end{definition}

\begin{example}
Clearly, the action defined in example \ref{cosetaction} is transitive.  
\end{example}

It is well know that if a group $G$ acts on a non-empty set $X,$ then $X$ can be expressed as the disjoint union of the orbits under the action of $G.$ We formulate this in the following Lemma \ref{orbitdecomp}.

\begin{lemma}\label{orbitdecomp}{\rm \cite{grouptheory}}
    Let $G$ be a group acting on a non-empty set $X.$ The action of $G$ on $X$ induces a partition of $X$ into its orbits.
\end{lemma}
\begin{proof}
We define the relation $\sim$ on $X$ as follows. For any $x,y\in X,$ let $x\sim y$ if there exists $g\in G$ such that $g\cdot x=y.$
We will show that $\sim$ is an equivalence relation. First note that $\sim$ is reflexive, since for all $x\in X$ we have that $e\cdot x=x$ and thus $x\sim x$. Further, $\sim$ is symmetric. Indeed, if $x\sim y,$ then there exists $g\in G$ such that $g\cdot x=y.$ Thus, $$g^{-1}\cdot y=g^{-1}\cdot (g\cdot x)=(g^{-1}g)\cdot x=e\cdot x=x$$ and since $g^{-1}\in G$ it follows that $y\sim x.$ Lastly, $\sim$ is associative. Let $x\sim y$ and $y\sim z$, and $g_1,g_2\in G$ be such that $g_1\cdot x=y$ and $g_2\cdot y=z.$ Then $g_2g_1\in G$ and $$(g_2g_1)\cdot x=g_2\cdot y=z,$$ and thus $x\sim z.$

The equivalence classes corresponding to $\sim$ are the orbits of $X$ under the action of $G.$ The result now follows from the fact that the set of equivalence classes forms a partition of $X.$
  
\end{proof}

\begin{definition}
    Let $G$ be a group acting on $X$. Then for an element $x\in X,$ the {\it stabiliser} of $x$ is the subgroup $$G_x=\{ g\in G \ \rvert \ g\cdot x=x\} $$ of $G.$ It is also called the {\it isotropy subgroup}.
\end{definition}

\begin{definition}
    A group action of $G$ on a set $X$ is said to be {\it free} if for all $x\in M,$ the stabiliser group is trivial. In other words, $X$ is free if for all $x\in X,$ $g\cdot x=x$ implies that $g=e.$
\end{definition}
Lastly, we define a notion of compatibility of a map between two sets with the actions of a group. 
\begin{definition}
    Let $G$ act on the set $X$ via the action $\alpha,$ and on $Y$ via the action $\beta$ respectively. A map $f:X\rightarrow Y$ is {\it G-equivariant} if for all $g\in G$ $$f(\alpha(g,x))=\beta(g,f(x)).$$
\end{definition}

%\begin{definition}
%    Let $M$ be a smooth manifold and $G$ a Lie group. We will call $G$ a {\it Lie transformation group} if there exists a smooth action $\alpha$ of $G$ on $M.$
%\end{definition}

\section{Homogeneous Spaces}\label{homogeneousspaces}
After having studied group actions, we are able to define the notion of a homogeneous space. Theorems \ref{constructionth} and \ref{characterisationth} show that all homogeneous spaces arise in shape of coset manifolds $G/K$, where $G$ is a Lie group and $K$ a closed subgroup of $G.$ We are particularly interested in the case when $G$ is the group of isometries of a manifold. 
Further, we will define reductive homogeneous spaces.

\begin{definition}\label{homspace}
A smooth manifold on which a Lie group $G$ acts transitively is called a {\it homogeneous space}.
\end{definition}
We will show how to construct and characterise homogeneous spaces in Theorems \ref{constructionth} and \ref{characterisationth}.

\begin{definition}\label{defcosetmanif}{\rm \cite{Helg}}
    Let $G$ be a Lie group and $K$ a closed subgroup of $G.$ The {\it coset space} is the set of left cosets $$G/K=\{gK\ \rvert \ g\in G\}.$$ The {\it natural projection} of $G$ onto $G/K$ is the map $$\pi:G\rightarrow G/K,\quad g\mapsto gK.$$
The topology on $G/K$ is uniquely determined by the condition that $\pi:G\rightarrow G/K$ is continuous and open.
\end{definition}

\begin{theorem}{\rm \cite{Lee}}\label{constructionth} (Construction Theorem)
    Let $G$ be a Lie group and $K$ a closed subgroup of $G.$ The left coset space $G/K$ is a topological manifold of dimension equal to $\dim G-\dim K,$ and has a unique smooth structure such that the projection map $\pi:G\rightarrow G/K$ is a smooth submersion. The left action of $G$ on $G/K$ given by $$g_1\cdot(g_2\cdot K)=(g_1g_2)\cdot K$$ turns $G/K$ into a homogeneous space.
\end{theorem}

    Theorem \ref{constructionth} tells us that the coset space $G/K$ is a homogeneous space. We will now see that every homogeneous space is of the form $G/K.$

\begin{theorem}{\rm \cite{Lee}}\label{characterisationth} (Characterisation Theorem)
    Let $G$ be a Lie group and $M$ a homogeneous manifold on which $G$ acts transitively. Let $p$ be any point of $M.$ Then the stabiliser subgroup $G_p$ is a closed subgroup of $G,$ and the map $\phi:G/G_p\rightarrow M$ defined by $$\phi(gG_p)=g\cdot p$$ is an equivariant diffeomorphism.
\end{theorem}
\begin{proof}
    Consider the orbit map $$\theta_p:G\rightarrow G\cdot p\subseteq M,\quad g\mapsto g\cdot p.$$ By continuity, $G_p=(\theta_p)^{-1}(p)$ is closed. 
    
    We now show that $\phi$ is well defined. Assume $g_1G_p=g_2G_p$ are two coset representatives, which means that there exists $k\in G_p$ such that $g_1=g_2k.$ Thus, $$\phi(g_1G_p)=g_1\cdot p=g_2k\cdot p=g_2\cdot(k\cdot p)=g_2\cdot p=\phi(g_2G_p).$$ 
    
    Further, $\phi$ is $G$-equivariant because for any $g_1,g_2\in G$ $$\phi(g_1\cdot g_2G_p)=\phi(g_1g_2G_p)=g_1g_2\cdot p=g_1\cdot(g_2\cdot p)=g_1\cdot\phi(g_2G_p).$$
    We see that $\phi$ is smooth since it is obtained from $\theta_p$ by passing to the quotient.
    
    Lastly, $\phi$ is bijective. Due to transitivity of the action of $G$ on $M$, for any $q\in M$ there exists $g\in G$ such that $g\cdot p=q.$ Thus, $$\phi(gG_p)=g\cdot p=q.$$ On the other hand, if $$\phi(g_1G_p)=\phi(g_2G_p),$$ then $g_1\cdot p=g_2\cdot p$, or equivalently, $g_2^{-1}g_1\in G_p.$ But this means that $g_1G_p=g_2G_p,$ and we conclude that $\phi$ is injective. 
    It can now be shown that $\phi$ is a diffeomorphism, see for example the Equivariant Rank Theorem in \cite{Lee}.
\end{proof}

This leads  us to the following important result by Myers and Steenrod from 1939.
\begin{theorem}[Myers-Steenrod, 1939]{\rm \cite{MS}}\label{MyersSteenrod}
    Any closed group of isometries of a Riemannian manifold of class $C^r \ (r\geq 2)$ is a Lie group of isometries. 
\end{theorem}

\begin{definition}{\rm \cite{Arv},\cite{Gud-Rie}}
    Denote $I(M)$ the isometry group of a Riemannian manifold $(M,g).$ A {\it Riemannian homogeneous space} is a Riemannian manifold on which $I(M)$ acts transitively. In other words, for all $p,q\in M$ there exists $\phi_{p,q}\in I(M)$ such that $\phi_{p,q}(p)=q.$
\end{definition}
Thus, the crucial difference between {\it Riemannian} homogeneous spaces and homogeneous spaces is the compatibility of the group action and the metric. Informally, Riemannian homogeneous spaces are often said to “look the same everywhere". This refers to both the smooth structure and the Riemannian metric.

As a first example, we see that Riemannian Lie groups are Riemannians homogeneous spaces. 
\begin{proposition}{\rm \cite{Gud-Rie}}
    A Lie group $G$ equipped with a left-invariant metric $g$ is a Riemannian homogeneous space.
\end{proposition}
\begin{proof}
    For every $p,q\in G$ the left translation $L_{qp^{-1}}$ satisfies $L_{qp^{-1}}(p)=q.$ Since $g$ is left-invariant, every left translation is an isometry. This shows that $I(G)$ acts transitively on $G.$
\end{proof}

   From Theorems \ref{characterisationth} and \ref{MyersSteenrod} it follows that a Riemannian homogeneous space $M$ is isomorphic to $G/K,$ where $G=I(M),$ and $K=G_p$ is the stabiliser subgroup of any $p\in M.$ It can be shown, for example in \cite{Helg} by Helgason, that $G_p$ is compact for any choice of $p.$

\begin{example}\label{sphereiso}{\rm \cite{Arv}}
The isometry group of $S^n\subset\rn^{n+1}$ is the group of orthogonal matrices $\O{n+1},$ where $\O{n+1}$ acts on $S^n$ by standard matrix multiplication.

The action is well defined. Note that $p\in S^n$ if and only if $(p,p)=p^tp=1.$ Assume that $p\in S^n$ and $A\in\O{n+1}.$ Then $A\cdot p\in S^n,$ since $$(Ap)^t\cdot(Ap)=p^tA^tAp=pp^t=1.$$
Unitality and associativity follow immediately from the properties of matrix multiplication.

Similarly, it can be shown that the action is an isometry. To see this, let $\langle\cdot,\cdot\rangle $ denote the standard metric on $\rn^{n+1}.$ Then for $A\in\O{n+1},$ and $X,Y\in C^{\infty}(S^n),$ $$\langle A\cdot X,A\cdot Y\rangle =X^tA^tAY=X^tY=\langle X,Y\rangle .$$

The action is transitive. Indeed, let $p_N=(1,0,\dots,0)^t$ and $p=(p_1,\dots,p_{n+1})^t\in S^n.$ Then with the Gram-Schmidt process we can find vectors $v_1,\dots v_n$ such that $\{p,v_1,\dots v_n\}$ forms an orthonormal basis of $\rn^{n+1}.$ Then $$A=(p,v_1,\dots v_n)\in\O{n+1}.$$ Clearly, $A\cdot p_N=p.$ For any $q\in S^n,$ there exists $B_q\in\O{n}$ such that $B_q\cdot p_N=q,$ as demonstrated earlier. Thus, $AB_q^{-1}\cdot q=p.$

The isotropy subgroup of $p_N$ is the subgroup of $\O{n+1}$ if the form 
$$\begin{pmatrix}
        1&0\\
        0&M
    \end{pmatrix}, \ M\in\O{n}.$$ Consequently, $$S^n\cong\O{n+1}/\O{n}.$$
\end{example}
    
\begin{definition}
    Let $G$ be a Lie group. The {\it identity component} $G_o$ of $G$ is the largest connected subset of $G$ containing the identity element $e$ of $G.$ 
    %It can be shown that $G_o$ is a normal subgroup of $G.$
\end{definition}

\begin{theorem}{\rm \cite{Helg}}\label{identitycomptransitive}
    Let $G$ be a Lie group acting transitively on a connected smooth manifold $M.$ Then the identity component $G_o$ of $G$ also acts transitively on $M.$
\end{theorem}
\begin{proof}
    Since $G_o$ is a subgroup of $G,$ we may consider the action of $G_o$ on $M$ induced by the action of $G$ on $M.$
    By Lemma \ref{orbitdecomp}, we may decompose $M$ into the disjoint union $$M\cong\coprod_{p\in X_0} G_o\cdot p,$$ where $X_0$ is a complete list of orbit representatives of $M$ under the action of $G_o.$ Thus, any two orbits $G_o\cdot p_1, G_o\cdot p_2$ are either disjoint or equal. However, each orbit is also open in the connected manifold $M,$ hence all orbits must coincide. We conclude that $$M\cong G_o\cdot p$$ for any $p\in M.$  
\end{proof}
Theorem \ref{identitycomptransitive} has the following consequence.
\begin{corollary}\label{identitycompisometries}
We denote the largest connected subgroup of $I(M)$ containing the identity element by $I_o(M).$ If $M$ is connected, then $I_o(M)$ also acts transitively on $M.$
\end{corollary}

\begin{example}\label{spheresoso}
It is well known that $\SO{n+1}$ is the connected component of $\O{n+1}.$
 By Theorem \ref{identitycomptransitive}, $\SO{n+1}$ acts transitively on $S^n.$   For the north pole $$p_N=(1,0,\dots,0)^t\in S^n,$$ an element in the stabiliser subgroup is of the form $$\begin{pmatrix}
        1&0\\
        0&M
    \end{pmatrix}, $$\ where $M\in\SO{n}.$ Hence we may identify the stabiliser subgroup with $\SO{n}.$ Thus, $$S^n\cong\SO{n+1}/\SO{n}.$$  
\end{example}
We will now define reductive homogeneous spaces. For this special class of homogeneous spaces, the tangent space of $G/K$ at the origin $o$ can be conveniently identified with a subspace of $\g.$
\begin{definition}{\rm \cite{Arv}}\label{reductive}
    Denote the Lie algebras of $G,K$ by $\g,\k$ respectively. A homogeneous space $G/K$ is called {\it reductive} if there exists a subspace $\m$ of $\g$ such that $$\g=\k\oplus\m$$ and $\Ad(k)\m\subset\m$ for all $k\in K,$ that is, $\m$ is $\Ad(K)$-invariant.
\end{definition}

\begin{remark}{\rm \cite{Arv}}
    If $\Ad(k)\m\subset\m$ for all $k\in K,$ then $[\k,\m]\subset\m.$ The converse holds if $K$ is connected.
\end{remark}

\begin{observation}\label{tangentspacem}{\rm \cite{Arv}}
    Let $M=G/K$ be a homogeneous space. Consider the differential $\diff\pi$ of the standard projection $\pi:G\rightarrow G/K$ given by $$\pi(g)\rightarrow gK.$$ We denote $o=\pi(e)=eK.$ Then the differential $\diff\pi_e:\g\rightarrow T_o(G/K)$ acts on $X\in\g$ as follows. 
    $$\diff\pi_e(X)=\frac{d}{ds}(\pi\circ\exp(sX))\rvert_{s=0}.$$
    Thus, $\ker \diff\pi_e=\k.$ Since $\diff\pi$ is surjective,it follows that $$\g/\k\cong T_o(G/K).$$ 
    If $M$ is reductive, it now immediately follows that $$\m\cong T_o(G/K).$$
\end{observation}

\begin{definition}{\rm \cite{Arv}}
    Let $M=G/K$ be a homogeneous space. Let $p\in M$ and $K=G_p.$ For every $a\in G,$ consider the diffeomorphism $$\tau_a:G/K\rightarrow G/K,\ gK\mapsto agK.$$ Then $$\rho:K\rightarrow GL(T_{o}M),\ k\mapsto(\diff\tau_k)_{o}$$ is called the {\it isotropy representation} of $K$ at $o=eK$
\end{definition}

\begin{remark}
    The isotropy representation is related to the adjoint representation in the following way. Let $G$ act on itself by conjugation, i.e. the group action $\alpha$ of $G$ is given by $$\alpha:G\times G, \ (g,h)\mapsto I_g(h),$$ where $I_g$ denotes the inner automorphism (see Definition \ref{innerauto}). Then the isotropy representation $\rho$ of $G_e$ at $e$ coincides with the adjoint representation $\Ad\rvert_K$ restricted to $K.$ For more details on this we refer to \cite{Arv}.
\end{remark}

\section{Symmetric Spaces}\label{symmetricspaces}
In the following, we discuss a particular class of homogeneous spaces, namely the symmetric spaces. We will prove that they are indeed homogeneous and reductive. This allows us to draw conclusions concerning the Lie algebra of a symmetric space. 
\begin{definition}{\rm \cite{Arv}}
    A connected Riemannian manifold $M$ is called a {\it symmetric space} if for each $p\in M$ there exists a unique isometry $s_p:M\rightarrow M$ such that $$s_p(p)=p,\ \text{and} \ (\diff s_p)_p=-\Id_p.$$
\end{definition}

\begin{example}{\rm \cite{Arv}}
    $\rn^n$ is symmetric. For $p\in\rn^n,$ the desired symmetry is given by $s_p(x)=2p-x.$
\end{example}
In the following lemma, we study how symmetries act on geodesics.

\begin{lemma}{\rm \cite{Arv}}\label{reversinggeodesics}
    Let $(M,g)$ be a symmetric space and for $p\in M$, let $s_p$ be the symmetry of $M$ at $p.$ Then $s_p$ reverses geodesics passing through $p$, i.e. for a geodesic $\gamma:(-\epsilon,\epsilon)\rightarrow U\subset M$ with $\gamma(0)=p,$ we have that $$s_p(\gamma(t))=\gamma(-t).$$
\end{lemma}
\begin{proof}
    Since isometries preserve geodesics, the image $\beta(t)=s_p(\gamma(t))$ of $\gamma$ under $s_p$ is a geodesic on $M.$ Further, we observe that $$\beta(0)=s_p(\gamma(0))=s_p(p)=p,\quad \Dot{\beta}(0)=-\Id(\Dot{\gamma}(0))=-\Dot{\gamma}(0).$$ However, the geodesic $\alpha(t)=\gamma(-t)$ also satisfies $$\alpha(0)=p,\quad \Dot{\alpha}(0)=-\Dot{\gamma}(0).$$ Due to the uniqueness of geodesics, $s_p(\gamma(t))=\gamma(-t)$ for all $t$ in the interval $(-\epsilon,\epsilon).$
\end{proof}
Lemma \ref{reversinggeodesics} proves to be useful for the following two results.

\begin{proposition}{\rm \cite{Helg}}
    A symmetric Riemannian manifold $(M,g)$ is geodesically complete.
\end{proposition}
\begin{proof}
    We wish to show that for every $p\in M$ and every $v\in T_pM$, there exists a geodesic $\gamma:\rn\rightarrow M$ defined on the whole of $\rn$ such that $\gamma(0)=p$ and $\Dot{\gamma}(0)=v.$ Assume towards a contradiction that there exist $p\in M$ and $ v\in T_pM$ such that the geodesic $\gamma:[0,\alpha)\rightarrow M,$ where $\alpha<\infty,$ which satisfies $\gamma(0)=p$ and $\Dot{\gamma}(0)=v$, cannot be extended. Choose $\beta\in(\frac{\alpha}{2},\alpha),$ and let $\Tilde{\gamma}:[0,2\beta)\rightarrow M$ be the curve $$
        \Tilde{\gamma}(t)=\begin{cases}
             \gamma(t) & {\rm if}\   t\in[0,\beta],\\
             s_{\gamma(\beta)}(\gamma(2\beta-t)) &{\rm if }\  t\in(\beta, 2\beta).
        \end{cases}$$
   This means that on the second part of the interval, we reflect $\gamma$ by $s_{\gamma(\beta)},$ as we have discussed in Lemma \ref{reversinggeodesics}. We observe that $\Tilde{\gamma}$ is continuous at $\beta$ and $\Tilde{\gamma}(\beta)=\gamma(\beta).$ Further, it holds that $$\Dot{\Tilde{\gamma}}(\beta)=\diff s_{\gamma(\beta)}(-\Dot{\gamma}(\beta))=\Dot{\gamma}(\beta).$$ By the uniqueness of geodesics, $\Tilde{\gamma}$ is an extension of $\gamma.$
   
   Since $2\beta>\alpha,$ this contradicts the maximality of $\gamma.$ We conclude that $M$ is geodesically complete.
\end{proof}

\begin{theorem}{\rm \cite{Helg}}\label{symmetricimplieshomogeneous}
    A symmetric Riemannian manifold (M,g) is homogeneous.
\end{theorem}
\begin{proof}
    We will show that the group of isometries $I(M)$ acts transitively on $M.$ Let $p,q\in M.$ 
    Assume there exists a geodesic $\gamma:[0,1]\rightarrow M$ such that $\gamma(0)=p,$ and $\gamma(1)=q.$ Then the isometry $s_{\gamma(\frac{1}{2})}\in I(M)$ reverses the geodesic, as it was shown in Lemma \ref{reversinggeodesics}. In particular, this means that $$s_{\gamma(\frac{1}{2})}(p)=q,\ {\rm and } \ s_{\gamma(\frac{1}{2})}(q)=p.$$
    In general, since $M$ is path connected, there exists a piecewise smooth geodesic joining any points $p$ and $q$ of $M.$ Then the finite composition of isometries as described above yields the desired element of $I(M).$
\end{proof}
Theorem \ref{thGK} shows how a symmetry $s_p$ of $M$ induces further structure on a symmetric space $M.$
\begin{theorem}{\rm \cite{Helg},\cite{O'Neill},\cite{Esch}}\label{thGK}
Let $M$ be a symmetric space. Then $M$ is of the form $M=G/K,$ where 
\begin{enumerate}
    \item $G$ is a connected Lie group,
    \item there exists an involution $\sigma:G\rightarrow G$ such that $G_\sigma=\{g\in G \ \rvert \ \sigma(g)=g\}$ and $G_\sigma^{o}$ is the identity component of $G_\sigma$,
    \item $K$ is a closed subgroup of $G$ satisfying $$G_\sigma^{o}\subset K\subset G_\sigma.$$
\end{enumerate}
\end{theorem}
\begin{proof}
    Assume $M$ is a symmetric space. We have seen in Theorem \ref{symmetricimplieshomogeneous} that $I(M)$ acts transitively on $M.$ Let $G=I_o(M)$ be the identity component of the isometry group of $M.$ Due to Corollary \ref{identitycompisometries}, $G$ acts transitively on $M.$ Choose $p\in M$ and let $K=G_p$ be the isotropy subgroup of $p.$ From the Characterisation Theorem \ref{characterisationth} it follows that $M$ is diffeomorphic to $G/K.$ It is left to show that there exists an involution $\sigma$ of $G$ such that $G_\sigma^{o}\subset K\subset G_\sigma.$    
    
    By definition, there exists a unique isometry $s_p$ of $M$ that fixes $p$ and satisfies $(\diff s_p)_p=-\Id_p.$ We now construct $\sigma$ as follows: $$\sigma:G\rightarrow G, \quad \sigma(g)=s_p\cdot g\cdot s_p.$$
    Clearly $\sigma^2=\Id_G,$ hence $\sigma$ is an involution.  If $k\in K,$ then $$\sigma(k)(p)=(s_p\cdot k\cdot s_p)(p)=p,$$ and $$(\diff\sigma(k))_p=(\diff s_p)_p\circ(\diff k)_p\circ(\diff s_p)_p=-\Id_p\circ(\diff k)_p\circ -\Id_p=(\diff k)_p.$$ This shows that $K\subset G_\sigma.$
    
    For the other inclusion, first note that $p$ is an isolated point of $$\{x\in M \ \rvert \ s_p(x)=x\}.$$ This is due to the fact that $ds_p$ maps every non-zero vector $v\in T_pM$ to $-v,$ as we have seen in Lemma \ref{reversinggeodesics}. This means that there exists an open neighbourhood $V$ of $p$ such that $p$ is the only point in $V$ which is fixed by $s_p.$ Let $U\subset G_\sigma$ be the neighbourhood of the identity element $e$ of $G$ that preserves $V,$ i.e. $$U=\{g\in G_\sigma \ \rvert \ g(p)\in V\}.$$ Then for every $g\in U,$ $$s_p\cdot g (p)=s_p\cdot g\cdot s_p(p)=\sigma(g)(p)=g(p).$$ In other words, $s_p$ fixes $g(p)\in V,$ but due to our assumptions on $V,$ we conclude that $g(p)=p.$ This shows that $K$ contains a neighbourhood of $e\in G_\sigma^{o}.$ It follows that $G_\sigma^{o}\subset K.$
\end{proof}

\begin{definition}
    We call $(G,K)$ a {\it symmetric pair} if there exists an involution $\sigma$ of $G$ such that $G_\sigma^{o}\subset K\subset G_\sigma.$
\end{definition}
Lemma \ref{lemmasympair} gives us more insight of the Lie algebra of a symmetric space. In particular, it shows that a symmetric space is reductive
\begin{lemma}{\rm \cite{O'Neill}}\label{lemmasympair}
    Let $G,K$ be a symmetric pair. Then: \begin{enumerate}
        \item We have the direct decomposition $$\g=\k\oplus\m,$$ where $\k,\m$ are the eigenspaces of $\diff\sigma$ corresponding to the eigenvalues $1$ and $-1,$ i.e. $$\k=\{X\in\g \ \rvert \ \diff\sigma(X)=X\},$$ $$\m=\{X\in\g \ \rvert \ \diff\sigma(X)=-X\}.$$
        \item $\k$ is the Lie algebra of $K.$
        \item $\m$ is $\Ad_K$ invariant.
    \end{enumerate}
\end{lemma}
\begin{proof}
    Since $\sigma$ is an involution, so is $\diff\sigma.$ For every $X\in\g$, we may now set $$X_\k=\frac{1}{2}\cdot(X+\diff\sigma(X))\in\k \quad \textrm{and}\quad X_\m=\frac{1}{2}\cdot(X-\diff\sigma(X))\in\m.$$ Since $X=X_\k+X_\m,$ it follows that $$\g=\k+\m.$$ Clearly $\k\cap\m=0.$ Thus, $\g=\k\oplus\m$ as a direct sum.

    For all $X\in\k$ we have that $\diff\sigma(X)=X$ because $K\subset G_\sigma.$
    Conversely suppose that $X\in\g$ satisfies $\diff\sigma(X)=X.$ Let $\gamma$ be the corresponding one-parameter subgroup. Since $\sigma\circ\gamma$ and $\gamma$ have the same initial velocity, it follows that $\sigma\circ\gamma=\gamma.$ We conclude that $\gamma\in G_\sigma^{o}\subset K,$ and thus $X\in\k.$

    For the last statement, let $X\in\m, k\in K.$ Note that $$\sigma\circ I_k(g)=\sigma(kgk^{-1})=k\sigma(g)k^{-1}=I_k\circ\sigma(g)$$ for all $g\in G.$ But now $$\diff\sigma(\Ad_k(X))=\diff(\sigma\circ I_k)(X)=\diff(I_k\circ\sigma)(X)=\Ad_k(-X)=-\Ad_k(X).$$ This shows that $\Ad_k(X)\in\m.$
\end{proof}

There exists a converse statement of Theorem \ref{thGK}.
\begin{theorem}{\rm \cite{O'Neill}}\label{GKsymspace}
    Let $(G,K)$ be a symmetric pair such that $G$ is connected and $K$ is closed. Let $\sigma$ be an involutive automorphism of $G$ such that $$G_\sigma^o\subset K\subset G_\sigma.$$ Then every $G$-invariant metric on $M=G/K$ turns $M$ into a Riemannian symmetric space such that $s\circ\pi=\pi\circ\sigma,$ where $s$ is the symmetry of $M$ at $o=eK,$ and $\pi:G\rightarrow M$ the standard projection.
\end{theorem}
\begin{proof}
    We define the function $s:M\rightarrow M$ by $s\circ\pi=\pi\circ\sigma.$ Note that $s$ is well-defined, since $\pi(g_1)=\pi(g_2)$ implies that $g_1K=g_2K.$ Further, $\sigma$ fixes $K,$ and thus $\sigma(g_1)K=\sigma(g_2)K$ and thus, $\pi\circ\sigma(g_1)=\pi\circ\sigma(g_2).$

    Note that $s$ is an involutive diffeomorphism, which follows from $\sigma$ being an involution.

    We have that $s(o)=o.$ For $x\in T_0M$ there exists $X\in\g$ such that $\diff\sigma(X)=-X$ and $\diff\pi(X)=x,$ as established in Lemma \ref{lemmasympair}. Now $$\diff s(x)=(\diff s\circ \diff\pi)(X)=(\diff\pi\circ \diff\sigma)(X)=\diff\pi(-X)=-x.$$ Hence $s$ is a symmetry at $o.$
    
In the following, we will make use of the fact  that $s\circ\tau_g=\tau_{\sigma(g)}\circ s.$ Indeed, for $a\in G$ it holds that
\begin{eqnarray*}
(s\circ\tau_g)(\pi(a))&=&(s\circ\tau_g)(aK)\\
&=&(s\circ\pi)(ga)\\
&=&(\pi\circ\sigma)(ga)\\
&=&\pi(\sigma(g)\sigma(a))\\
&=&(\tau_{\sigma(g)}\circ\pi)(\sigma(a))\\
&=&(\tau_{\sigma(g)}\circ s)(\pi(a)).
\end{eqnarray*}
    Finally, $s$ is an isometry with respect to any $G$-invariant metric $g$ on $M.$ To see this, let $y\in T_g(M), y_0=\diff\tau_{g^{-1}}(y)\in T_0M.$ 
Now, \begin{eqnarray*}
    g(\diff s(y),\diff s(y))&=&g(\diff s\circ \diff \tau_g(y_0),\diff s \circ\diff \tau_g(y_0))\\
    &=&g(\diff\tau_{\sigma(g)}\circ\diff s(y_0),\diff\tau_{\sigma(g)}\circ\diff s(y_0))\\
    &=&g(\diff s(y_0),\diff s(y_0))\\
    &=&g(-y_0,-y_0)\\
    &=&g(y,y).
\end{eqnarray*}
    Hence we have found a global symmetry $s$ at $o.$ To conclude the proof, we note that at every point $p=\tau(o)\in M$ of a homogeneous space, the desired symmetry at $p$ is given by $\tau s\tau^{-1}.$
\end{proof}

\section{The Duality}\label{duality}
In this section we will talk about the notion of duality between compact and non-compact symmetric spaces. We will also give an overview over the simply connected, irreducible symmetric spaces of compact type.

We first introduce the so-called Killing form of a Lie algebra.
\begin{definition}{\rm \cite{Helg}}
    Let $\g$ be a complex Lie algebra. The {\it Killing form} $B$ of $\g$ is the symmetric bilinear form $$B:\g\times\g\rightarrow\rn, \quad B(X,Y)=\trace(\ad X\circ\ad Y).$$
\end{definition}
We recall the following properties bilinear forms can take. 
\begin{remark}
     A bilinear form $f$ on a finite dimensional vector space $V$ is \begin{enumerate}
        \item non-degenerate if $f(x,y)=0$ for all $y\in V$ implies that $x=0,$
        \item positive definite if $f(v,v)\geq0$ for all $v\in V,$ with equality only if $v=0,$
        \item negative definite if $f(v,v)\leq0$ for all $v\in V,$ with equality only if $v=0.$
    \end{enumerate}
\end{remark}
This allows us to make the following definitions.
\begin{definition}{\rm \cite{Arv}}
    A Lie algebra $\g$ is called {\it semisimple} if its Killing form is non-degenerate. %It is called simple if it is non-abelian and its only ideals are $\{0\}$ and $\g.$
    We call a Lie group $G$ {\it semisimple} if its Lie algebra is semisimple.
\end{definition}

We have seen in Lemma \ref{lemmasympair} that a symmetric space is reductive. Recall that reductive spaces have a direct decomposition $\g=\k\oplus\m,$ which we have elaborated in Definition \ref{reductive}.
\begin{definition}{\rm \cite{Arv}}
A symmetric space is said to be of {\it compact type} if the Killing form $B$ of $\g$ is negative definite, and of {\it non-compact type} if $B$ is negative definite on $\k$ and positive definite on $\m.$
\end{definition}
Before we define duality, we clarify what is meant by a normal homogeneous space.
\begin{remark}{\rm \cite{Arv}}
    Let $G/K$ be a homogeneous space and $\langle\cdot,\cdot\rangle $ be an $\Ad$-invariant scalar product on The Lie algebra $\g$ of $G.$ Let $\g=\k\oplus\m$ be a reductive decomposition, where $\m$ is the orthogonal complement of $\k$ in $\g$ with respect to the scalar product $\langle\cdot,\cdot\rangle .$ Then the restriction of $\langle\cdot,\cdot\rangle $ to $\m$ induces a $G$-invariant metric on $G/K,$ which is called a {\it normal homogeneous Riemannian metric.}
\end{remark}

\begin{definition}{\rm \cite{Arv}}
    Two normal symmetric spaces $M=G/K$ and $M^*=G^*/K^*$ are said to be {\it dual} if the following statements hold:\begin{enumerate}
        \item There exists an isomorphism of Lie algebras $\phi:\k\rightarrow\k^*$ such that $$B^*(\phi(V),\phi(W))=-B(V,W)$$ for all $V,W\in\k.$
        \item There exists a linear isometry $\Tilde{\phi}:\m\rightarrow\m^*$ such that $$[\Tilde{\phi}(X),\Tilde{\phi}(Y)]=-\Tilde{\phi}([X,Y])$$ for all $X,Y\in\m.$
    \end{enumerate}
\end{definition}

We conclude the chapter with Cartan's famous result from 1926. He gave a complete classification of the Riemannian symmetric spaces.
\begin{theorem}{\rm \cite{Arv}}
The simply connected, isotropy irreducible symmetric spaces of compact type are the following.
\begin{enumerate}
    \item \textbf{Compact simply connected groups}: $\SU{n},$ $\Spin{n},$ $\Sp{n},$ $E_6,$ $E_7,$ $E_8,$ $F_4,$ $G_2.$
    \item \textbf{Classical spaces:} $\SO{m+n}/\SO{m}\times\SO{n},$ $\SU{m+n}/\SU{m}\times\SU{n},$ $\Sp{m+n}/\Sp{m}\times\Sp{n},$ $\SU{n}/\SO{n},$ $\SU{2n}/\Sp{n},$ $\SO{2n}/\U{n},$ $\Sp{n}/\U{n}.$
    \item \textbf{Exceptional spaces:} $E_6/\SU{6}\times\SU{2},$ $E_6/\SO{10}\times\SO{2},$ $E_6/F_4,$ $E_6/\Sp{4},$ $E_7/\SU{8},$ $E_7/\SO{12}\times\SU{2},$ $E_7/E_{6}\times\SO{2},$ $E_8/\SO{16},$ $E_8/E_7\times\SU{2},$ $F_4/\Sp{3}\times\Sp{1},$ $G_2/\SO{4}.$
\end{enumerate}
    
\end{theorem}

%%%%%%%%%%%%%%%%%%%%%%%%%%%%%%%%%%%%%%%%%%%%%%%%%%%%%%%%%%%

\chapter{Minimal Submanifolds via Complex-valued Eigenfunctions}

Section \ref{eigenfunctions} serves as an introduction to eigenfunctions. They are functions which are eigen with respect to the tension field and conformality operator. There exist eigenfunctions on all compact simply connected Lie groups and classical symmetric spaces, as shown in Table \ref{table}. A lot of those contributions are due to Gudmundsson, Siffert and Sobak (see \cite{Gud-Sif-Sob-2}) on the symmetric spaces, and Gudmundsson and Ghandour on the real, complex and quaternionic Grassmannians (see \cite{Gha-Gud-4} and \cite{Gha-Gud-5}). In their paper \cite{Gud-Sak-1}, Gudmundsson and Sakovich have shown that eigenfunctions can be used to produce harmonic morphisms. Eigenfunctions can also be applied to construct proper $p$-harmonic functions, as shown by Gudmundsson and Sobak in \cite{Gud-Sob-1}.

In Section \ref{newmethod}, we state Theorem \ref{GM}, which is a fundamental tool in the following chapters. This new result by Gudmundsson and Munn from their work \cite{Gud-Mun-1} allows us to find minimal submanifolds with the help of complex-valued eigenfunctions. 

\section{Eigenfunctions} \label{eigenfunctions}
We will now define eigenfunctions. In Definitions \ref{LB} and \ref{conformalityop} we have already defined the Laplace-Beltrami and conformality operators. We will show in Example \ref{efvshm} and Theorem \ref{eigenfamharmonicmorph} how eigenfunctions are related to harmonic morphisms. We will then give some important results, which will help us construct eigenfunctions on quotient spaces. 
\begin{definition}
    Let $(M,g)$ be a Riemannian manifold. A function $\phi:M\rightarrow\mathbb{C}$ is a $(\lambda,\mu)$-{\it eigenfunction} if there exist $\lambda,\mu\in\mathbb{C}$ such that $$\tau(\phi)=\lambda\cdot\phi\ \textrm{and} \ \kappa(\phi,\phi)=\mu\cdot\phi^2.$$
\end{definition}
This definition was first formulated by Gudmundsson and Sakovich in \cite{Gud-Sak-1} in 2007.

As we will see, we have already encountered a particular class of eigenfunctions in a previous chapter. 
\begin{example}\label{efvshm}
    Clearly, a harmonic morphism (see Definition \ref{harmonicmorphism}) $\phi:M\rightarrow\mathbb{C}$ is an eigenfunction with $(\lambda,\mu)=(0,0).$ This follows from the famous result by Fuglede and Ishihara, which was stated here as Theorem \ref{fuglede}. The details of this are discussed by Ghandour and Gudmundsson in their paper \cite{Gha-Gud-1}.
\end{example}
Gudmundsson and Sobak showed in \cite{Gud-Sob-1}, that eigenfunctions can be used to construct proper $p$-harmonic functions.
\begin{theorem}{\rm \cite{Gud-Sob-1}}
    Let $\phi:(M,g)\rightarrow\cn$ be a complex-valued $(\lambda,\mu)$-eigenfunction on a Riemannian manifold. Then for a natural number $p\geq1$ and $(c_1,c_2)\in\cn^2,$ any non-vanishing function $$\Phi_p:W=\{x\in M \ \rvert \ \phi(x)\notin(-\infty,0]\}\rightarrow\cn$$ satisfying
    $$\Phi_p(x)=\begin{cases}
        c_1\log(\phi(x))^{p-1}, & \textrm{if}\ \mu=0,\lambda\neq0\\
        c_1\log(\phi(x))^{2p-1}+c_2\log(\phi(x))^{2p-2},&\textrm{if}\ \mu\neq0,\lambda=\mu\\
        c_1\phi(x)^{1-\frac{\lambda}{\mu}}\log(\phi(x))^{p-1}+c_2\log(\phi(x))^{p-1},   & \textrm{if} \ \mu\neq0,\lambda\neq\mu\\
    \end{cases}$$
    is proper $p$-harmonic on its open domain $W$ in $M.$
\end{theorem}

We will now introduce the notion of eigenfamilies. 
\begin{definition}
    Let $(M,g)$ be a Riemannian manifold. A set $\mathcal{E}$ of complex-valued functions from $M$  is a $(\lambda,\mu)$-{\it eigenfamily} (or just eigenfamily) if there exist $\lambda,\mu\in\mathbb{C}$ such that for all $\phi,\psi\in\mathcal{E}$ $$\tau(\phi)=\lambda\cdot\phi,\ \textrm{and} \ \kappa(\phi,\psi)=\mu\cdot\phi\cdot\psi.$$
\end{definition}
Eigenfamilies are powerful tools to construct harmonic morphisms, as Gudmundsson and Sakovich showed in their paper \cite{Gud-Sak-1}.
\begin{theorem}{\rm \cite{Gud-Sak-1}}\label{eigenfamharmonicmorph}
    Let $(M,g)$ be a semi-Riemannian manifold and $$\mathcal{E}=\{\phi_1,\dots,\phi_n\}$$ be a finite eigenfamily of complex-valued functions on $M.$ If $P,Q:\cn^n\rightarrow\cn$ are linearly independent homogeneous polynomials of the same positive degree, then the quotient $$\frac{P(\phi_1,\dots,\phi_n)}{Q(\phi_1,\dots,\phi_n)}$$ is a non-constant harmonic morphism on the open and dense subset $$\{p\in M \ \rvert \ Q(\phi_1(p),\dots,\phi_n(p))\neq0\}.$$
\end{theorem}
Our aim is now to show how to construct eigenfunctions on Riemannian symmetric spaces. We first need the well-known composition rule for the tension field $\tau.$

\begin{proposition}{\rm \cite{Bai-Woo-book}}\label{taucompositionrule}
    (Composition Law) The tension field of the composition of two maps $\phi:M\rightarrow N$ and $\psi:N\rightarrow P$ is given by 
    $$\tau(\psi\circ\phi)=\diff\psi(\tau(\phi))+\trace\nabla \diff\psi(\diff\phi,\diff\phi).$$
\end{proposition}
Leveraging the composition law, Gudmundsson and Ghandour established the following result.
\begin{proposition}{\rm \cite{Gha-Gud-4}}\label{dilationtau}
    Let $\pi:(\hat{M},\hat{g})\rightarrow(M,g)$ be a submersive harmonic morphism between Riemannian manifolds. Further let $\phi:(M,g)\rightarrow\cn$ be a smooth function and $\hat{\phi}:(\hat{M},\hat{g})\rightarrow\cn$ be the composition $\hat{\phi}=\phi\circ\pi.$ If the dilation $\Lambda:\hat{M}\rightarrow\rn$ of $\pi$ is constant-valued, then the tension fields on $M$ and $\hat{M}$ satisfy
    $$\tau(\phi)\circ\pi=\Lambda^{-2}\tau (\hat{\phi}).$$
\end{proposition}
The following lemma allows us to construct eigenfunctions on quotient spaces. This has already been established in the paper \cite{Gud-Sif-Sob-2} by Gudmundsson, Siffert and Sobak.
\begin{proposition}{\rm \cite{Gud-Mon-Rat-1},\cite{Gud-Sif-Sob-2}}\label{efquotient}
    Let $\pi:(\hat M,\hat g)\to (M,g)$ be a harmonic Riemannian submersion between Riemannian manifolds. Further let $\phi:(M,g)\to\C$ be a smooth function and $\hat\phi:(\hat M,\hat g)\to\C$ be the composition $\hat\phi=\phi\circ\pi$. Then the corresponding tension fields $\tau$ and conformality operators $\kappa$  satisfy
$$\tau(\hat\phi)=\tau(\phi)\circ\pi\ \ \text{and}\ \
\kappa(\hat\phi,\hat\psi) = \kappa(\phi,\psi)\circ\pi.$$
\end{proposition}

\begin{proof} To compute the tension field we use Proposition \ref{dilationtau} and the fact that $\pi$ is an harmonic morphism with constant dilation $\lambda\equiv1.$ The arguments needed here can be found in \cite{Gud-Sif-Sob-2}.
\end{proof}
The diagram below illustrates the situation considered in Proposition \ref{efquotient}
\begin{center}\begin{tikzpicture}
  \matrix (m)
    [
      matrix of math nodes,
      row sep    = 3em,
      column sep = 4em
    ]
    {
      G/K              &  \\
      G &      \cn       \\
    };
  \path
    (m-2-1) edge [->>] node [left] {$\pi$} (m-1-1)
    (m-1-1)
      edge [->] node [above] {${\phi}$} (m-2-2)
    (m-2-1) edge [->] node [below] {$\hat{\phi}$} (m-2-2);
\end{tikzpicture}\end{center} 

Proposition \ref{efquotient} has the following consequence. 
\begin{corollary}\label{inducedefs}
Let $\pi:(\hat M,\hat g)\to (M,g)$ be a harmonic Riemannian submersion.  For a complex-valued function smooth function $\phi:(M,g)\to\C$ let $\hat \phi:(\hat M,\hat g)\to\C$ be the composition $\hat \phi=\phi\circ\pi$. Then the following statements are equivalent.
\begin{enumerate}
\item[(i)] 
$\phi:M\to\C$ is a $(\lambda,\mu)$-eigenfunction on $M$.
\item[(ii)] 
$\hat\phi:\hat M\to\C$ is a $(\lambda,\mu)$-eigenfunction on $\hat M$.
\end{enumerate}
\end{corollary}

We will now discuss eigenfunctions in the context of the duality between compact and non-compact spaces. 
\begin{proposition}{\rm \cite{Gud-Sif-Sob-2}}
    If $\phi:G/K\rightarrow\cn$ is an eigenfunction on the non-compact symmetric space $G/K,$ such that
    $$\tau(\phi)=\lambda\cdot\phi, \ \textrm{and} \ \kappa(\phi,\phi)=\mu\cdot\phi^2,$$ then its dual function $\phi^*:U/K\rightarrow\cn$ on the compact $U/K$ satisfies $$\tau(\phi^*)=-\lambda\cdot\phi^*, \ \textrm{and} \ \kappa(\phi^*,\phi^*)=-\mu\cdot{\phi^*}^2.$$
\end{proposition}
Recently, Riedler and Siffert showed in their paper \cite{Rie-Sif} that the eigenvalues $\lambda,\mu$ of an eigenfunction on a compact manifold satisfy the following condition.

\begin{theorem}{\rm \cite{Rie-Sif}}
Let $(M,g)$ be compact and $f:M\rightarrow\cn$ a $(\lambda,\mu)$-eigenfunction. Then $\lambda\leq\mu\leq0.$
\end{theorem}

Table \ref{table}, which is taken from \cite{Gud-Mun-1}, gives an overview of some currently known eigenfunctions. Here $U/K$ refers to the domain of the eigenfunction.
\renewcommand{\arraystretch}{2}
\begin{table}[h]\label{table}
	\makebox[\textwidth][c]{
		\begin{tabular}{cccc}
			\midrule
			\midrule
$U/K$	& $\lambda$ & $\mu$ & Eigenfunctions \\
\midrule
\midrule
$\SO n$ & $-\,\frac{(n-1)}2$ & $-\,\frac 12$ & 
see \cite{Gud-Sak-1}\\
\midrule
$\SU n$ & $-\,\frac{n^2-1}n$ & $-\,\frac{n-1}n$ & 
see \cite{Gud-Sob-1} \\
\midrule
$\Sp n$ & $-\,\frac{2n+1}2$ & $-\,\frac 12$ & 
see \cite{Gud-Mon-Rat-1} \\
\midrule
$\SU n/\SO n$ & $-\,\frac{2(n^2+n-2)}{n}$& $-\,\frac{4(n-1)}{n}$ & 
see \cite{Gud-Sif-Sob-2} \\
\midrule
$\Sp n/\U n$ & $-\,2(n+1)$ & $-\,2$ & 
see \cite{Gud-Sif-Sob-2} \\
\midrule
$\SO{2n}/\U n$ & $-\,2(n-1)$ & $-1$ & 
see \cite{Gud-Sif-Sob-2} \\
\midrule
$\SU{2n}/\Sp n$ & $-\,\frac{2(2n^2-n-1)}{n}$ & $-\,\frac{2(n-1)}{n}$ & 
see \cite{Gud-Sif-Sob-2} \\
\midrule
$\SO{m+n}/\SO m\times\SO n$ & $-(m+n)$ & $-2$ & 
see \cite{Gha-Gud-4} \\
\midrule
$\U{m+n}/\U m\times\U n$ & $-2(m+n)$ & $-2$ & 
see \cite{Gha-Gud-5} \\
\midrule
$\Sp{m+n}/\Sp m\times\Sp n$ & $-2(m+n)$ & $-1$ & 
see \cite{Gha-Gud-5} \\
\midrule
\midrule
\end{tabular}	
}
\bigskip
\caption{Eigenfunctions on the classical compact  irreducible Riemannian symmetric spaces.}
%\label{table-eigenfunctions}	
\end{table}
\renewcommand{\arraystretch}{1}

\newpage

\section{A New Method for Finding Minimal Submanifolds}\label{newmethod}
In the following, we explain the new method introduced by Gudmundsson and Munn in \cite{Gud-Mun-1}, which allows us to construct compact minimal submanifolds with the help of eigenfunctions. We first state two important results by Eells and Sampson (\cite{Eel-Sam}) and Baird and Eells (\cite{Bai-Eel}) respectively. They yield an important connection between minimal submanifolds and weakly conformal and harmonic maps. Finally, we will state Theorem \ref{GM} by Gudmundsson and Munn. The goal of this thesis is to apply it to all known eigenfunctions which satisfy its conditions.

We need the following definitions.
\begin{definition}{\rm \cite{Bai-Woo-book}}
    Let $(M^m,g)$ and $(N^n,h)$ be Riemannian manifolds with $m\geq n.$ Let $\phi:M^m\rightarrow N^n$ be a differentiable map. $\phi$ is said to be a submersion if for every $p\in M,$ the differential $\diff\phi_p:T_pM\rightarrow T_pN$ is surjective.
\end{definition}

\begin{definition}{\rm \cite{Bai-Woo-book}}
Let $\phi:(M^m,g)\rightarrow(N^n,h)$ be a smooth map between Riemannian manifolds. A point $x\in M$ is said to be {\it critical} if $\rank \diff\phi_x<\min\{m,n\}.$ The image of a critical point is called a critical value. 
\end{definition}

\begin{definition}{\rm \cite{Bai-Woo-book}}
    Let $(M,g)$ be a Riemannian manifold and $\phi:M\rightarrow\cn$ be a complex-valued function on $M.$ We say that $z\in\phi(M)$ is a {\it regular} value if $\nabla\phi\neq0$ along $\phi^{-1}(\{z\}).$
\end{definition}
In Chapter 1, we have already mentioned the following result by Eells and Sampson from 1964 in their paper \cite{Eel-Sam}. Here we are citing it from Chapter 3 of the book \cite{Bai-Woo-book}.

\begin{theorem}{\rm \cite{Bai-Woo-book}}
    A weakly conformal map from a Riemannian manifold of dimension two is harmonic if and only if its image is minimal at regular points.
\end{theorem}

In 1981, Baird and Eells showed in their work \cite{Bai-Eel} that harmonic morphisms are useful instruments in the study of minimal submanifolds. 

\begin{theorem}{\rm \cite{Bai-Eel}}
Let $\phi:(M,g)\rightarrow\cn$ be a complex-valued harmonic morphism from a Riemannian manifold. Then every regular fibre of $\phi$ is a minimal submanifold of $(M,g)$ of codimension two.
\end{theorem}

The next result, by Gudmundsson and Munn from 2023, now brings eigenfunctions into play. 

\begin{theorem}{\rm \cite{Gud-Mun-1}}\label{GM}
    Let $\phi:(M,g)\rightarrow\mathbb{C}$ be a complex-valued eigenfunction on a Riemannian manifold, such that $0\in\phi(M)$ is a regular value for $\phi.$ Then the fibre $\phi^{-1}(\{0\})$ is a minimal submanifold of $M$ of codimension two.
\end{theorem}
This provides us with a new method to find minimal submanifolds. In the following chapters, we will look at the classical compact symmetric spaces and check if the known eigenfunctions are regular over $0\in\cn.$ If so, then Theorem \ref{GM} allows us to construct families of compact minimal submanifolds.

A theorem from Riedler and Siffert's paper \cite{Rie-Sif} supplies us with a straightforward way of checking whether an eigenfunction attains the required value $0\in\cn$.
\begin{theorem}{\rm \cite{Rie-Sif}}
    Let $(M,g)$ be a compact and connected Riemannian manifold and let $\phi:M\rightarrow\cn$ be a $(\lambda,\mu)$-eigenfunction not identically zero. Then the following are equivalent.
    \begin{enumerate}
        \item $\lambda=\mu$.
        \item $\rvert\phi\rvert^2$ is constant.
        \item $\phi(x)\neq0$ for all $x\in M.$
    \end{enumerate}
\end{theorem}
This immediately yields a highly useful corollary.
\begin{corollary}\label{phi0}
    If $\phi:M\rightarrow\cn$ is a complex-valued $(\lambda,\mu)$-eigenfunction on a compact and connected Riemannian manifold $(M,g)$ such that $\lambda\neq\mu,$ then there exists $x\in M$ such that $\phi(x)=0$.
\end{corollary}

Lastly, we state another useful lemma by Milnor (1968).
\begin{lemma}\label{milnor}{\bf \cite{Milnor}}
    Let $V$ be a real or complex algebraic set defined by a single polynomial equation $f(x)=0,$ where $f$ is irreducible. In the real case make the additional hypothesis that $V$ contains a regular point of $f.$ Then every polynomial which vanishes on $V$ is a multiple of $f.$
\end{lemma}

%%%%%%%%%%%%%%%%%%%%%%%%%%%%%%%%%%%%%%%%%%%%

\chapter{The Special Orthogonal Group $\SO{n}$}

We now start with the more practical part of the thesis. 
In Section \ref{eigenfunctionsson} we define eigenfunctions on the special orthogonal group $\SO{n}.$ In Theorem \ref{eigenfamilySOn} we will introduce a family of eigenfunctions on $\SO{n},$ which stems from Gudmundsson and Sakovich's paper \cite{Gud-Sak-1}. 
We take a closer look at a particular function in Example \ref{complexxij}.

In Section \ref{cmsson}, we apply Theorem \ref{GM} to the above mentioned eigenfunctions. We first give an easy example, which generalises an example 
given by Gudmundsson and Munn \cite{Gud-Mun-1}. In Theorem \ref{familySOn}, we show that all the functions defined in Theorem \ref{eigenfamilySOn} can be used to construct {\it compact} minimal submanifolds.

\section{Eigenfunctions on $\SO{n}$}\label{eigenfunctionsson}
We now look at eigenfunctions on $\SO{n}.$ Lemma \ref{xijeigenfunctions} and Theorem \ref{eigenfamilySOn} both stem from Gudmundsson and Sakovich's paper \cite{Gud-Sak-1}. Example \ref{complexxij} gives us a first example of an eigenfunction and illustrates that they are not necessarily complicated.

For the construction of eigenfunctions on $\SO{n}$ we need the following Lemma, which shows how the Laplace-Beltrami and conformality operators act on the coordinate functions. 
\begin{lemma}{\rm \cite{Gud-Sak-1}}\label{xijeigenfunctions}
    For $1\leq j,\alpha\leq n$, let $x_{j\alpha}:\SO n\rightarrow\rn$ be the real-valued coordinate function given by 
    $$x_{j\alpha}:x\mapsto e_j\cdot x\cdot e_\alpha^t,$$ where $\{e_1,\dots,e_n\}$ is the canonical basis for $\rn^n.$ Then the following relations hold.
    $$\tau(x_{j\alpha})=-\frac{n-1}{2}\cdot x_{j\alpha},$$
    $$\kappa(x_{j\alpha},x_{k\beta})=-\frac{1}{2}(x_{j\beta}x_{k\alpha}-\delta_{\alpha\beta}\cdot\delta_{jk}).$$
\end{lemma}
From this we can derive the following eigenfamily. 
\begin{theorem}{\rm \cite{Gud-Sak-1}}\label{eigenfamilySOn}
    Let $V$ be a maximal isotropic subspace of $\cn^n$ and $p\in\cn^n$ be a non-zero element. Then the set $$\mathcal{E}_V(p)=\left\{\phi_a:\SO{n}\rightarrow\cn \ \rvert \ \phi_a(x)=\trace(p^tax^t), \ a\in V\right\}$$ of complex-valued functions is an eigenfamily on $\SO n$.
\end{theorem}

In the following example we see Theorem \ref{eigenfamilySOn} in action.

\begin{example}\label{complexxij}
    From Theorem \ref{eigenfamilySOn}, we obtain eigenfunctions on $\SO{n}$ of the form 
$$x_{j\alpha}+ix_{j\beta}:\SO n\rightarrow\cn.$$ Indeed, let $p=e_j,$ and $a=e_\alpha+i\cdot e_\beta.$ Here, $e_j$ denotes the standard unit vector of $\cn^n$ with zeros everywhere except in the $j$-th entry. Then $p,a\in\cn^n$ clearly satisfy the conditions of Theorem \ref{eigenfamilySOn}, and hence the function $\trace(p^tax^t)$ is an eigenfunction on $\SO{n}.$ Further, \begin{eqnarray*}
    \trace(p^tax^t)&=&\sum_{s=1}^n\sum_{t=1}^n p_sa_tx_{st}\\
    &=&x_{j\alpha}+ix_{j\beta}.
\end{eqnarray*}

In general, it can also be shown directly through Lemma \ref{xijeigenfunctions} that if the indices $j,\alpha,k,\beta$ satisfy either  $j\neq k$ and $\alpha=\beta,$ or $j=k$ and $\alpha\neq \beta,$ then  $$x_{j\alpha}+ix_{k\beta}:\SO n\rightarrow\cn$$ is an eigenfunction on $\SO{n}.$ Due to linearity, $x_{j\alpha}+ix_{k\beta}$ is eigen with respect to the Laplace-Beltrami operator $\tau.$ Furthermore,
\begin{eqnarray*}
\kappa(x_{j\alpha}+ix_{k\beta},x_{j\alpha}+ix_{k\beta})&=&\kappa(x_{j\alpha},x_{j\alpha})+2i\cdot\kappa(x_{j\alpha},x_{k\beta})-\kappa(x_{k\beta},x_{k\beta})\\
    &=&-\frac{1}{2}(x_{j\alpha}^2-1)-i\cdot(x_{j\beta}x_{k\alpha}-\delta_{\alpha\beta}\cdot\delta_{jk})+\frac{1}{2}(x_{k\beta}^2-1)\\
    &=&-\frac{1}{2}\cdot(x_{j\alpha}^2+2i\cdot x_{j\beta}x_{k\alpha}-x_{k\beta}^2)\\
    &=&-\frac{1}{2}\cdot(x_{j\alpha}^2+2i\cdot x_{j\alpha}x_{k\beta}-x_{k\beta}^2)\\
    &=&-\frac{1}{2}\cdot(x_{j\alpha}+ix_{k\beta})^2.
\end{eqnarray*}
   In this computation we used our assumptions on the indices $j,\alpha,k,\beta.$
\end{example}

\section{Compact Minimal Submanifolds of $\SO{n}$}\label{cmsson}
We now wish to apply Theorem \ref{GM} to the eigenfunctions on $\SO{n}.$ 
Example \ref{xijgeneral} generalises Example 7.3 from Gudmundsson and Munn's paper  \cite{Gud-Mun-1}. Here we look at the function given in Example \ref{complexxij}. We then give a general statement in Theorem \ref{familySOn}, where we show that we can construct a family of compact minimal submanifolds of $\SO{n}$ with the eigenfunctions from Theorem \ref{eigenfamilySOn}.
\begin{example}\label{xijgeneral}
     Let $n>2$. If for the indices $j,\alpha,k,\beta$ it either holds that $j\neq k$ and $\alpha=\beta,$ or $j=k$ and $\alpha\neq\beta,$ then $x_{j\alpha}+ix_{k\beta}$ is an eigenfunction on $\SO{n},$ as we have shown in Example \ref{complexxij}.
     
     We will now show that the gradient of $$x_{j\alpha}+ix_{k\beta}:\SO n\rightarrow\cn$$ does not vanish. Indeed, for all $1\leq r<s\leq n, $
    $$Y_{rs}(x_{j\alpha})=\frac{1}{\sqrt{2}}\cdot(-x_{js}\cdot\delta_{\alpha r}+x_{jr}\cdot\delta_{\alpha s}).$$
    It follows that for $1\leq r<\alpha<s\leq n,$ 
    $$Y_{\alpha s}(x_{j\alpha})=-\frac{x_{js}}{\sqrt{2}}, \quad Y_{r\alpha}(x_{j\alpha})=\frac{x_{jr}}{\sqrt{2}}.$$
    Similarly, for $1\leq r<\beta<s\leq n,$
    $$Y_{\beta s}(x_{k\beta})=-\frac{x_{ks}}{\sqrt{2}}, \quad Y_{r\beta}(x_{k\alpha})=\frac{x_{kr}}{\sqrt{2}}.$$
    We conclude that if for all $1\leq r<s\leq n$ it holds that $$Y_{rs}(x_{j\alpha}+ix_{k\beta})=0,$$ then the $j$-th row is zero everywhere except possibly $x_{j\alpha},$ and the $k$-th row is zero everywhere except possibly $x_{k\beta}.$ In the case that $j=k$ and $\alpha\neq\beta$ we see that the $j$-th row vanishes entirely. If $j\neq k$ and $\alpha=\beta,$ then the $j$-th and $k$-th rows are multiples of one another. In either case we arrive at a contradiction, since the matrix $x\in\SO n$ is assumed to be of full rank. Hence it cannot be the case that $\nabla(x_{j\alpha}+ix_{k\beta})=\nabla x_{j\alpha}+i\nabla x_{k\beta}=0.$ 
    Theorem \ref{GM} now shows that the fibre $(x_{j\alpha}+ix_{k\beta})^{-1}(\{0\}),$ which satisfies $x_{j\alpha}=x_{k\beta}=0,$ is a minimal submanifold of $\SO n.$ 
\end{example}
We conclude this chapter with an application of Theorem \ref{GM}. Namely we show that the eigenfunctions as given in Theorem \ref{eigenfamilySOn} are regular over $0\in\cn$ if $p$ is not isotropic. As a consequence, they can be used to construct minimal submanifolds of $\SO{n}.$

\begin{theorem}\label{familySOn}
    Let $n>2.$
    Let $a\in\cn^n\backslash\{0\}$ be an element satisfying $(a,a)=0$ and let $p\in\cn^n\backslash\{0\}$ such that $(p,p)\neq0.$ Then the complex-valued eigenfunction $$\phi_{a,p}:\SO{n}\rightarrow\cn,\quad \phi_{a,p}(x)=\trace(p^tax^t)$$ satisfies the following:
    \begin{enumerate}
        \item $\phi_{a,p}$ is a submersion,
        \item $\phi_{a,p}^{-1}(\{0\})$ is a minimal submanifold of $\SO{n}$ of codimension two.
    \end{enumerate}
\end{theorem}

\begin{proof}
     It follows from Theorem \ref{eigenfamilySOn} that $\phi_{a,p}$ is an eigenfunction on $\SO{n}.$ For simplicity,  for all $1\leq t\leq n,$ we will denote the $t$-th column of $x$ by $x_t.$
    A simple calculation shows that for $Y_{rs}\in\so{n},$ $$Y_{rs}(\phi_{a,p})=-a_r\sum_{j=1}^n p_jx_{js}+a_s\sum_{j=1}^n p_jx_{jr}=-a_r(p,x_s)+a_s(p,x_r).$$ 
    Assume towards a contradiction that there exists a matrix $x\in\SO{n}$ such that the gradient at $x$ satisfies $\nabla\phi_{a,p}=0.$
    Then for all $1\leq r<s\leq n,$ it holds that $Y_{rs}(\phi_{a,p})=0,$ i.e. 
 \begin{equation}\label{arpxscn}
     a_r(p,x_s)=a_s(p,x_r).
 \end{equation}
    Since $a\in\cn^n\backslash\{0\}$ satisfies $(a,a)=0,$ there exist at least two distinct indices $1\leq t<m\leq n$ such that $a_t,a_m\neq0.$ To see this, we refer to Lemma \ref{lemmaiso}.
    
    Note that the columns of $x$ form a basis of the vector space $\cn^n$ over the field $\cn$ and thus there exists an index $1\leq l\leq n$ such that $(p,x_l)\neq 0.$ Without loss of generality, we can assume $l\neq t.$
    
    Entering the indices $l$ and $t$ into Equation \ref{arpxscn}, we obtain
    $$a_t(p,x_l)=a_l(p,x_t).$$ The left hand side of the equation is non-zero, hence  $(p,x_t)\neq0.$ Thus, for every $j\neq t,$ it follows from Equation \ref{arpxscn} that $$\frac{(p,x_j)}{(p,x_t)}=\frac{a_j}{a_t}.$$
    From this, we obtain the following result.
    \begin{eqnarray*}
        0&=&(a,a)\\
        &=&\sum_{j=1}^n a_j^2\\
        &=&a_t^2+\sum_{j\neq t}a_t^2\cdot\frac{(p,x_j)^2}{(p,x_t)^2}\\
        &=&\frac{a_t^2}{(p,x_t)^2}\cdot\sum_{j=1}^n (p,x_j)^2.\\
    \end{eqnarray*}
    Since we assume that $a_t\neq0,$ we conclude that $\sum_{j=1}^n (p,x_j)^2=0.$
    Recall that due to the assumption that $x\in\SO{n},$ it holds that any two rows of $x$ are orthogonal, i.e. $$\sum_{j=1}^n x_{sj}x_{tj}=\delta_{st}.$$
    It follows that \begin{eqnarray*}
        0&=&\sum_{j=1}^n (p,x_j)^2\\
        &=&\sum_{j=1}^n \left(\sum_{s=1}^n p_sx_{sj}\cdot\sum_{t=1}^n p_tx_{tj}\right)\\
        &=&\sum_{s,t=1}^n\left( p_sp_t\sum_{j=1}^n x_{sj}x_{tj}\right)\\
        &=&\sum_{s,t=1}^n p_sp_t\cdot\delta_{st}\\
        &=&\sum_{s=1}^n p_s^2.
    \end{eqnarray*}
    This contradicts our assumption that $(p,p)\neq0.$

Thus, we conclude that the gradient $\nabla\phi_{a,p}$ is non-vanishing.  An easy computation using Lemma \ref{xijeigenfunctions} shows that the eigenvalues of the functions $$\phi_{a,p}=\trace(p^tax^t)$$ satisfy $$\lambda=-\frac{n-1}{2}\quad \textrm{and} \quad \mu=-\frac{1}{2}.$$ If $n>2$, then clearly $\lambda\neq\mu$ and by Corollary \ref{phi0}, $0$ lies in the image of $\phi_{a,p}$.
 In particular, $0$ is a regular value.
The last statement now follows from Theorem \ref{GM}.
\end{proof}

\begin{example}
    In Theorem \ref{familySOn} it is important that we assume $(p,p)\neq0.$ Consider the setting $n=3,$ $a=p=(1,i,0).$ Then the matrix $$x=\begin{pmatrix}
        \cos\theta&\sin\theta&0\\
        -\sin\theta&\cos\theta&0\\
        0&0&1
    \end{pmatrix}\in\SO{3}$$ clearly satisfies $\phi_{a,p}(x)=0,$ and further $\nabla(\phi_{a,p})=0$ at $x.$ Hence $\phi_{a,p}$ is not regular over 0. 
\end{example}

In Example \ref{complexxij} we have remarked that the eigenfunctions $x_{j\alpha}+ix_{k\beta}$ belong to the eigenfamily $\mathcal{E}_V(p)$ which we have defined in Theorem \ref{eigenfamilySOn}. This means in particular that Theorem \ref{familySOn} shows directly that the functions $x_{j\alpha}+ix_{k\beta}$ are regular over $0,$ which we have shown explicitly in Example \ref{xijgeneral} through a computation.
\begin{remark}
    It is clear that if $a=C\cdot b$ and $p=D\cdot q$ for some constants $C,D\in\cn,$ then $\phi_{a,p}(x)=\trace(p^tax^t)$ and $\phi_{b,q}$ induce the same minimal submanifolds. Indeed, assume that $x\in\phi_{a,p}^{-1}(\{0\}).$ Then $$0=\trace(p^tax^t)=\trace(CD\cdot q^tbx^t)=CD\cdot\trace(q^tbx^t)$$ and thus $x\in\phi_{b,q}^{-1}(\{0\}).$ 
\end{remark}
The claim formulated in the above remark should be intuitively clear. We can however also make a converse statement using Milnor's Lemma \ref{milnor}.
\begin{proposition}
    Let $n>2,$ and let $a,p,b,q\in\cn^n\backslash\{0\}$ satisfy $(a,a)=(b,b)=0$ and $(p,p)\neq0, (q,q)\neq0.$ Let $\phi_{a,p}=\trace(p^tax^t).$
    Then $$\phi_{a,p}^{-1}(\{0\})=\phi_{b,q}^{-1}(\{0\})$$ if and only if $[a]=[b]$ and $[p]=[q]$ in $\cn P^{n-1}$.
\end{proposition}
\begin{proof}
    Since the polynomial $\phi_{a,p}(x)$ with complex coefficients is of degree one, it is irreducible. We have shown in the proof of Theorem \ref{familySOn} that $\phi_{a,p}(x)$ is regular everywhere, and consequently $\phi_{a,p}^{-1}(\{0\})$ contains a regular point. By Lemma \ref{milnor}, it now follows that $\phi_{a,p}$ and $\phi_{b,q}$ are multiples of one another if and only if $\phi_{a,p}^{-1}(\{0\})=\phi_{b,q}^{-1}(\{0\})$.
\end{proof}

\chapter{The Special Unitary Group $\SU{n}$}
In this chapter we aim to find compact minimal submanifolds of the special unitary group $\SU{n}.$ In Section \ref{efsun} we define a family of eigenfunctions on $\SU{n},$ which were found by Gudmundsson and Sobak in their paper \cite{Gud-Sob-1}. Thereafter, we study compact minimal submanifolds of $\SU{n}$ in Section \ref{cmssun}.
Theorem \ref{thminsubsun} shows that those eigenfunctions are submersive and can hence be used to construct minimal submanifolds. We will then provide two more examples of {\it compact} minimal submanifolds of the special unitary group $\SU{n}$.

\section{Eigenfunctions on $\SU{n}$}\label{efsun}
We first look at the coordinate functions in Lemma \ref{coordsun}. Theorem \ref{thsu} then provides us with a family of eigenfunctions on $\SU{n}.$

\begin{lemma}{\rm \cite{Gud-Sif-Sob-2}}\label{coordsun}
For $1\leq j,\alpha\leq n$, let $z_{j\alpha}:\SU n\rightarrow\cn$ be the complex-valued coordinate function given by 
    $$z_{j\alpha}:x\mapsto e_j\cdot z\cdot e_\alpha^t,$$ where $\{e_1,\dots,e_n\}$ is the canonical basis for $\cn^n.$ Then the following relations hold.
    \begin{eqnarray*}
        \tau(z_{j\alpha})&=&-\frac{n^2-1}{n}\cdot z_{j\alpha},\\
        \kappa(z_{j\alpha},z_{k\beta})&=&-z_{j\beta}z_{k\alpha}+\frac{1}{n}\cdot z_{j\alpha}z_{k\beta}.
    \end{eqnarray*}
\end{lemma}

\begin{theorem}{\rm \cite{Gud-Sob-1}}\label{thsu}
    For $a,v\in\cn^n\backslash\{0\},$ let the complex-valued function $\phi_{a,v}:\SU{n}\rightarrow\cn$ be given by $$\phi_{a,v}(z)=\trace(a^tvz^t)=\sum_{j,\alpha=1}^n a_jv_\alpha z_{j\alpha}.$$ Then $\phi_{a,v}$ is an eigenfunction on $\SU{n}$ and for the tension field $\tau$ and conformality operator $\kappa$ on $\SU{n}$ we obtain
$$\tau(\phi_{a,v})=-\frac{n^2-1}{n}\cdot\phi_{a,v} \ \ \textrm{and} \ \ 
        \kappa(\phi_{a,v},\phi_{a,v})=-\frac{n-1}{n}\cdot\phi_{a,v}^2.
$$

\end{theorem}

\section{Compact Minimal Submanifolds of $\SU{n}$}\label{cmssun}
In this section, we will apply Theorem \ref{GM} to the eigenfunctions we have stated in Theorem \ref{thsu}. As we will show, they can be used to construct compact minimal submanifolds of $\SU{n}.$ We then state two more examples of compact minimal submanifolds of the special unitary group. In Example \ref{clifford} we look at the well-known Clifford torus in $S^3\cong\SU{2},$ and Example \ref{Gudsvevil} describes a minimal submanifold of $\SU{n}$ that Gudmundsson, Svensson and Ville constructed in their paper \cite{Gud-Sve-Vil-1}.
\begin{theorem}\label{thminsubsun}
Let $n\geq2$. Let $a,v\in\cn^n\backslash\{0\}.$ Then the complex-valued eigenfunction
$$\phi_{a,v}(z)=\trace(a^tvz^t)$$ satisfies the following: \begin{enumerate}
        \item $\phi_{a,v}$ is a submersion,
        \item $\phi_{a,v}^{-1}(\{0\})$ is a compact minimal submanifold of $\SU{n}$ of codimension two.
    \end{enumerate} 
\end{theorem}
\begin{proof}
First note that by Theorem \ref{thsu}, $\phi_{a,v}$ is an eigenfunction on $\SU{n}.$
    For $1\leq r<s\leq n,$ the following holds.  \begin{eqnarray}
        Y_{rs}(\phi_{a,v})&=&\frac{1}{\sqrt{2}}\sum_{j=1}^n(-a_jv_rz_{js}+a_jv_sz_{jr})=\frac{1}{\sqrt{2}}(-v_r(a,z_s)+v_s(a,z_r)),\ \ \label{yrssun}\\
        iX_{rs}(\phi_{a,v})&=&\frac{i}{\sqrt{2}}\sum_{j=1}^n(a_jv_rz_{js}+a_jv_sz_{jr})=\frac{i}{\sqrt{2}}(v_r(a,z_s)+v_s(a,z_r)),\label{xrssun}\\
        iD_{rs}(\phi_{a,v})&=&\frac{i}{\sqrt{2}}\cdot\sum_{j=1}^n (a_jv_rz_{jr}-a_jv_sz_{js})=\frac{i}{\sqrt{2}}(v_r(a,z_r)-v_s(a,z_s)).\ \ \ \ \label{drssun}
    \end{eqnarray}
    Assume towards a contradiction that $$Y_{rs}(\phi_{a,v})=0,\  iX_{rs}(\phi_{a,v})=0 \ {\rm  and} \ iD_{rs}(\phi_{a,v})=0$$ for all $1\leq r<s\leq n.$ Then by Equations \ref{yrssun} and \ref{xrssun}, for all such $1\leq r<s\leq n$ we have that

    \begin{eqnarray*}
        0&=&v_r\cdot\sum_{j=1}^n(a_jz_{js})=v_r\cdot(a,z_s),\\
        0&=&v_s\cdot\sum_{j=1}^n(a_jz_{jr})=v_s\cdot(a,z_r).
    \end{eqnarray*}
    Since $v\neq0,$ $v$ has at least one non-zero entry, say $v_t\neq0.$
    We then see that for all indices $k\neq t,$ it holds that $0=(a,z_k).$ 
    Thus by Equation \ref{drssun}, for any index $k\neq t$ we have that $$v_t(a,z_t)=v_k(a,z_k)=0.$$ Since $v_t\neq0$ by assumption, we conclude that $(a,z_t)=0$ as well. 
    
   This means that the rows of $z$ are linearly dependent, contradicting the assumption that $z$ is an element of $\SU{n}$. Hence $\phi_{a,v}$ is submersive everywhere. Theorem \ref{thsu} shows that for $n>1$ the eigenvalues $\lambda$ and $\mu$ of $\phi_{a,v}$ are distinct. Corollary \ref{phi0} now shows that $0$ lies in the image of $\phi_{a,v}$.
   In particular, $\phi_{a,v}$ is regular over $0\in\cn.$ 
The last statement now follows from Theorem \ref{GM}.
\end{proof}

\begin{proposition}
    Let $n>1$ and for $a,v\in\cn^n\backslash\{0\},$ define the complex valued function $\phi_{a,v}$ on $\SU n$ by $\phi_{a,v}(z)=\trace(a^tvz^t)$. Then $$\phi_{a,v}^{-1}(\{0\})=\phi_{b,w}^{-1}(\{0\})$$ if and only if $[a]=[b]$ and $[v]=[w]$ in $\cn P^{n-1}.$
\end{proposition}
\begin{proof}
    This is a direct consequence of Lemma \ref{milnor}.
\end{proof}

For the interested reader, we provide two more examples of compact minimal submanifolds of $\SU{n}.$
\begin{example}\label{clifford}
    Let now $n=2.$ It is well known that $S^3\cong\SU{2}$ via the map $$(z,w)\mapsto\begin{pmatrix}
        z&w\\
        -\Bar{w}&\Bar{z}
    \end{pmatrix}.
    $$
    A well-known minimal submanifold of $S^3$ is the Clifford torus $$T=\left(\frac{1}{\sqrt{2}}e^{i\alpha},\frac{1}{\sqrt{2}}e^{i\beta}\right).$$
    In 1970, Lawson  famously conjectured in his paper \cite{Lawson} that this is the only compact minimal submanifold of genus $1$ embedded in $S^3$. This result was proven by Brendle in 2013 in \cite{Brendle}.
    There however exist infinitely many immersed tori in $S^3,$ which were constructed by Hsiang and Lawson in \cite{Hsi-Law}.
    As it was shown by the author in \cite{JG}, the torus $T$ remains minimal if we equip $S^3$ with a generalised Berger metric. 
\end{example}

\begin{example}{\rm \cite{Gud-Sve-Vil-1}}\label{Gudsvevil}
    Let $H\in\cn^{n\times n}$ have $n$ different eigenvalues. Then the compact subset $$M=\left\{(z_1,\dots,z_n)\in\SU{n}\ \rvert \ z_1^tH\Bar{z}_2=0\right\}$$ is a minimal submanifold of the special unitary group of codimension two. Here, $z_1,\dots,z_n$ denote the columns of $z.$
\end{example}

\chapter{The Quaternionic Unitary Group $\Sp{n}$ }
We now aim to construct compact minimal submanifolds of the quaternionic unitary group $\Sp{n}.$ In Section \ref{efspnsection}, we will first state eigenfunctions on $\Sp{n},$ which were provided by Gudmundsson and Sakovich in their paper \cite{Gud-Sak-1}. Th main result of Section \ref{cmsspn} is
Theorem \ref{thminsubmspn}, which shows that these functions can be used to construct {\it compact} minimal submanifolds of $\Sp{n}.$ 

\section{Eigenfunctions on $\Sp{n}$}\label{efspnsection}
We first look at the coordinate functions on $\Sp{n}.$ The following Lemma \ref{coordspn} was first given by Gudmundsson and Sakovich in \cite{Gud-Sak-1} and then improved by Gudmundsson, Montaldo and Ratto in \cite{Gud-Mon-Rat-1}. Theorem \ref{efspn} then helps us construct eigenfunctions on $\Sp{n}.$

\begin{lemma}{\rm \cite{Gud-Sak-1}, \cite{Gud-Mon-Rat-1}}\label{coordspn}
For $1\leq j,\alpha\leq n,$ let $z_{j\alpha},w_{j\alpha}:\Sp{n}\rightarrow\cn$ be the complex valued coordinate functions given by $$z_{j\alpha}:g\mapsto e_j\cdot g\cdot e_\alpha^t, \ w_{j\alpha}:g\mapsto e_j\cdot g\cdot e_{n+\alpha}^t.$$ Then the following relations hold.
\begin{eqnarray*}
    \tau(z_{j\alpha})&=&-\frac{2n+1}{2}\cdot z_{j\alpha}, \ \tau(w_{j\alpha})=-\frac{2n+1}{2}\cdot w_{j\alpha},\\
    \kappa(z_{j\alpha},z_{k\beta})&=&-\frac{1}{2}\cdot z_{j\beta}z_{k\alpha}, \ \kappa(w_{j\alpha},w_{k\beta})=-\frac{1}{2}\cdot w_{j\beta}w_{k\alpha},\\
    \kappa(z_{j\alpha},w_{k\beta})&=&-\frac{1}{2}\cdot(w_{j\beta}z_{k\alpha}).
\end{eqnarray*}
    
\end{lemma}
Lemma \ref{coordspn} now allows us to construct eigenfunctions on $\Sp{n}.$ The following eigenfunctions were constructed by Gudmundsson and Sakovich in their paper \cite{Gud-Sak-1}.
\begin{proposition}{\rm \cite{Gud-Sak-1}}\label{efspn}
    Let $a,u,v\in\cn^n.$ Then the function $\phi:\Sp{n}\rightarrow\cn$ given by $$\phi(z+jw)=\trace(a^tvz^t+a^tuw^t)$$ satisfies 
$$
        \tau(\phi)=-\frac{2n+1}{2}\cdot\phi \ \ \textrm{and} \ \ 
        \kappa(\phi,\phi)=-\frac{1}{2}\cdot\phi^2.
$$
\end{proposition}

\section{Compact Minimal Submanifolds of $\Sp{n}$}\label{cmsspn}
We now wish to employ Theorem \ref{GM}. We start with an easy case in Example \ref{spnexample}, and then show with Theorem \ref{thminsubmspn} that in general, the eigenfunctions given in Theorem \ref{efspn} are regular over $0$ and hence can be used to construct compact minimal submanifolds of $\Sp{n}.$ In the following, we use Proposition \ref{liealgebraspn}, which provides us with a basis for the Lie algebra $\sp{n}$ of $\Sp{n}.$ 

\begin{example}\label{spnexample}
    Let $v=(1,0,\dots,0), u=0$ and $a\in\cn^n\backslash\{0\}.$ Then according to Proposition \ref{efspn},$$\phi_v:\Sp{n}\rightarrow\cn,$$ given by $$\phi_v(z+jw)=\trace(a^tvz^t)$$ is an eigenfunction. We will show that it is regular over $0\in\cn.$
   
For our choices of $a,u$ and $v,$ a simple calculation shows that $\phi_a $ now simplifies to $$\phi_a (z+jw)=\sum_{j=1}^na_jz_{j1}=(a,z_1).$$
     From \cite{Gud-Mun-1} we have the following identities for all $1<s\leq n$:
    \begin{eqnarray*}
        X_{1s}^{a}(\phi_a )&=&\frac{i}{\sqrt{2}}\cdot\sum_{j=1}^n a_jz_{js}=\frac{i}{\sqrt{2}}\cdot(a,z_s),\\
        X_{1s}^{b}(\phi_a )&=&\frac{i}{\sqrt{2}}\cdot\sum_{j=1}^n a_jw_{js}=\frac{i}{\sqrt{2}}\cdot(a,w_s),\\
        D_{1}^{a}(\phi_a )&=&\frac{i}{\sqrt{2}}\cdot\sum_{j=1}^n a_jz_{j1}=\frac{i}{\sqrt{2}}\cdot(a,z_1).\\
        D_{1}^{c}(\phi_a )&=&\frac{-1}{\sqrt{2}}\cdot\sum_{j=1}^n a_jw_{j1}=\frac{-1}{\sqrt{2}}\cdot(a,w_1).\\
    \end{eqnarray*}
     Assume now that $\phi_a (z+jw)=0,$ and that for every $X\in\mathcal{B}_{\sp{n}}, \ X(\phi_a )=0.$ Then by the above, for all $1\leq t\leq n,$
     $$0=\sum_{j=1}^n a_jz_{jt} \quad  \textrm{ and} \quad 0=\sum_{j=1}^n a_jw_{jt}.$$ In particular this means that the rows of $z$ and $w$ respectively are linearly dependent. This contradicts the assumption that $q=z+jw\in\Sp{n}.$ Hence, $\phi_a$ must be regular over $0.$ Due to Theorem \ref{GM}, $\phi_v ^{–1}(\{0\})$ is a minimal submanifold of $\Sp{n}.$ \end{example}
\begin{theorem}\label{thminsubmspn}
    Let $a\in\cn^n\backslash\{0\},$ and $u,v\in\cn^n$ such that at least one of the vectors $u,v$ is non-zero. Then, the complex-valued eigenfunction 
    $\phi_{a,v,u}:\Sp{n}\rightarrow\cn$ given by $$\phi_{a,v,u}(z+jw)=\trace(a^tvz^t+a^tuw^t)$$ satisfies the following: \begin{enumerate}
        \item $\phi_{a,v,u}$ is a submersion,
        \item $\phi_{a,v,u}^{-1}(\{0\})$ is a compact minimal submanifold of $\Sp{n}$ of codimension two.
    \end{enumerate} 
\end{theorem}
\begin{proof}
     Gudmundsson and Munn showed in Section 9 of \cite{Gud-Mun-1} that the elements of the basis $\mathcal{B}_\sp{n}$ as specified in Proposition \ref{liealgebraspn} act on the coordinate function $z_{j\alpha}:\Sp{n}\rightarrow\cn$ as follows.
    \begin{eqnarray*}
        X_{rs}^{a}(z_{j\alpha})&=&\frac{i}{\sqrt{2}}(z_{jr}\delta_{s\alpha}+z_{js}\delta_{r\alpha}),\\
          X_{rs}^{b}(z_{j\alpha})&=&\frac{i}{\sqrt{2}}(w_{jr}\delta_{s\alpha}+w_{js}\delta_{r\alpha}),\\
        D_t^{a}(z_{j\alpha})&=&\frac{i}{\sqrt{2}}z_{jt}\delta_{\alpha t},\\
        D_t^{c}(z_{j\alpha})&=&\frac{-1}{\sqrt{2}}w_{jt}\delta_{\alpha t}.\\
    \end{eqnarray*}
    Further, we computed that
    \begin{eqnarray*}
          X_{rs}^{c}(z_{j\alpha})&=&\frac{-1}{\sqrt{2}}(w_{jr}\delta_{s\alpha}+w_{js}\delta_{r\alpha}),\\
        D_t^{b}(z_{j\alpha})&=&\frac{i}{\sqrt{2}}w_{jt}\delta_{\alpha t}.
    \end{eqnarray*}
    We found how the elements of the basis $\mathcal{B}_\sp{n}$ act on the coordinate functions $w_{j\alpha}:\Sp{n}\rightarrow\cn.$ 
    \begin{eqnarray*}
        X_{rs}^{a}(w_{j\alpha})&=&\frac{-i}{\sqrt{2}}(w_{jr}\delta_{s\alpha}+w_{js}\delta_{r\alpha}),\\
        X_{rs}^{b}(w_{j\alpha})&=&\frac{i}{\sqrt{2}}(z_{jr}\delta_{s\alpha}+z_{js}\delta_{r\alpha}),\\
        X_{rs}^{c}(w_{j\alpha})&=&\frac{1}{\sqrt{2}}(z_{jr}\delta_{s\alpha}+z_{js}\delta_{r\alpha}),\\
        D_t^{a}(w_{j\alpha})&=&\frac{-i}{\sqrt{2}}w_{jt}\delta_{\alpha t},\\
        D_t^{b}(w_{j\alpha})&=&\frac{i}{\sqrt{2}}z_{jt}\delta_{\alpha t},\\
        D_t^{c}(w_{j\alpha})&=&\frac{1}{\sqrt{2}}z_{jt}\delta_{\alpha t}.\\
    \end{eqnarray*}
    The eigenfunction $\phi_{a,v,u}:\Sp{n}\rightarrow\cn$ given by $$\phi_{a,v,u}(z+jw)=\trace(a^tvz^t+a^tuw^t),$$ can be rewritten as follows.
    $$\phi_{a,v,u}(z+jw)=\sum_{j,\alpha=1}^{n}a_j(v_\alpha z_{j\alpha}+u_\alpha w_{j\alpha}).$$ Then for all $1\leq r<s\leq n$ and $1\leq t\leq n,$ the basis elements of the Lie algebra $\sp{n}$ act on $\phi_{a,v,u}$ as follows. 
    \begin{eqnarray}
        X_{rs}^{a}(\phi_{a,v,u})&=&\frac{i}{\sqrt{2}}\sum_j a_j(v_sz_{jr}+v_rz_{js}-u_sw_{jr}-u_rw_{js}),\label{eqxrsaspn}\\
        X_{rs}^{b}(\phi_{a,v,u})&=&\frac{i}{\sqrt{2}}\sum_j a_j(v_sw_{jr}+v_rw_{js}+u_sz_{jr}+u_rz_{js}),\label{eqxrsbspn}\\
        X_{rs}^{c}(\phi_{a,v,u})&=&\frac{1}{\sqrt{2}}\sum_j a_j(-v_sw_{jr}-v_rw_{js}+u_sz_{jr}+u_rz_{js}),\label{eqxrscspn}\\
        D_{t}^{a}(\phi_{a,v,u})&=&\frac{i}{\sqrt{2}}\sum_j a_j(v_tz_{jt}-u_tw_{jt}),\label{eqdta}\\
        D_{t}^{b}(\phi_{a,v,u})&=&\frac{i}{\sqrt{2}}\sum_j a_j(v_tw_{jt}+u_tz_{jt}),\label{eqdtb}\\
        D_{t}^{c}(\phi_{a,v,u})&=&\frac{1}{\sqrt{2}}\sum_j a_j(-v_tw_{jt}+u_tz_{jt}).\label{eqdtc}
    \end{eqnarray}
    Let us assume that the complex-valued vectors $a,v$ are non-zero. The case that $u\neq0$ is done in the same way. In particular, there exists $1\leq k\leq n$ such that the $k$-th entry of v is non-zero, i.e. $v_k\neq 0.$ 
    Assume towards a contradiction that for every $X\in\mathcal{B}_\sp{n},$ it holds that $X(\phi_{a,v,u})=0.$ Then Equations \ref{eqdtb} and \ref{eqdtc} imply that for every $1\leq t\leq n$ we have that \begin{eqnarray*}
        0&=&v_t(a,w_t),\\
      %  0&=&u_t(a,z_t).
    \end{eqnarray*}
    Since $v_k\neq0,$ we deduce that $(a,w_k)=0.$ Equation \ref{eqdta} now implies that $$v_k(a,z_k)=u_k(a,w_k)=0.$$ This shows that $(a,z_k)=0.$
    Combining Equations \ref{eqxrsbspn} and \ref{eqxrscspn}, we see that for every $t\neq k,$ the following holds.
    \begin{eqnarray*}
        0&=&v_k(a,w_t).\\
    \end{eqnarray*}
    Hence, for every $1\leq t\leq n,$ we have that $(a,w_t)=0.$
    Equation \ref{eqxrsaspn} now implies that for $t\neq k$ $$v_k(a,z_t)=u_k(a,w_t)+u_t(a,w_k)=0,$$ and thus for every $1\leq t\leq n$ we see that $(a,z_t)=0.$
    Since $a\neq0,$ we arrive at a contradiction. Hence, $\nabla\phi_{a,v,u}$ is non-vanishing and in particular $0$ is a regular value. Corollary \ref{phi0} confirms that $0$ indeed lies in the image of $\phi_{a,v,u}$ for any positive integer $n.$ The last statement now follows from Theorem \ref{GM}. 
\end{proof}

%%%%%%%%%%%%%%%%%%%%%%%%%%%%%%%%%%%%%%%%%%

\chapter{The Symmetric Space $\SU{n}/\SO{n}$} 
We will now consider the symmetric space $\SU{n}/\SO{n}$.  In Section \ref{efsusosection}, we consider eigenfunctions from Gudmundsson, Siffert and Sobak's paper \cite{Gud-Sif-Sob-2}. In Section \ref{cmssuso}, we apply 
Theorem \ref{GM} and show that these eigenfunctions can be used to construct a family of compact minimal submanifolds of $\SU{n}/\SO{n}.$ We will first look at a simplified case in Example \ref{ex1suso}, whereas Theorem \ref{thsuso} proves the general case. 

\section{Eigenfunctions on $\SU{n}/\SO{n}$}\label{efsusosection}
We will first show that $\SU{n}/\SO{n}$ is a symmetric space, which helps us to decompose the Lie algebra $\su{n}$ of $\SU{n}$ in a suitable way.  
Thereafter, we look at eigenfunctions on the symmetric space $\SU{n}/\SO{n}$, which are provided in
Proposition \ref{efsuso}.

\begin{example}{\rm \cite{Lind}}
    Since the special orthogonal group $\SO{n}$ of real matrices is a subgroup of the special unitary group $\SU{n}$ of complex matrices,  the quotient space $\SU{n}/\SO{n}$ is defined in a natural way.
    Further, the mapping $\sigma:\SU{n}\rightarrow\SU{n}$ defined by $$\sigma(z)=\Bar{z}$$ is an involutive automorphism of $\SU{n}.$ 
    Indeed, for all $z,w\in\SU{n},$ we have that
    \begin{enumerate}
        \item $\sigma(zw)=\overline{zw}=\Bar{z}\Bar{w}=\sigma(z)\sigma(w),$
        \item $\sigma$ is bijective,
        \item $\sigma^2(z)=z.$
    \end{enumerate}
    Clearly $\SO{n}$ is fixed by $\sigma.$ 
    It follows from
    Theorem \ref{GKsymspace} that $\SU{n}/\SO{n}$ is indeed a symmetric space.
\end{example}

\begin{remark}
    Recall from Proposition \ref{basissun} that the Lie algebra $\su{n}$ of $\SU{n}$ satisfies
$\su{n}=\so{n}\oplus i\m,$ where $$\m=\left\{X\in\mathbb{R}^{n\times n} \ \rvert \ X=X^t \  \textrm{and} \trace(X)=0\right\}.$$ The subspace $i\m$ is spanned by the elements of $\mathcal{S}_{\m}=\left\{iX_{rs}, \ iD_{rs} \ \rvert \ 1\leq r<s\leq n\right\}.$
It is easy to see that $\so{n}$ and $i\m$ are the subspaces of $\su{n}$ corresponding to the eigenvalues $\pm1$ of $\diff\sigma.$

In light of Oberservation \ref{tangentspacem}, we see that $i\m=T_o(\SU{n}/\SO{n}).$ A similar argument will be used in the following chapters.
\end{remark}
The following remark will help us constructing $\SO{n}$-invariant functions on $\SU{n}.$
\begin{remark}
Define the function $\Phi:\SU{n}\rightarrow\SU{n}$ by $$\Phi(z)=zz^t.$$
    For $x\in\SO{n},$ we have by definition that $\Phi(x)=xx^t=I.$ Consequently, $\Phi$ is $\SO{n}$-invariant, i.e. for all $z\in\SU{n}$ and  
$x\in\SO{n},$ $$\Phi(zx)=zxx^tz^t=zz^t=\Phi(z).$$
\end{remark}

\begin{proposition}{\rm \cite{Gud-Sif-Sob-2}}\label{efsuso}
    Let the complex symmetric matrix $A$ be given by $A=aa^t$ for some non-zero element $a\in\mathbb{C}^n.$ Consider the function $\hat{\phi}_a:\SU{n}\rightarrow\mathbb{C}$ with
    $$\hat{\phi}_a(z)=\trace(A\Phi(z))=\trace(z^tAz).$$
    Then $\hat{\phi}_a$ is a $\SO{n}$-invariant eigenfunction on $\SU{n},$ and thus induces an eigenfunction $\phi_a$ on $\SU{n}/\SO{n}.$
    The eigenvalues $\lambda$ and $\mu$ satisfy $$\lambda=-\frac{2(n^2+n-2)}{n}\quad \textrm{and}\quad \mu=-\frac{4(n-1)}{n}.$$
\end{proposition}
\begin{remark}
Let $Z$ be an element of the Lie algebra $\su{n}$ of $\SU{n}.$ Then $Z$ acts on the function $\hat{\phi}$ as follows:
\begin{eqnarray*}
    Z(\hat{\phi}_a)&=&\frac{d}{ds}(\hat{\phi}(z\cdot \exp(sZ)))\rvert_{s=0}\\
    &=& \frac{d}{ds}(\trace(\exp(sZ^t)\cdot z^t\cdot A\cdot z\cdot \exp(sZ))\rvert_{s=0}\\
    &=&\trace(Z^t\cdot z^t\cdot A\cdot z+z^t\cdot A\cdot z\cdot Z)\\
    &=&\trace(Z^t\cdot z^t\cdot A\cdot z)+\trace(z^t\cdot A\cdot z\cdot Z)\\
    &=&\trace(Z^t\cdot z^t\cdot A\cdot z)+\trace(Z\cdot z^t\cdot A\cdot z)\\
    &=&\trace(Z^t\cdot z^t\cdot A\cdot z+Z\cdot z^t\cdot A\cdot z)\\
    &=&\trace((Z^t+Z)\cdot z^t\cdot A\cdot z)\\
\end{eqnarray*}
This shows that $\hat{\phi}_a: \SU{n}\rightarrow\mathbb{C}$ is invariant under the action of the subgroup $\SO{n}$ of $\SU{n}$ with Lie algebra $$\so{n}=\left\{X\in\mathbb{R}^{n\times n} \ \rvert \ X+X^t=0\right\}.$$
Hence $\hat{\phi}_a$ induces a complex-valued eigenfunction $\phi_a:\SU{n}/\SO{n}\rightarrow\mathbb{C}$ on the quotient space $\SU{n}/\SO{n}.$
\end{remark} 

\section{Compact Minimal submanifolds of $\SU{n}/\SO{n}$}\label{cmssuso}
The goal of this section is to apply Theorem \ref{GM} to the eigenfunctions found in Proposition \ref{efsuso} in order to construct compact minimal submanifolds. Example \ref{ex1suso} should clarify how we proceed, whereas Theorem \ref{thsuso} gives a general statement.

\begin{example}\label{ex1suso}
We first consider the simplest case. Let $a=(1,0,\dots, 0).$ Then $A$ is the matrix with entries zero everywhere, except $A_{11}=1.$

Then, \begin{eqnarray}
    \hat{\phi}_a(z)=\trace(z^t Az)=(zz^t)_{11}=\sum_{j=1}^n z_{1j}^2.\label{eqsuso2}
\end{eqnarray} 
For $iX_{rs}=\frac{i}{\sqrt{2}}\cdot (E_{rs}+E_{sr})\in i\m,$ we have that
\begin{eqnarray*}
    iX_{rs}(\hat{\phi}_a)&=&\trace((iX_{rs}^t+iX_{rs})\cdot z^t\cdot A\cdot z)\\
    &=&\trace((2iX_{rs})\cdot z^t\cdot A\cdot z)\\
    &=&\trace(A\cdot z\cdot(2iX_{rs})\cdot z^t )\\
    &=&2i\cdot(z\cdot(X_{rs})\cdot z^t )_{11}\\
    &=&2\sqrt{2}i\cdot(z_{1r}z_{1s}).
\end{eqnarray*}

We now wish to show that $0$ is a regular value of ${\phi}_a$. Assume towards a contradiction that there exists $z\in\SU{n}$ satisfying $\hat{\phi}_a(z)=0$ and $\nabla(\hat{\phi}_a)=0$ at $z.$ Then in particular, for all $1\leq r<s\leq n,$ it holds that $iX_{rs}(\hat{\phi}_a)=0.$ The above computation shows that this is the case if and only if for all choices of such $r,s$ it holds that \begin{eqnarray}
z_{1r}\cdot z_{1s}=0.\label{eqsuso1}\end{eqnarray}

Since $z$ has full rank, the first row has at least one non-zero entry, say $z_{1r_0}$ for some $1\leq r_0\leq n$. Equation \ref{eqsuso1} now implies that for all $t\neq r_0,$ we have that $$z_{1t}=0.$$ In other words, $z_{1r_0}$ is the only non-zero entry of the first row. 
Plugging this into the Equation \ref{eqsuso2}, we see that $$0=\hat{\phi}_a(z)=\sum_{j=1}^n z_{1j}^2=z_{1r_0}^2.$$
It now immediately follows that $z_{1r_0}=0,$ a contradiction. We conclude that $0$ is a regular value of $\hat{\phi}_a$ and thus of $\phi_a$.
By Theorem \ref{GM}, the fibre $\phi_a^{-1}(\{0\})$ is a minimal submanifold of $\SU{n}/\SO{n}.$
\end{example}

\begin{theorem}\label{thsuso}
    For $n\geq2$, let $a\in\mathbb{C}^{n}\backslash\{0\}$ and $A=aa^t.$ Then the eigenfunction $\phi_a:\SU{n}/\SO{n}\rightarrow\mathbb{C}$ induced by the $\SO{n}$-invariant eigenfunction $\hat{\phi}_a:\SU{n}\rightarrow\cn$, which is given by
    $$\hat{\phi}_a(z)=\trace(A\Phi(z))=\trace(z^tAz),$$ satisfies the following:
    \begin{enumerate}
        \item $\phi_a$ is a submersion,
        \item the fibre $\phi_a^{-1}(\{0\})$ is a compact minimal submanifold of $\SU{n}/\SO{n}$ of codimension two.
    \end{enumerate}
\end{theorem}
\begin{proof}
Let $a=(a_1,\dots,a_{n})$ be a non-zero vector in $\mathbb{C}^{n}.$ It follows from Proposition \ref{efsuso} that $\phi_a$ is an eigenfunction on 
$\SU{n}/\SO{n}.$

We wish to show that the gradient of $\hat{\phi}_a$ is non-vanishing. Assume towards a contradiction that there exist a critical point $z\in\SU{n}$. Then for all $iX\in i\m,$ we have that $iX(\hat{\phi}_a)=0$ at $z.$

Setting $iX_{rs}(\hat{\phi}_a)=0$ for all $1\leq r<s\leq n,$ we obtain the equations \begin{eqnarray}\label{Xrs}
0&=&\left(\sum_{k=1}^n a_k z_{kr}\right)\left(\sum_{k=1}^n a_k z_{ks}\right). \end{eqnarray}Hence for all $1\leq r<s\leq n,$ either $$0=\sum_{k=1}^n a_k z_{kr}\ \ \textrm{or} \ \ 0=\sum_{k=1}^n a_k z_{ks}.$$ 

Since $z\in\SU{n},$ the columns of $z$ form an orthonormal basis of $\cn^n.$ Therefore, there exists a column $z_t$ of $z$ such that $$\langle z_t,\Bar{a}\rangle=\sum_{k=1}^n a_k z_{kt}\neq0.$$

 By Equation \ref{Xrs}, for any other index $1\leq m\leq n,$ $m\neq t,$ $$0=\sum_{k=1}^n a_k z_{km}.$$ By assumption, $iD_{rs}(\hat{\phi}_a)=0$ for all $1\leq r<s\leq n.$ The latter is equivalent to $$0=\left(\sum_{k=1}^n a_k z_{kr}+\sum_{k=1}^n a_k z_{ks}\right)\left(\sum_{k=1}^n a_k z_{kr}-\sum_{k=1}^n a_k z_{ks}\right).$$ Using our previous findings, we see that
$$0=\left(\sum_{k=1}^n a_k z_{kt}\right)^2,$$ which again implies that $0=\sum_{k=1}^n a_k z_{kt},$ a contradiction.

Hence, there must exists some $X\in i\m$ such that $X(\hat{\phi}_a)\neq0.$ It follows that $\hat{\phi}_a$ is a submersion, and so is $\phi_a.$ Due to Corollary \ref{phi0}, the value $0\in\cn$ indeed lies in the image of $\phi_a$ for all $n\geq2.$
We conclude that $0$ is a regular value of $\phi_a.$
The last statement now follows from Theorem \ref{GM}.\end{proof}
\begin{remark}
    By letting the vector $a\in\cn^n$ vary, this yields a family of minimal submanifolds of $\SU{n}/\SO{n}.$
\end{remark}

%%%%%%%%%%%%%%%%%%%%%%%%%%%%%%%%%%%%
\chapter{The Symmetric Space $\Sp{n}/\U{n}$}
We start this chapter by studying the symmetric space $\Sp{n}/\U{n}$ in Section \ref{prelimspnun}.
Gudmundsson, Siffert and Sobak's paper \cite{Gud-Sif-Sob-2} provides us with eigenfunctions on $\Sp{n}/\U{n}$, which we state in Section \ref{sectionefspnun}. We then show in Section \ref{cmsspnun}, that those eigenfunctions are regular over $0\in\cn$ and hence they can be used to construct compact minimal submanifolds. 

\section{Preliminaries}\label{prelimspnun}
We will see in Proposition \ref{propuninspn}, that the unitary group $\U{n}$ lies in the quaternionic unitary group $\Sp{n}$. In Proposition \ref{liealgspnun} we will show how the Lie algebra $\sp{n}$ of $\Sp{n}$ can be decomposed.

\begin{proposition}{\rm \cite{Hall}}\label{propuninspn}
    The map $\psi:\U{n}\rightarrow\Sp{n}$ given by $$x+iy\mapsto\begin{pmatrix}
        x&y\\
        -y&x
    \end{pmatrix}$$ is an injective homomorphism of Lie groups. Consequently, we may identify $\U{n}$ with its image under $\psi.$
\end{proposition}
\begin{proof}
    We start by showing that $\psi$ is a homomorphism. 
    \begin{eqnarray*}
        \psi((x_1+iy_1)\cdot(x_2+iy_2))&=&\psi(x_1x_2-y_1y_2+i(x_1y_2+y_1x_2))\\
        &=&\begin{pmatrix}
            x_1x_2-y_1y_2 & x_1y_2+y_1x_2\\
            -x_1y_2-y_1x_2 & x_1x_2-y_1y_2
        \end{pmatrix}\\
        &=&\begin{pmatrix}
            x_1&y_1\\
        -y_1&x_1
        \end{pmatrix}\cdot\begin{pmatrix}
            x_2&y_2\\
        -y_2&x_2
        \end{pmatrix}\\
        &=&\psi(x_1+iy_1)\cdot\psi(x_2+iy_2).
    \end{eqnarray*}
    Further, the image $q=\psi(x+iy)$ satisfies $q\Bar{q}^t=I_n$ since $x+iy\in\U{n}.$
    Lastly, $\psi$ is injective. Indeed, $\psi(x+iy)=I_{2n}$ if and only if $x=I_n.$ Hence, $\ker(\psi)=I_n.$
\end{proof}

\begin{remark}{\rm \cite{Lind}}
    We will now show that $\Sp{n}/\U{n}$ is a symmetric space. Indeed, we define the map $\sigma:\Sp{n}\rightarrow\Sp{n}$ by $$\sigma:\begin{pmatrix}
        z&w\\
        -\Bar{w}&\Bar{z}
    \end{pmatrix}\mapsto\begin{pmatrix}
        \Bar{z}&\Bar{w}\\
        -w&z
    \end{pmatrix}.$$ It is easily checked that $\sigma$ is an involutive automorphism. It is also clear that $\sigma$ fixes $\U{n}.$ By Theorem \ref{GKsymspace}, $\Sp{n}/\U{n}$ is a symmetric space.
\end{remark}

\begin{proposition}{\rm \cite{Gud-Sif-Sob-2}}\label{liealgspnun}
 Let $M=\Sp{n}/\U{n}.$ Consider the following decomposition 
    $$\sp{n}=\u{n}\oplus\mathfrak{m}$$ of the Lie algebra $\sp{n}$ of $\Sp{n}.$ An orthonormal basis of $\m$ is given by 
$$ \mathcal{B}_\mathfrak{m}=\left\{ X_{rs}^{a}=\frac{i}{2}\begin{pmatrix}
    X_{rs}&0\\
    0&-X_{rs}
\end{pmatrix}, D_{t}^{a}=\frac{i}{\sqrt{2}}\begin{pmatrix}
    D_t&0\\
    0&-D_t
\end{pmatrix},X_{rs}^{b}=\frac{i}{2}\begin{pmatrix}
    0&X_{rs}\\
    X_{rs}&0
\end{pmatrix},\right.$$ 
$$\left. D_{t}^{b}=\frac{i}{\sqrt{2}}\begin{pmatrix}
    0&D_t\\
    D_t&0
\end{pmatrix}  \ \rvert \ 1\leq r<s\leq n, \ 1\leq t\leq n \right\}. $$ 
\end{proposition}

\section{Eigenfunctions on $\Sp{n}/\U{n}$}\label{sectionefspnun}
In this section, we look at eigenfunctions on $\Sp{n}/\U{n}$. They were found by Gudmundsson, Siffert and Sobak in their work \cite{Gud-Sif-Sob-2}.
\begin{remark}
    We have shown that if $z\in\U{n},$ then $\sigma(z)=z.$ Now define a map $\Phi:\Sp{n}\rightarrow\Sp{n}$ by $$\Phi(q)=qq^t.$$
    It follows that $\Phi$ is $\U{n}$-invariant, in the sense that for all $q\in\Sp{n}$ and $z\in\U{n},$  $$\Phi(qz)=qzz^tq^t=qq^t=\Phi(q).$$
\end{remark}
This justifies the next construction.
\begin{proposition}{\rm \cite{Gud-Sif-Sob-2}}\label{thspnun1}
    Let $a\in\mathbb{C}^{2n}\backslash\{0\},$ and let $A=aa^t.$ Then $$\hat{\phi}_a(q)=\trace(A\cdot\Phi(q))=\trace(q^t Aq)$$ is a $\U{n}$-invariant eigenfunction on $\Sp{n}$ and thus induces an eigenfunction $\phi_a$ on $\Sp{n}/\U{n}.$ The eigenvalues $\lambda$ and $\mu$ of $\phi_a$ satisfy $$\lambda=-2(n+1)\quad\textrm{and}\quad\mu=-2.$$
\end{proposition} 

\begin{remark}
For any $X\in\sp{n},$ \begin{eqnarray*}
    X(\hat{\phi}_a)&=& \frac{d}{ds}(\trace(\exp(sX^t)q^tA\exp(sX)))|_{s=0}\\
    &=&\trace(X^tq^tAq+q^tAqX)\\
    &=&\trace(q^tAq(X^t+X)).
\end{eqnarray*}
Hence, if $X\in\u{n},$ then $X^t+X=0.$ This shows that $\hat{\phi}_a$ is constant on $\U{n}$ and indeed induces an eigenfunction on $\Sp{n}/\U{n}.$\end{remark}

\section{Compact Minimal Submanifolds of $\Sp{n}/\U{n}$}\label{cmsspnun}
In this section we construct minimal submanifolds of $\Sp{n}/\U{n}$ via the eigenfunction $\phi_a.$ Example \ref{ex1spnun} illustrates the procedure in a simple way, while in Theorem \ref{thspnun2} we obtain a family of minimal submanifolds. 
\begin{example}\label{ex1spnun}
We consider the case $a=(1,0,\dots,0)\in\cn^{2n}\backslash\{0\}.$ Then $A=aa^t$ has zero entries everywhere except for $A_{11}=1.$ 
Thus the eigenfunction $\hat{\phi}_a$ satisfies $$\hat{\phi}_a(q)=\sum_{j=1}^{2n} q_{1j}^2.$$
We wish to show that 0 is a regular value. Assume towards a contradiction that for all $X\in\mathcal{B}_\mathfrak{m},$ we have that $X(\hat{\phi}_a)=0.$ 

In particular, $D_t^{a}(\hat{\phi}_a)=0$ and $D_t^{b}(\hat{\phi}_a)=0$ for all $1\leq t\leq n$ if and only if
    \begin{eqnarray}
    0&=&q_{1t}^2-q_{1,n+t}^2, \label{eq-SpnUn-1}\\
    0&=&q_{1t}q_{1,n+t} \label{eq-SpnUn-2}
   \end{eqnarray}
for every such $t.$ It is readily seen from Equation \ref{eq-SpnUn-2} that for all $1\leq t\leq n,$ either $q_{1t}=0$ or $q_{1,n+t}=0.$ But then Equation \ref{eq-SpnUn-1} implies that in either case the other entry must vanish as well. Consequently, the entire first row is zero, which contradicts the fact that $q$ is full rank. We conclude that $0$ is a regular value of $\hat{\phi}_a$ and of $\phi_a$, and further, by Theorem \ref{GM}, $\phi_a^{-1}(\{0\})$ is a minimal submanifold.
\end{example}
\begin{theorem}\label{thspnun2}
    Let $a=(a_1,\dots,a_{2n})\in\mathbb{C}^{2n}\backslash\{0\}$ and $A=aa^t.$ Then the eigenfunction
    $\phi_a:\Sp{n}/\U{n}\rightarrow\mathbb{C}$ induced by the $\U{n}$-invariant eigenfunction $\hat{\phi}_a:\Sp{n}\rightarrow\cn,$ which is
    given by  $$\hat{\phi}_a(q)=\trace(q^t Aq),$$ satisfies the following:
    \begin{enumerate}
        \item $\phi_a$ is a submersion,
        \item $\phi_a^{-1}(\{0\})$ is a compact minimal submanifold of $\Sp{n}/\U{n}$ of codimension two.   
    \end{enumerate}
\end{theorem}
\begin{proof}
Let $a=(a_1,\dots,a_{2n})\in\mathbb{C}^{2n}\backslash\{0\}.$ 
By Theorem \ref{thspnun1}, $\hat{\phi}_a(q)=\trace(q^t Aq)$ with $A=aa^t$, is an eigenfunction on $\Sp{n}.$

We now show that $\nabla(\hat{\phi}_a)\neq0$. Assume towards a contradiction that there exists $q\in\Sp{n},$ such that for all $X\in\mathcal{B}_\m,$ we have that $X(\hat{\phi}_a)=0.$

From setting $D_t^{a}(\hat{\phi}_a)=0$ and $D_t^{b}(\hat{\phi}_a)=0$ for all $1\leq t\leq n$ we obtain the following equations:
\begin{eqnarray*}
0&=&\left(\sum_{k=1}^{n} a_k q_{kt}\right)^2-\left(\sum_{k=1}^{n} a_k q_{k,n+t}\right)^2,\\
0&=&\left(\sum_{k=1}^{n} a_k q_{kt}\right)\cdot\left(\sum_{k=1}^{n} a_k q_{k,n+t}\right).\\
\end{eqnarray*}
From the second equation we learn that for all $1\leq t\leq n$ either $$0=\sum_{k=1}^{n} a_k q_{kt},\quad\textrm{or}\quad 0=\sum_{k=1}^{n} a_k q_{k,n+t}.$$ By plugging this into the first equation, we now see that the respectively other sum must also vanish, i.e. we have that 
    $$0=\sum_{k=1}^{n} a_k q_{kt}, \quad\textrm{and}\quad0=\sum_{k=1}^{n} a_k q_{k,n+t},$$
for all $1\leq t\leq n.$

This however contradicts the fact that the columns of $q$ span $\cn^{2n},$ which means that there exists a column $q_t,$ with $1\leq t\leq 2n$ such that $$\langle q_t,\Bar{a}\rangle=\sum_{k=1}^{n} a_k q_{kt}\neq0.$$
We conclude that $\hat{\phi}_a$ and thus $\phi_a$ are submersions.
In particular, $0$ is a regular value, which, in light of Corollary \ref{phi0}, is attained by $\phi_a$ for any positive integer $n.$ Consequently, by Theorem \ref{GM}, $\phi_a^{-1}(\{0\})$ is a minimal submanifold for any non-zero $a\in\mathbb{C}^{2n}.$\end{proof}

\chapter{The Symmetric Space $\SO{2n}/\U{n}$ }
In this chapter we examine the symmetric space $\SO{2n}/\U{n}.$ Section \ref{so2nunprelim} gives us some useful background information. We will first show in Proposition \ref{embeddingunso2n} how $\U{n}$ is embedded into $\SO{2n},$ and then decompose the Lie algebra $\so{2n}$ of $\SO{2n}.$ In Section \ref{efso2nunsection}, we then turn our attention to the eigenfunctions on $\SO{2n}/\U{n}$ which Gudmundsson, Siffert and Sobak constructed in their paper \cite{Gud-Sif-Sob-2}. In Theorem \ref{so2nunnew}, we show that the conditions on these eigenfunctions can be relaxed. 

In Section \ref{cmsso2nun}, we aim to apply Theorem \ref{GM}. We provide examples and will show that a particular family of eigenfunctions as given in Theorem \ref{so2nunnew} can be used to construct compact minimal submanifolds.

\section{Preliminaries}\label{so2nunprelim}
We will first show that the unitary group $\U{n}$ can be identified with a subgroup of $\SO{2n}.$ 
\begin{proposition}{\rm \cite{Lind}}\label{embeddingunso2n}
The map $$\psi:\U{n}\rightarrow\SO{2n}, \ x+iy\mapsto \begin{pmatrix}
    x&-y\\
    y&x
\end{pmatrix}$$ is an embedding of $\U n$ into $\SO{2n}.$
\end{proposition}
\begin{proof}

The map is well defined. Indeed, since $x+iy\in\U n,$ $$(x+iy)(x^t-iy^t)=xx^t+yy^t+i(yx^t-xy^t)=I_n.$$
Thus, since $x$ and $y$ have real-valued entries, we have that
    $$I_n=xx^t+yy^t\quad{\rm and} \quad
    0=yx^t-xy^t.
$$
As a consequence, it follows that \begin{eqnarray*}
\psi(x+iy)\cdot(\psi(x+iy))^t&=&\begin{pmatrix}
    x&-y\\
    y&x
\end{pmatrix}\cdot \begin{pmatrix}
    x^t&y^t\\
    -y^t&x^t
\end{pmatrix}\\
&=&\begin{pmatrix}
    xx^t+yy^t&xy^t-yx^t\\
    yx^t-xy^t&xx^t+yy^t
\end{pmatrix}\\
&=&\begin{pmatrix}
    I_n&0\\
    0&I_n
\end{pmatrix}\\
&=&I_{2n}.
\end{eqnarray*}
Further, $\det(\psi(x+iy))=1,$ as shown in the following computation. Here we use that for $x+iy\in\U{n},$ $|\det(x+iy)|=1.$
\begin{eqnarray*}
    \det(\psi(x+iy))&=&\det\begin{pmatrix}
        x&-y\\
        y&x
    \end{pmatrix}\\
    &=&-i\cdot\det\begin{pmatrix}
        x&-y\\
        iy&ix
    \end{pmatrix}\\
    &=&-\det\begin{pmatrix}
        x&-iy\\
        x+iy&-(x+iy)
    \end{pmatrix}\\
    &=&-\det\begin{pmatrix}
        x-iy&-iy\\
        0&-(x+iy)
    \end{pmatrix}\\
    &=&\det\begin{pmatrix}
        x-iy&-iy\\
        0&(x+iy)
    \end{pmatrix}\\
    &=&\overline{\det(x+iy)}\cdot\det(x+iy)\\
    &=&|\det(x+iy)|^2=1.
\end{eqnarray*} Thus, $\psi(\U n)\in\SO{2n}.$

Further, $\psi$ is a group homomorphism. For $x+iy,a+ib\in\U n,$ \begin{eqnarray*}
\psi((x+iy)\cdot(a+ib))&=&\begin{pmatrix}
    xa-yb&-(xb+ya)\\
    xb+ya&xa-yb
\end{pmatrix}\\
&=&\begin{pmatrix}
    x&-y\\
    y&x
\end{pmatrix}\cdot \begin{pmatrix}
    a&-b\\
    b&a
\end{pmatrix}\\
&=&\psi(x+iy)\cdot\psi(a+ib).
\end{eqnarray*}
Lastly, $\psi$ is injective. Indeed, $\psi(x+iy)=I_{2n}$ if and only if $y=0$ and $x=I_n.$ Hence, $\ker\psi=I_n.$

We conclude that $\psi$ is a group isomorphism of $\U n$ onto its image and we may henceforth identify $\U n$ with $\psi(\U n)\in\SO{2n}.$
\end{proof}
\begin{remark}{\rm \cite{Lind}}
    We will now show that $\SO{2n}/\U{n}$ is indeed a symmetric space. For this, let $$J_n=\begin{pmatrix}
    0&I_n\\
    -I_n&0
\end{pmatrix}$$ and consider the map $\sigma:\SO{2n}\rightarrow\SO{2n}$ given by $$\sigma(x)=J_nxJ_n^t.$$ If we write $x$ as a block matrix $$x=\begin{pmatrix}
        x_{11}&x_{12}\\
        x_{21}&x_{22}
    \end{pmatrix},$$ we see that $$\sigma(x)=\begin{pmatrix}
        x_{11}&-x_{21}\\
        -x_{12}&x_{22}
    \end{pmatrix}.$$ Hence, if $x\in\U{n}$ then $x$ is fixed by $\sigma.$ Further, we see that $\sigma$ is an ivolutive automorphism since \begin{enumerate}
        \item $\sigma(xy)=J_nxyJ_n^t=J_nxJ_n^t J_nyJ_n^t=\sigma(x)\sigma(y),$
        \item $\sigma$ is bijective,
        \item $\sigma^2(x)=x.$
    \end{enumerate}
    By Theorem \ref{GKsymspace}, $\SO{2n}/\U{n}$ is a symmetric space. 
\end{remark}

Proposition \ref{decompositionso2nun} shows us how to decompose the Lie algebra $\so{2n}$ of $\SO{2n}.$
\begin{proposition}\label{decompositionso2nun}{\rm \cite{Lind}}
    For the orthogonal decomposition $$\so{2n}=\u{n}\oplus\m,$$ where $\u{n}$ is the Lie algebra of $\U{n}\subset\SO{2n},$ a basis of $\m$ is given by $$ \mathcal{B}_\m=\left\{Y_{rs}^{a}=\frac{1}{\sqrt{2}}\begin{pmatrix}
        Y_{rs}&0\\
        0&-Y_{rs}
    \end{pmatrix},Y_{rs}^{b}=\frac{1}{\sqrt{2}}\begin{pmatrix}
        0&Y_{rs}\\
        Y_{rs}&0
    \end{pmatrix}\ \rvert \ 1\leq r<s\leq n\right\}. $$
\end{proposition}
\begin{proof}
    In light of Proposition \ref{embeddingunso2n}, we note that the Lie algebra of $\u{n}$ in $\so{2n}$ is the vector subspace $$\u{n}=\left\{\begin{pmatrix}
        Y&-W\\
        W&Y
    \end{pmatrix}\ \rvert \ W^t-W=0,\ \textrm{and} \ Y^t+Y=0\right\}.$$ It can now be easily verified that the subspace $$\m=\left\{\begin{pmatrix}
        Y&W\\
        W&-Y
    \end{pmatrix}\in\so{2n} \ \rvert \ Y^t+Y=0,\ \textrm{and} \ W^t+W=0\right\}$$ of $\so{2n}$, spanned by $\mathcal{B}_\m$ as given above, is the orthogonal complement of $\u{n}$ in $\so{2n}.$
\end{proof}

\section{Eigenfunctions on $\SO{2n}/\U n$ }\label{efso2nunsection}
In their paper \cite{Gud-Sif-Sob-2}, Gudmundsson, Siffert and Sobak constructed the family of eigenfuntions given in Proposition \ref{so2nunprop}. In Proposition \ref{so2nunnew}, we show that the conditions on the vectors $a,b$ can be relaxed. The proof still works as before. This serves two purposes. We not only obtain more eigenfunctions, but as we will see in Section \ref{cmsso2nun}, the eigenfunctions given in Proposition \ref{so2nunprop} might not always be regular over $0\in\cn$. Further, the computation of the gradient can be very complicated. We  however successfully apply Theorem \ref{GM} to a subset of the newly found eigenfunctions, see Theorem \ref{thmsso2nun}.

\begin{remark}
Notice that for $y\in\U{n},$ $$\Phi(y)=yJ_ny^t=J_n.$$
Hence $\Phi$ is $\U{n}$-invariant, in the sense that for any $y\in\U{n},x\in\SO{2n},$ $$\Phi(xy)=xyJ_ny^tx^t=xJ_nx^t=\Phi(x).$$
\end{remark}

\begin{proposition}{\rm \cite{Gud-Sif-Sob-2}}\label{so2nunprop}
Let $V$ be an isotropic subspace of $\cn^{2n}$ and let $a,b\in V$ be linearly independent. Let $A\in\cn^{2n\times2n}$ be the skew-symmetric matrix $$A=\sum_{j,\alpha=1}^{2n} a_jb_\alpha Y_{j\alpha}.$$ Then $$\hat{\phi}_{a,b}(x)=\trace(A\Phi(x))=\trace(AxJ_nx^t)$$
is a $\U{n}$-invariant eigenfunction on $\SO{2n},$ and hence induces an eigenfunction $\phi_{a,b}$ on the quotient space $\SO{2n}/\U n.$ The eigenvalues $\lambda$ and $\mu$ of $\phi_{a,b}$ satisfy $$\lambda=-2(n-1)\quad \textrm{and}\quad \mu=-1.$$
\end{proposition}
\begin{remark}
    Let $X$ be an elements of the Lie algebra $\so{2n}$ of $\SO{2n}.$ Then $X$ acts on the function $\hat{\phi}_{a,b}$ as follows.
    \begin{eqnarray*}
        X(\hat{\phi}_{a,b})&=&\frac{d}{ds}(\hat{\phi}_{a,b}(x\cdot \exp(sX))\rvert_{s=0}\\
        &=&\trace((XJ_n+J_nX^t)x^tAx).
    \end{eqnarray*}
    We have seen in Remark \ref{decompositionso2nun} that if $X\in\u{n},$ then $X$ is of the form $$X=\begin{pmatrix}
        Y&-W\\
        W&Y
    \end{pmatrix},$$ where $W^t-W=0$ and $Y^t+Y=0$.
    In this case, $$XJ_n+J_nX^t=\begin{pmatrix}
        W-W^t&Y+Y^t\\
        -(Y+Y^t)&W-W^t
    \end{pmatrix}=0.$$
    This demonstrates that $\hat{\phi}_{a,b}$ is $\U{n}$-invariant.
\end{remark}

    The following result is the same as Proposition \ref{so2nunprop} from \cite{Gud-Sif-Sob-2}, except that the conditions on $a,b$ are slightly weakened.

\begin{proposition}\label{so2nunnew}
    Let $a,b\in\cn^{2n}$ be linearly independent vectors satisfying $$(a,a)(b,b)-(a,b)^2=0.$$ Let $A\in\cn^{2n\times2n}$ be the skew-symmetric matrix $$A=\sum_{j,\alpha=1}^{2n} a_jb_\alpha Y_{j\alpha}.$$ Then $$\hat{\phi}_{a,b}(x)=\trace(A\Phi(x))=\trace(AxJ_nx^t)$$
is an $\U{n}$-invariant eigenfunction on $\SO{2n},$ and hence induces an eigenfunction $\phi_{a,b}$ on the quotient space $\SO{2n}/\U n.$
\end{proposition}
\begin{proof}
In this proof we proceed as in the proof of Proposition \ref{so2nunprop}. Let $\Phi:\SO{2n}\rightarrow\SO{2n}$ be the map $$\Phi(x)=x\cdot J_n\cdot x^t.$$ Then, as shown in \cite{Gud-Sif-Sob-2}, the components of $\Phi$ satisfy the following:
\begin{eqnarray*}
    \tau\left(\Phi_{j\alpha}\right)&=&-2(n-1)\cdot\Phi_{j\alpha},\\
    \kappa\left(\Phi_{j\alpha},\Phi_{k\beta}\right)&=&-(\Phi_{j\beta}\Phi_{k\alpha}+\Phi_{jk}\Phi_{\alpha\beta})-(\delta_{k\alpha}\delta_{j\beta}-\delta_{jk}\delta_{\alpha\beta}).
\end{eqnarray*}
Since $A_{j\alpha}=\frac{1}{\sqrt{2}}(a_jb_\alpha-a_\alpha b_j),$ we may now write $$\hat{\phi}_{a,b}(x)=\trace(A\Phi(x))=-\sum_{j,\alpha=1}^{2n} A_{j\alpha}\Phi_{j\alpha}.$$
Using linearity of $\tau,$ we immediately obtain that \begin{eqnarray*}
    \tau(\hat{\phi}_{a,b})&=&-\sum_{j,\alpha=1}^{2n} A_{j\alpha}\tau(\Phi_{j\alpha})\\
    &=&\sum_{j,\alpha=1}^{2n} A_{j\alpha}\cdot2(n-1)\cdot\Phi_{j\alpha}\\
    &=&-2(n-1)\cdot\hat{\phi}_{a,b}.
\end{eqnarray*}
We now calculate $\kappa(\hat{\phi}_{a,b},\hat{\phi}_{a,b}).$
\begin{eqnarray*}
    \kappa(\hat{\phi}_{a,b},\hat{\phi}_{a,b})&=&\sum_{j,\alpha,k,\beta=1}^{2n} A_{j\alpha}A_{k\beta}\cdot\kappa(\Phi_{j\alpha},\Phi_{k\beta})\\
    &=&-\sum_{j,\alpha,k,\beta=1}^{2n} A_{j\alpha}A_{k\beta}\cdot(\Phi_{j\beta}\Phi_{k\alpha}+\Phi_{jk}\Phi_{\alpha\beta})\\
    &&-\sum_{j,\alpha,k,\beta=1}^{2n} A_{j\alpha}A_{k\beta}\cdot(\delta_{k\alpha}\delta_{j\beta}-\delta_{jk}\delta_{\alpha\beta})\\
    &=&-\frac{1}{2}\cdot\sum_{j,\alpha,k,\beta=1}^{2n}(a_jb_\alpha-a_\alpha b_j)(a_kb_\beta-a_\beta b_k)\cdot(\Phi_{j\beta}\Phi_{k\alpha}+\Phi_{jk}\Phi_{\alpha\beta})\\
    &&-\sum_{j,k=1}^{2n}A_{jk}A_{kj}+\sum_{j,k=1}^{2n}A_{jk}^2.
\end{eqnarray*}
The first sum now simplifies to $2\cdot\hat{\phi}_{a,b}^2,$ using the identity $$\hat{\phi}_{a,b}^2=2\cdot\left(\sum_{j,\alpha} a_jb_\alpha\Phi_{j\alpha}\right)^2.$$ Terms of the form $$\left(\sum_{j,\alpha}a_ja_\alpha\Phi_{j\alpha}\right)\cdot\left(\sum_{k,\beta}a_ka_\beta\Phi_{k\beta}\right)$$ appear twice with a plus and twice with a minus sign, and thus they cancel.
By skew-symmetry of $A,$ we have that $$-\sum_{j,k=1}^{2n}A_{jk}A_{kj}+\sum_{j,k=1}^{2n}A_{jk}^2=2\cdot\sum_{j,k=1}^{2n}A_{jk}^2.$$
    As in the proof of Proposition \ref{so2nunprop}, we obtain the following equation:
    $$\kappa(\hat{\phi}_{a,b},\hat{\phi}_{a,b})=-\hat{\phi}_{a,b}^2+2\cdot\left((a,a)(b,b)-(a,b)^2\right).$$
    Since we have assumed that $(a,a)(b,b)-(a,b)^2=0$ we conclude that $$\kappa(\hat{\phi}_{a,b},\hat{\phi}_{a,b})=-\hat{\phi}_{a,b}^2.$$ 
\end{proof}

\section{Compact Minimal Submanifolds of $\SO{2n}/\U n$}\label{cmsso2nun}
We will now investigate the regularity over $0\in\cn$ of the functions given in Propositions \ref{so2nunprop} and \ref{so2nunnew}. 

In Example \ref{exampleso2nunold} we show that the functions described in Proposition \ref{so2nunprop} are not necessarily regular over $0\in\cn.$
On the other hand, Theorem \ref{thmsso2nun} shows that a particular subset of the functions given in Proposition \ref{so2nunnew} are indeed regular over $0\in\cn.$ This allows us to construct compact minimal submanifolds using Theorem \ref{GM}.

Remark \ref{remarkso2n} provides us with some useful computations.
\begin{remark}\label{remarkso2n}
    Let $\hat{\phi}_{a,b}(x)=\trace(AxJ_nx^t)$ be the function given in Proposition \ref{so2nunprop}. 
    %Recall that we require both of the linearly independent vectors $a,b\in\cn^{2n}\backslash\{0\}$ to satisfy $(a,a)=(b,b)=0.$
    Skew-symmetry of $A$ and $\Phi$ yields \begin{eqnarray*}
        \hat{\phi}_{a,b}(x)&=&-\sum_{j,\alpha} A_{j\alpha}\Phi_{j\alpha}\\
        &=&-\frac{1}{\sqrt{2}}\cdot\sum_{j,\alpha}\left((a_jb_\alpha-b_ja_\alpha)\cdot\sum_{t=1}^n(-x_{j,n+t}x_{\alpha t}+x_{jt}x_{\alpha,n+t})\right)\\
        &=&\frac{2}{\sqrt{2}}\cdot\sum_{t=1}^n\left((a,x_{n+t})(b,x_t)-(a,x_t)(b,x_{n+t})\right).
    \end{eqnarray*}
Consider $Y_{rs}^{a}$ and $Y_{rs}^{b}$ as defined in Proposition \ref{decompositionso2nun}. A computation shows that $Y_{rs}^{a}(\hat{\phi}_{a,b})=0$ if and only if %$$0=\sum_{j,\alpha =1}^{2n}(a_jb_\alpha-b_ja_\alpha)\cdot(-x_{\alpha,n+s}\cdot x_{j,r}+x_{\alpha,n+r}\cdot x_{j,s}-x_{\alpha,s}\cdot x_{j,n+r}+x_{\alpha,r}\cdot x_{j,n+s}),$$ 
$$0=(a,x_r)(b,x_{n+s})+(a,x_{n+r})(b,x_s)-(a,x_s)(b,x_{n+r})-(a,x_{n+s})(b,x_r)$$
and similarly, $Y_{rs}^{b}(\hat{\phi}_{a,b})=0$ if and only if 
%$$0=\sum_{j,\alpha =1}^{2n}(a_jb_\alpha-b_ja_\alpha)\cdot(-x_{\alpha,s}\cdot x_{j,r}+x_{\alpha,r}\cdot x_{j,s}+x_{\alpha,n+s}\cdot x_{j,n+r}-x_{\alpha,n+r}\cdot x_{j,n+s}).$$
$$0=(a,x_r)(b,x_{s})+(a,x_{n+s})(b,x_{n+r})-(a,x_s)(b,x_{r})-(a,x_{n+r})(b,x_{n+s}).$$
\end{remark}

\begin{example}\label{exampleso2nunold}
    Let $n=2,$ and consider $a=(1,i,0,0),b=(0,0,1,i).$ It is clear that $a,b$ satisfy $(a,a)(b,b)-(a,b)^2$. Let $$x=\begin{pmatrix}
        0&0&0&-1\\
        0&0&1&0\\
        1&0&0&0\\
        0&1&0&0
    \end{pmatrix}\in\SO{2n}.$$ We will now use the computations from Remark \ref{remarkso2n}. It is easy to see that $$(a,x_1)=(a,x_2)=(b,x_3)=(b,x_4)=0$$ and $$(a,x_3)=i,\ (a,x_4)=-1,\ (b,x_1)=1,\ (b,x_2)=i.$$
    Then, since $$(a,x_3)(b,x_1)+(a,x_4)(b,x_2)=i-i=0,$$ it follows that $$\hat{\phi}_{a,b}(x)=0.$$ Further, $Y_{12}^{a}(\hat{\phi}_{a,b})=0 \ \textrm{and}\ Y_{12}^{b}(\hat{\phi}_{a,b})=0.$ Consequently, $\hat{\phi}_{a,b}$ is not regular over $0$ for our choices of $a,b.$ 
\end{example}
We will now show that with the relaxed conditions of Proposition \ref{so2nunnew}, it is easier to find suitable vectors $a,b$ that render $\hat{\phi}_{a,b}$ regular over $0$. We start with an easy example in the setting $n=2.$
\begin{example}\label{example2so2nun}
    Let now $\hat{\phi}_{a,b}$ be as in Proposition \ref{so2nunnew}, with the relaxed condition on $b.$ Let $n=2.$ We set $a=(1,i,0,0),$ and $b=(0,0,1,0).$ Clearly the conditions $$(a,a)=0, (a,b)=0$$ are satisfied, and consequently $\hat{\phi}_{a,b}$ as defined in Proposition \ref{so2nunnew} is an eigenfunction on $\SO{4}/\U{2}.$
    
    We now wish to show that $\hat{\phi}_{a,b}$ is regular over $0.$ For our choice of $a$ and $b$ we have that
    $$\hat{\phi}_{a,b}(x)=\frac{2}{\sqrt{2}}((x_{13}+ix_{23})x_{31}+(x_{14}+ix_{24})x_{32}-(x_{11}+ix_{21})x_{33}-(x_{12}+ix_{22})x_{34})).$$
    Since all entries of $x$ are real-valued, $\hat{\phi}_{a,b}(x)=0$ if and only if the following two equations are satisfied:
    \begin{eqnarray}
        0&=&x_{13}x_{31}+x_{14}x_{32}-x_{11}x_{33}-x_{12}x_{34},\label{so2nunphi3}\\
        0&=&x_{23}x_{31}+x_{24}x_{32}-x_{21}x_{33}-x_{22}x_{34}.
    \end{eqnarray}
    Similarly, by separating the real and complex parts, $Y_{12}^{a}(\hat{\phi}_{a,b})=Y_{12}^{b}(\hat{\phi}_{a,b})=0$ if and only if 
    \begin{eqnarray}
        0&=&-x_{14}x_{31}+x_{13}x_{32}-x_{12}x_{33}+x_{11}x_{34},\label{so2nunrow1}\\
        0&=&-x_{24}x_{31}+x_{23}x_{32}-x_{22}x_{33}+x_{21}x_{34},\\
        0&=&-x_{12}x_{31}+x_{11}x_{32}+x_{14}x_{33}-x_{13}x_{34},\label{so2nunrow2}\\
        0&=&-x_{22}x_{31}+x_{21}x_{32}+x_{24}x_{33}-x_{23}x_{34}.
    \end{eqnarray}
    Lastly, $x\in\SO{4}$ yields the following equations:
    \begin{eqnarray}
    0&=&x_{11}x_{31}+x_{12}x_{32}+x_{13}x_{33}+x_{14}x_{34},\label{eqso4}\\
    0&=&x_{21}x_{31}+x_{22}x_{32}+x_{23}x_{33}+x_{24}x_{34}.
    \end{eqnarray}
    We now have 4 equations where entries of the third row and entries of the first appear, namely Equations \ref{so2nunphi3},\ref{so2nunrow1},\ref{so2nunrow2} and \ref{eqso4}. Now assume that for some $x\in\SO{4},$ they all hold true. This means in particular that the non-zero vector $(x_{31},x_{32},x_{33},x_{34})^t$ belongs to the kernel of the matrix $$M=\begin{pmatrix}
        x_{13}&x_{14}&-x_{11}&-x_{12}\\
        -x_{14}&x_{13}&-x_{12}&x_{11}\\
        -x_{12}&x_{11}&x_{14}&-x_{13}\\
        x_{11}&x_{12}&x_{13}&x_{14}.
    \end{pmatrix}.$$
    However, $$\det(M)=(x_{11}^2+x_{12}^2+x_{13}^2+x_{14}^2)^2=1,$$ 
    since $x\in\SO{4}.$ By a result from elementary linear algebra, $\ker M=\{0\}.$
    This contradicts the assumption that the third row has non-zero entries. It follows that $\hat{\phi}_{a,b}$ is regular over $0.$ 

    As a consequence, we see that $\phi_{a,b}$ is regular over $0$ and that $\phi_{a,b}^{-1}(0)$ is a minimal submanifold, due to Theorem \ref{GM}.
\end{example}
We will now give a more general statement on how to choose vectors $a,b$ such that the corresponding function $\phi_{a,b}$ is regular over $0.$
\begin{theorem}\label{thmsso2nun}
    Let  $a,b\in\cn^{2n}$ satisfy the following conditions: \begin{enumerate}
        \item $a$ and $b$ are linearly independent,
        \item $(a,a)=(a,b)=0,$
        \item $b=z\cdot\Tilde{b}$ for some $z\in\cn^*$ and $\Tilde{b}\in\rn^{2n}.$
    \end{enumerate}
    Let $$A=\sum_{j,\alpha=1}^{2n} a_jb_\alpha Y_{j\alpha}.$$
     Then the eigenfunction $$\phi_{a,b}:\SO{2n}/\U{n}\rightarrow\cn,$$ induced by the $\U{n}$-invariant function $\hat{\phi}_{a,b}:\SO{2n}\rightarrow\cn,$ given by $$\hat{\phi}_{a,b}(x)=\trace(AxJ_nx^t),$$ is regular over $0\in\cn$ and further, the fibre $\phi_{a,b}^{-1}(\{0\})$ is a compact minimal submanifold of $\SO{2n}/\U{n}$ of codimension two.
\end{theorem}
\begin{proof}
     Certainly $\phi_{a,b}$ satisfies the conditions of Proposition \ref{so2nunnew}, and thus $\phi_{a,b}$ is an eigenfunction on $\SO{2n}/\U{n}.$ Due to Corollary \ref{phi0}, it is clear that $0$ lies in the image of $\phi_{a,b}.$
     For $x\in\SO{2n},$ we denote $x_1,\dots,x_{2n}$ the linearly independent columns of $x.$ 

     We will first consider the case $b\in\rn^{2n}.$
     Assume towards a contradiction that for a fixed $x\in\SO{2n},$ $\hat{\phi}_{a,b}(x)=0$ and $\hat{\phi}_{a,b}$ is not regular at $x,$ meaning for all $1\leq r<s\leq2n$ we have that $Y_{rs}^{a}(\hat{\phi}_{a,b})=0 \  \textrm{and} \  Y_{rs}^{b}(\hat{\phi}_{a,b})=0.$

     In the following we will use the formulae given in Remark \ref{remarkso2n}. Further, we decompose $a$ into its real and imaginary parts $$a=u+i\cdot v,$$ where $u,v\in\rn^{2n}\backslash\{0\}.$ 
 \begin{eqnarray*}
        \hat{\phi}_{a,b}(x)
        &=&\frac{2}{\sqrt{2}}\cdot\sum_{t=1}^n\left((u+iv,x_{n+t})(b,x_t)-(u+iv,x_t)(b,x_{n+t})\right).
    \end{eqnarray*}
    Since all entries of the matrix $x$ and the vector $b$ are real-valued, $\hat{\phi}_{a,b}(x)=0$ implies that   \begin{eqnarray*}
        0&=&\sum_{t=1}^n(u,x_{n+t})(b,x_t)-\sum_{t=1}^n(u,x_{j,t})(b,x_{n+t}),\label{so2neq1}
    \end{eqnarray*}
    We now turn our attention to the gradient.
    Remark \ref{remarkso2n} states that $Y_{rs}^{a}(\hat{\phi}_{a,b})=0$ if and only if 
$$0=(a,x_{n+s})(b,x_r)-(a,x_{n+r})(b,x_s)+(a,x_s)(b,x_{n+r})-(a,x_r)(b,x_{n+s})$$
and similarly, $Y_{rs}^{b}(\hat{\phi}_{a,b})=0$ if and only if 
$$0=(a,x_s)(b,x_{r})-(a,x_r)(b,x_{s})-(a,x_{n+s})(b,x_{n+r})+(a,x_{n+r})(b,x_{n+s}).$$

    By taking the real parts of the equations $Y_{rs}^{a}(\hat{\phi}_{a,b})=0$ and $Y_{rs}^{b}(\hat{\phi}_{a,b})=0$ for $1\leq r<s\leq2n,$ we obtain the following equations:
    \begin{eqnarray*}
        0&=&-(u,x_{n+s})  (b,x_{r})+(u,x_{n+r})  (b,x_{s})-(u,x_{s})  (b,x_{n+r})+(u,x_{r})  (b,x_{n+s}),\\
        0&=&-(u,x_{s})  (b,x_{r})+(u,x_{r})  (b,x_{s})+(u,x_{n+s})  (b,x_{n+r})-(u,x_{n+r})  (b,x_{n+s}),\\
    \end{eqnarray*}
As we have established in Lemma \ref{lemmaiso}, $u\neq0.$  Since the columns $x_1,\dots,x_{2n}$ form an orthonormal basis of $\rn^{2n},$ there exist at least one index $1\leq r\leq2n,$ such that $$(u,x_r)\neq0.$$ Without loss of generality, assume $r=1.$ Further, we note that $$0=(a,b)=(u+iv,b)=(u,b)+i(v,b)=\sum_{t=1}^{2n}(u,x_t)(b,x_t)+i\cdot\sum_{t=1}^{2n}(v,x_t)(b,x_t).$$ Then clearly the real part of the right hand side vanishes, i.e. $$0=\sum_{t=1}^{2n}(u,x_t)(b,x_t).$$
Then for $ 1< k\leq n,$ we obtain the following system of $2(n-1)+2=2n$ equations.
     \begin{eqnarray*}
        0&=&\sum_{t=1}^{2n}(u,x_t)(b,x_{t}),\\
        0&=&-(u,x_{k})  (b,x_{r})+(u,x_{r})  (b,x_{k})+(u,x_{n+k})  (b,x_{n+r})-(u,x_{n+r})  (b,x_{n+k}),\\
        0&=&-\sum_{t=1}^n(u,x_{n+t})(b,x_{t})+\sum_{t=1}^n(u,x_{t})(b,x_{n+t}),\\
        0&=&-(u,x_{n+k})  (b,x_{r})+(u,x_{n+r})  (b,x_{k})-(u,x_{k})  (b,x_{n+r})+(u,x_{r})  (b,x_{n+k}).
    \end{eqnarray*}
These equations are true if and only if the vector $$((b,x_1),\dots,(b,x_{2n}))^t$$ belongs to the kernel of the following matrix $M.$
    {\small $$\begin{pmatrix}
        (u,x_1)&(u,x_2)&(u,x_3)&\dots&(u,x_n)&(u,x_{n+1})&(u,x_{n+2})&(u,x_{n+3})&\dots&(u,x_{2n})\\
        -(u,x_2)&(u,x_1)&0&\dots&0&(u,x_{n+2})&-(u,x_{n+1})&0&\dots&0\\
        -(u,x_3)&0&(u,x_1)&\dots&0&(u,x_{n+3})&0&-(u,x_{n+1})&\dots&0\\
        -(u,x_n)&0&0&\dots&(u,x_1)&(u,x_{2n})&0&0&\dots&-(u,x_{n+1})\\
        \vdots&\vdots&\vdots&\vdots&\vdots&\vdots&\vdots&\vdots&\vdots&\vdots\\
        -(u,x_{n+1})&(u,x_{n+2})&(u,x_{n+3})&\dots&(u,x_{2n})&(u,x_{1})&(u,x_{2})&(u,x_{3})&\dots&(u,x_{n})\\
        -(u,x_{n+2})&(u,x_{n+1})&0&\dots&0&-(u,x_{2})&(u,x_{1})&0&\dots&0\\
        -(u,x_{n+3})&0&(u,x_{n+1})&\dots&0&-(u,x_{3})&0&(u,x_{1})&\dots&0\\
        \vdots&\vdots&\vdots&\vdots&\vdots&\vdots&\vdots&\vdots&\vdots&\vdots\\
        -(u,x_{2n})&0&0&\dots&(u,x_{n+1})&-(u,x_{n})&0&0&\dots&(u,x_1)\\
    \end{pmatrix}$$}
     
    It is now easy to see that $M=S+(u,x_1)\cdot I,$ where $S$ is a skew-symmetric real-valued matrix. It is well known that matrices of this form are invertible (as shown in Lemma \ref{SIinvertible}), and consequently $((b,x_1),\dots,(b,x_{2n}))^t=(0,\dots,0).$
We conclude that $b=0$. This contradicts our assumptions, and we conclude that the $\U{n}$-invariant function $\hat{\phi}_{a,b}:\SO{2n}\rightarrow\cn$ is regular over 0. It follows that the compact fibre $\phi_{a,b}^{-1}(0)$ is regular.

        According to Theorem \ref{GM}, $\phi_{a,b}^{-1}(\{0\})$ is a compact minimal submanifold of $\SO{2n}/\U n$ of codimension two.
        
    If $b=z\cdot\Tilde{b},$ for some non-zero $z\in\cn$ and $\Tilde{b}\in\rn^{2n},$ then due to linearity of the form $(\cdot,\cdot),$ we may divide all equations by $z$ before taking real and complex parts. We then proceed as in the real case.

\end{proof}

%%%%%%%%%%%%%%%%%%%%%%%%%%%%%%%%%%%%%%%%%%%%%%%%%%%%%%%%%%%%%%

%%%%%%%%%%%%%%%%%%%%%%%%%%%%%%%%%%%%%%%%

\chapter{The Symmetric Space $\SU{2n}/\Sp{n}$}
 We start off with some basic calculations in Section \ref{prelimsu2nspn}, and in Section \ref{sectionefsu2nspn} we look at a family of eigenfunctions, which was constructed by Gudmundsson, Siffert and Sobak in their paper \cite{Gud-Sif-Sob-2}. In Section \ref{cmssu2nspn} we use those eigenfunctions to construct compact minimal submanifolds of the symmetric space $\SU{2n}/\Sp{n}.$ 

\section{Preliminaries}\label{prelimsu2nspn}
We show how the quaternionic unitary group $\Sp{n}$ can be identified with a subgroup of $\SU{2n}.$ We also present the corresponding orthogonal decomposition of the Lie algebra $\su{n}$ of $\SU{n},$ which we will work with later.
\begin{proposition}\label{spninsu2n}
The map $\psi:\GLH{n}\rightarrow\GLC{2n}$ given by $$\psi:q=z+jw\mapsto\begin{pmatrix}
    z&-\Bar{w}\\
    w&\Bar{z}
\end{pmatrix}$$ is an injective homomorphism of Lie groups. Consequently, we may identify $\Sp{n}$ with its image under $\psi.$ Further, $\psi(\Sp{n})\subset\SU{2n}.$
\end{proposition}
\begin{proof}
    We will first show that $\psi$ is a homomorphism. \begin{eqnarray*}
        \psi((z_1+jw_1)\cdot(z_2+jw_2))&=&\psi(z_1z_2-\Bar{w}_1w_2+j(w_1z_2+\Bar{z}_1w_2))\\
        &=&\begin{pmatrix}z_1z_2-\Bar{w}_1w_2& -\Bar{w}_1\Bar{z}_2-{z}_1\Bar{w}_2\\
        w_1z_2+\Bar{z}_1w_2&\Bar{z}_1\Bar{z}_2-{w}_1\Bar{w}_2
        \end{pmatrix}\\
        &=&\begin{pmatrix}
            z_1&-\Bar{w}_1\\
    w_1&\Bar{z}_1
        \end{pmatrix}\cdot \begin{pmatrix}
            z_2&-\Bar{w}_2\\
    w_2&\Bar{z}_2
        \end{pmatrix}\\
        &=&\psi(z_1+jw_1)\cdot\psi(z_2+jw_2).
    \end{eqnarray*}
Clearly $\psi$ is injective since $\ker\psi=I_n.$
If $q=z+jw\in\Sp{n},$ then $$q\Bar{q}^t=z\Bar{z}^t+\Bar{w}w^t+j(w\Bar{z}^t-\Bar{z}w^t)=I_n.$$ It follows that 
\begin{eqnarray*}
    I_n&=&z\Bar{z}^t+\Bar{w}w^t,\\
    0&=&w\Bar{z}^t-\Bar{z}w^t.
\end{eqnarray*}
Hence, \begin{eqnarray*}
    \psi(z+jw)\cdot\overline{\psi(z+jw)}^t&=&\begin{pmatrix}
        z\Bar{z}^t+\Bar{w}w^t&z\Bar{w}^t-\Bar{w}z^t\\
        w\Bar{z}^t-\Bar{z}w^t&\Bar{z}z^t+w\Bar{w}^t
    \end{pmatrix}\\
    &=&\begin{pmatrix}
        I_n&0\\
        0&I_n
    \end{pmatrix}\\
    &=&I_{2n}.
\end{eqnarray*}
The previous result shows that $|\det(\psi(z+jw))|=1.$ It is left to show that $\det(\psi(z+jw))=1.$ Let $$J_n=\begin{pmatrix}
    0&I_n\\
    -I_n&0
\end{pmatrix}.$$ As we will show in Lemma \ref{PhiSu2nSpn}, for $z+jw\in\Sp{n},$ it holds that $$\psi(z+jw)\cdot J_n\cdot\psi(z+jw)=J_n.$$  Using an identity of the {\it Pfaffian}, we see that $$\pf(J_n)=\pf\left(\psi(z+jw)\cdot J_n\cdot\psi(z+jw)\right)=\det(\psi(z+jw))\cdot \pf(J_n).$$ Since $\pf(J_n)^2=\det(J_n)\neq 0,$ it follows that $\det(\psi(z+jw))=1.$ A definition of the Pfaffian and detailed statements of its properties can be found in Haber's lecture notes \cite{Haber}.
\end{proof}

\begin{remark}{\rm \cite{Lind}}
    The map $\sigma:\SU{2n}\rightarrow\SU{2n}$ given by $$\sigma(z)=J_n\Bar{z}J_n$$ is an involutive automorphism which fixes $\Sp{n}.$ By Lemma \ref{GKsymspace}, $\SU{2n}/\Sp{n}$ is indeed a symmetric space.
\end{remark}

\begin{proposition}\label{orthdecompsu2nspn}
    The Lie algebra $\su{2n}$ of $\SU{2n}$ admits an orthogonal decomposition $$\su{2n}=\sp{n}\oplus\m,$$
    where $$\m=\left\{\begin{pmatrix}
        Z&W\\
        -\Bar{W}&Z^t
    \end{pmatrix}\ \rvert \ Z^*+Z=0, \ \trace(Z)=0, \ W^t+W=0\right\}.$$
\end{proposition}

\begin{proof}
    In Definition \ref{liealgebraspn}, we stated that the Lie algebra $\sp{n}$ of $\Sp{n}$ can be written as $$\sp{n}=\left\{ \begin{pmatrix}
    Z&-\Bar{W}\\
    W&\Bar{Z}
\end{pmatrix}\in\mathbb{C}^{2n\times2n} \ \rvert \ Z^*+Z=0, \ W^t-W=0 \right\}. $$
Generally, since $X\in\su{2n}$ satisfies $X^*+X=0$ and $\trace(X)=0,$ an element of $\su{2n}$ must be of the form $$X=\begin{pmatrix}
    Z&W\\
    -W^*&Y
\end{pmatrix},$$ such that $Z^*+Z=Y^*+Y=0,$ and $\trace(Z)+\trace(Y)=0.$
    Let  $$\su{2n}=\sp{n}\oplus\m$$ be an orthogonal decomposition. Then  $\m$ has elements of the form $\begin{pmatrix}
        Z&W\\
        -\Bar{W}&-\Bar{Z}
    \end{pmatrix}, $ where $Z,W$ satisfy \begin{enumerate}
        \item $Z^*+Z=0,$
        \item $\Im\trace(Z)=0,$
        \item $W^t+W=0.$
    \end{enumerate}
Since the first condition on $Z$ implies that the real parts of the diagonal entries of $Z$ vanish, we see that due to the second condition $\trace(Z)=0.$
Thus, $$\m=\left\{\begin{pmatrix}
        Z&W\\
        -\Bar{W}&Z^t
    \end{pmatrix}\ \rvert \ Z^*+Z=0, \ \trace(Z)=0, \ W^t+W=0\right\}.$$
\end{proof}

\begin{proposition} Consider the orthogonal decomposition of the Lie algebra $\su{2n}$ of $\SU{2n}$ given by $$\su{2n}=\sp{n}\oplus\m.$$
    Then $\m$ is spanned by $$\mathcal{S}_\m=\left\{Y_{rs}^{a}=\frac{1}{2}\begin{pmatrix}
    Y_{rs}&0\\
    0&-Y_{rs}
\end{pmatrix},X_{rs}^{a}=\frac{1}{2}\begin{pmatrix}
    iX_{rs}&0\\
    0&iX_{rs}
\end{pmatrix},Y_{rs}^{b}=\frac{1}{2}\begin{pmatrix}
    0&Y_{rs}\\
    Y_{rs}&0
\end{pmatrix},\right.$$ $$\left.Y_{rs}^{c}=\frac{1}{2}\begin{pmatrix}
    0&iY_{rs}\\
    -iY_{rs}&0
\end{pmatrix} \  | \  1\leq r<s\leq n\right\}$$
$$\bigcup \ \left\{D_{rs}^{a}=\frac{1}{2}\begin{pmatrix}
    iD_{rs}&0\\
    0&iD_{rs}
\end{pmatrix}, \ \rvert \ 1\leq r<s\leq n\right\},$$ where $D_{rs}=\frac{1}{\sqrt{2}}\cdot(E_{rr}-E_{ss}).$
\end{proposition}
\begin{proof}
    In Proposition \ref{orthdecompsu2nspn} we have seen that an element of $\m$ is of the form $$\begin{pmatrix}
        Z&W\\
        -\Bar{W}&Z^t
    \end{pmatrix},$$ where \begin{enumerate}
        \item $Z^*+Z=0$,
        \item $\trace(Z)=0$ and 
        \item $ W^t+W=0.$
    \end{enumerate}
    We must now find generators of all matrices of the forms $Z$ and $W,$ as described above.
    
    We see that such matrices $Z$ must be in the Lie algebra $\su{n}$ of $\SU{n},$ which is is spanned by $$\{Y_{rs},iX_{rs},iD_{rs} \ \rvert \ 1\leq r<s\leq n\}.$$ The corresponding elements in $\mathcal{S}_m$ are $Y_{rs}^{a},X_{rs}^{a},D_{rs}^{a}.$ 
    
    Further, we note that the subspace of matrices $W\in\cn^{n\times n}$ is generated by $$\{Y_{rs},iY_{rs} \ \rvert \ 1\leq r<s\leq n\}.$$ The corresponding elements in $\mathcal{S}_\m$ are $Y_{rs}^{b}$ and $Y_{rs}^{c}.$
\end{proof}

\section{Eigenfunctions on $\SU{2n}/\Sp{n}$}\label{sectionefsu2nspn}
We now aim to construct eigenfunctions on $\SU{2n}/\Sp{n}.$ Lemma \ref{PhiSu2nSpn} helps us to construct $\Sp{n}$-invariant functions on $\SU{2n}.$ Proposition \ref{propsusp} from Gudmundsson, Siffert and Sobak's paper \cite{Gud-Sif-Sob-2} then provides us with a family of eigenfunctions on $\SU{2n}/\Sp{n}.
$
\begin{lemma}\label{PhiSu2nSpn}{\rm \cite{Hall}}

    Let $\Phi:\SU{2n}\rightarrow\SU{2n}$ be the map given by $$\Phi(z)=z\cdot J_n\cdot z^t,$$ where $$J_n=\begin{pmatrix}
    0&I_n\\
    -I_n&0
\end{pmatrix}.$$ An easy computation shows that if $z\in\Sp{n}$ then
$\Phi(z)=J_n$. In particular, $\Phi$ is $\Sp{n}$-invariant.
\end{lemma}

\begin{proposition}{\rm \cite{Gud-Sif-Sob-2}}\label{propsusp}
For non-zero linearly independent elements $a,b\in\mathbb{C}^{2n},$ let $A\in\mathbb{C}^{2n\times2n}$ be the skew-symmetric matrix $$A=\sum_{i,j=1}^{2n} a_ib_jY_{ij}.$$ 
Let $\Phi$ be the map given by $$\Phi(z)=z\cdot J_n\cdot z^t.$$ Define the function $\hat{\phi}_{a,b}:\SU{2n}\rightarrow\mathbb{C}$ by $$\hat{\phi}_{a,b}(z)=\trace(A\cdot\Phi(z))=\trace(z^tAzJ_n)=-\sum_{j,\alpha=1}^{2n}A_{j\alpha}\Phi_{j\alpha}(z).$$ Then $\hat{\phi}_{a,b}$ is a $\Sp{n}$-invariant eigenfunction on $\SU{2n},$ and thus induces an eigenfunction $\phi_{a,b}$ on the quotient space $\SU{2n}/\Sp{n}.$ The eigenvalues $\lambda$ and $\mu$ of $\phi_{a,b}$ satisfy $$\lambda=-\frac{2(2n^2-n-1)}{n}\quad\textrm{and}\quad\mu=-\frac{2(n-1)}{n}.$$
\end{proposition}

\section{Compact Minimal Submanifolds of $\SU{2n}/\Sp{n}$}\label{cmssu2nspn}
We wish to apply Theorem \ref{GM} to the functions stated in proposition \ref{propsusp}. As Theorem \ref{minsubmsuspn} shows, they are regular over $0\in\cn$ and thus can be used to construct compact minimal submanifolds of $\SU{2n}/\Sp{n}.$

\begin{theorem}\label{minsubmsuspn}
    Let $n\geq2$. Let $a,b\in\cn^{2n}$ be linearly independent and let $A$ be the skew-symmetric matrix given by $$A=\sum_{i,j=1}^{2n} a_ib_jY_{ij}.$$ Let $\phi_{a,b}:\SU{2n}/\Sp{n}\rightarrow\cn$  be the eigenfunction induced by $$\hat{\phi}_{a,b}=\trace(z^tAzJ_n).$$ Then $\phi_{a,b}$ is regular over $0\in\cn,$ and $\phi_{a,b}^{-1}(\{0\})$ is a compact minimal submanifold of $\SU{2n}/\Sp{n}$ of codimension two.
\end{theorem}
\begin{proof}
First of all, $\phi_{a,b}$ is an eigenfunction on $\SU{2n}/\Sp{n} $ by Proposition \ref{propsusp}. Due to Corollary \ref{phi0}, $0$ lies in the image of $\phi_{a,b}.$
    
    Assume towards a contradiction that for $z\in\SU{2n}$ we have that $\hat{\phi}_{a,b}(z)=0,$ and for every $Y\in\m,$ $Y(\hat{\phi}_{a,b})=0.$ We will show that under those assumptions, $z$ does not have full rank. We will start by unravelling what our assumptions mean.
    
    For $Y\in\m,$ $$Y(\hat{\phi}_{a,b})=\trace(Az(YJ_n+J_nY^t)z^t).$$ For simplicity we will now write $(a,z_t)=\sum_{j} a_jz_{jt}$ for all $1\leq t\leq 2n.$
    It can now be shown that for all $1\leq r<s\leq n,$ $$Y_{rs}^{a}(\hat{\phi}_{a,b})=0, \ Y_{rs}^{b}(\hat{\phi}_{a,b})=0,\ X_{rs}^{a}(\hat{\phi}_{a,b})=0, \ Y_{rs}^{c}(\hat{\phi}_{a,b})=0 \ $$ if and only if 
\begin{eqnarray}
    0&=&(a,z_{n+s})(b,z_{r})-(a,z_{n+r})(b,z_{s})+(a,z_{s})(b,z_{n+r})-(a,z_{r})(b,z_{n+s}),\quad\quad\label{suspeq1}\\
    0&=&(a,z_{s})(b,z_{r})-(a,z_{r})(b,z_{s})-(a,z_{n+s})(b,z_{n+r})+(a,z_{n+r})(b,z_{n+s}),\label{suspeq2}\\
    0&=&(a,z_{n+s})(b,z_{r})+(a,z_{n+r})(b,z_{s})-(a,z_{s})(b,z_{n+r})-(a,z_{r})(b,z_{n+s}),\label{suspeq3}\\
    0&=&(a,z_{s})(b,z_{r})-(a,z_{r})(b,z_{s})+(a,z_{n+s})(b,z_{n+r})-(a,z_{n+r})(b,z_{n+s}),\label{suspeq4}
\end{eqnarray}
By adding and subtracting Equations \ref{suspeq1} and \ref{suspeq3}, and Equations \ref{suspeq2} and \ref{suspeq4} respectively, we see that for all $1\leq r<s\leq n,$
\begin{eqnarray}
    0&=&(a,z_{n+s})(b,z_{r})-(a,z_r)(b,z_{n+s}),\label{suspeq5}\\
    0&=&(a,z_{n+r})(b,z_{s})-(a,z_s)(b,z_{n+r}),\label{suspeq6}\\
    0&=&(a,z_s)(b,z_{r})-(a,z_r)(b,z_{s}).\label{suspeq7}\\
    0&=&(a,z_r)(b,z_{s})-(a,z_s)(b,z_{r}).\label{suspeq8}
\end{eqnarray}  
Since the columns of $\Bar{z}$ form a basis of $\cn^{2n}$, there exists an index $1\leq t\leq 2n$ such that $$\langle a,\Bar{z}_t\rangle =(a,z_t)\neq0.$$ Without loss of generality, $1\leq t\leq n.$ Equations \ref{suspeq5},\ref{suspeq6},\ref{suspeq7} and \ref{suspeq8} now show that for all $1\leq s\leq n, \ s\neq t,$ it holds that $$(b,z_{n+s})=\frac{(a,z_{n+s})}{(a,z_t)}\cdot(b,z_t), \ \textrm{and} \ (b,z_{s})=\frac{(a,z_s)}{(a,z_t)}\cdot(b,z_t).$$
Now note that
 $\hat{\phi}_{a,b}(z)=0$ if and only if $$0=(a,z_{n+1})(b,z_{1})+\dots +(a,z_{2n})(b,z_{n})-(a,z_{1})(b,z_{n+1})-\dots-(a,z_{n})(b,z_{2n}).$$
Entering our latest findings into this equation, we note that
\begin{eqnarray*}
    0&=&(a,z_{n+t})(b,z_t)+\sum_{s\neq t}(a,z_{n+s})\cdot\frac{(a,z_s)}{(a,z_t)}\cdot(b,z_t)\\
    &&-(a,z_t)(b,z_{n+t})-\sum_{s\neq t}(a,z_s)\cdot\frac{(a,z_{n+s})}{(a,z_t)}\cdot(b,z_t)\\
    &=&(a,z_{n+t})(b,z_{t})-(a,z_t)(b,z_{n+t}).
\end{eqnarray*}
Thus, we also have that $$(b,z_{n+t})=\frac{(a,z_{n+t})}{(a,z_t)}\cdot (b,z_{t}).$$
To summarise, we have now established that for every $1\leq s\leq 2n,$ $$\langle b,\Bar{z}_s\rangle =(b,z_{s})=\frac{(a,z_{s})}{(a,z_t)}\cdot(b,z_t)=\frac{(b,z_t)}{(a,z_t)}\cdot\langle a, \Bar{z}_s\rangle.$$ 
Thus, $$b=\sum_{s=1}^{2n}\langle b,\Bar{z}_s\rangle\cdot \Bar{z}_s=\frac{(b,z_t)}{(a,z_t)}\cdot\sum_{s=1}^{2n}\langle a,\Bar{z}_s\rangle\cdot \Bar{z}_s=\frac{(b,z_t)}{(a,z_t)}\cdot a.$$ 
This is a contradiction, since we assumed $a$ and $b$ to be linearly independent. It follows that $\phi_{a,b}$ is regular over $0.$ 
By Theorem \ref{GM}, $\phi_{a,b}^{-1}(\{0\})$ is a minimal submanifold of $\SU{2n}/\Sp{n}.$
\end{proof}

\chapter{The Real Grassmannians $\SO{m+n}/\SO{m}\times\SO{n}$}
Section \ref{efrealgrass} provides us with some useful details on the real Grassmannians. We then study a family of eigenfunctions on $\SO{m+n}/\SO{m}\times\SO{n}$, which stems from Ghandour and Gudmundsson's article \cite{Gha-Gud-4}.
In Section \ref{section2realgrass}, we study minimal submanifolds of $\SO{m+n}/\SO{m}\times\SO{n}$.
We show in Theorem \ref{exsososo2} that none of the aforementioned eigenfunctions is regular over $0.$ Example \ref{exsososo1} illustrates this fact in the setting $m=n=2.$
For $m=3$ and $n=1,$ Example \ref{consolationprize} shows the existence of a compact minimal submanifold.

For the rest of the chapter we assume that the positive integers $m,n$ are not simultaneously equal to $1.$

\section{Eigenfunctions on $\SO{m+n}/\SO{m}\times\SO{n}$}\label{efrealgrass}
We first define the space $\SO{m+n}/\SO{m}\times\SO{n}$ and give a suitable decomposition of the Lie algebra $\so{m+n}$ of $\SO{m+n}.$ Thereafter, we show how to construct eigenfunctions on $\SO{m+n}/\SO{m}\times\SO{n}$, which we will state in Theorem \ref{thsososo}.
\begin{remark}
    
$\SO{m}\times\SO{n}$ consists of the block matrices of the form $$\SO{m}\times\SO{n}=\left\{ \begin{pmatrix}
    x&0\\
    0&y
\end{pmatrix}\ \rvert \ \ x\in\SO{m}, y\in\SO{n} \right\}. $$
    Let $\so{m+n}=\mathfrak{k}\oplus\mathfrak{m}$ be an orthogonal decomposition, where $\mathfrak{k}$ is the Lie algebra of the subgroup $\SO{m}\times\SO{n}$ of $\SO{m+n}$. Then an orthonormal basis for the subspace $\m$ of $\so{m+n}$ is given by
    $$\mathcal{B}_\mathfrak{m}=\{ Y_{rs} \ \rvert \ 1\leq r\leq m<s\leq m+n \}. $$
\end{remark}

We now study $\SO{m}\times\SO{n}$-invariant eigenfunctions on $\SO{m+n}.$ They stem from Gudmundsson and Ghandour's work \cite{Gha-Gud-4}. We start off by looking at the coordinate functions, which are our smallest building blocks. 
\begin{lemma}{\rm \cite{Gha-Gud-4}}\label{lemmasososocoord}
    For $1\leq j,k,\alpha,\beta\leq m+n,$ let $x_{j\alpha}:\SO{m+n}\rightarrow\rn$ be the real-valued matrix elements of the standard representation of $\SO{m+n}.$ Then the following relations hold:
    $$\tau (x_{j\alpha})=-\frac{m+n-1}{2}\cdot x_{j\alpha},\ \ \textrm{and}\ \ \kappa (x_{j\alpha},x_{k\beta})=-\frac{1}{2}\cdot(x_{j\beta}x_{k\alpha}-\delta_{jk}\delta_{\alpha\beta}).$$
\end{lemma}

\begin{remark}\label{sigmasososo}
Let $$I_{m,n}=\begin{pmatrix}
I_m & 0 \\
0 & -I_n
\end{pmatrix}.$$ A simple calculation shows that if $x\in\SO{m}\times\SO{n}$, then
\begin{eqnarray*}
    \sigma(x)=x\cdot I_{m,n}\cdot x^t=I_{m,n}.
\end{eqnarray*} Hence, $\sigma$ is $\SO{m}\times\SO{n}$-invariant. Indeed, for all $w\in\SO{m+n}$ and $x\in\SO{m}\times\SO{n},$ it holds that $$\sigma(wx)=wx\cdot I_{m,n}\cdot x^tw^t=w\cdot I_{m,n}\cdot w^t=\sigma(w).$$
\end{remark}
Using Lemma \ref{lemmasososocoord} and Remark \ref{sigmasososo}, Gudmundsson and Ghandour constructed $\SO{m}\times\SO{n}$-invariant functions on $\SO{m+n}$. 
\begin{lemma}{\rm \cite{Gha-Gud-4}}\label{lemmasososo}
    For $1\leq j,\alpha\leq m+n,$ define the function $\hat{\phi}_{j\alpha}:\SO {m+n}\rightarrow\rn$ by $$\hat{\phi}_{j\alpha}(x)=\sum_{t=1}^m x_{jt}x_{\alpha t}.$$ This function is $\SO{m}\times\SO{n}$-invariant and the tension field $\tau $ and the conformality operator $\kappa $ on $\SO{m+n}$ satisfy
    \begin{eqnarray*}
        \tau (\hat{\phi}_{j\alpha})&=&-(m+n)\hat{\phi}_{j\alpha}+\delta_{j\alpha}\cdot m,\\
        \kappa (\hat{\phi}_{j\alpha},\hat{\phi}_{k\beta})&=&-(\hat{\phi}_{j\beta}\cdot\hat{\phi}_{k\alpha}+\hat{\phi}_{jk}\cdot\hat{\phi}_{\alpha\beta})+\frac{1}{2}(\delta_{jk}\hat{\phi}_{\alpha\beta}+\delta_{\alpha\beta}\hat{\phi}_{jk}+\delta_{j\beta}\hat{\phi}_{k\alpha}+\delta_{k\alpha}\hat{\phi}_{j\beta}).
    \end{eqnarray*}
\end{lemma}

We now introduce a family of eigenfunctions on the real Grassmannian, which Gudmundsson and Ghandour constructed in \cite{Gha-Gud-4}.
\begin{theorem}{\rm \cite{Gha-Gud-4}}\label{thsososo}
Let $A$ be a complex symmetric $(m+n)\times(m+n)$ matrix such that $A^2=0.$ Let $$\hat{\phi}_{j\alpha}(x)=\sum_{t=1}^m x_{jt}x_{\alpha t}.$$ Then the $\SO{m}\times\SO{n}$-invariant function $\hat{\Phi}_A: \SO{m+n}\rightarrow\mathbb{C}$ given by $$\hat{\Phi}_A(x)=\sum_{j,\alpha=1}^{m+n} A_{j\alpha}\cdot\hat{\phi}_{j\alpha}(x)$$ induces an eigenfunction $\Phi_A:\SO{m+n}/\SO{m}\times\SO{n}\rightarrow\mathbb{C} $ on the real Grassmannian, if $\rank A=1$ and $\trace A=0.$ In this case,
$$\tau(\Phi_A)=-(m+n)\cdot\Phi_A \ \ \textrm{and} \ \ \kappa(\Phi_A,\Phi_A)=-2\cdot\Phi_A^2.$$
\end{theorem}

 \begin{remark}
     Lemma \ref{lemmaAaat} implies that all matrices $A$ which satisfy the conditions of Theorem \ref{thsososo} must be of the form $A=aa^t$ for some isotropic $a\in\cn^{m+n}\backslash\{0\}.$ To underline this fact and for consistency of notation, we will write $\Phi_a$ instead of $\Phi_A$ for the remainder of this chapter.
 \end{remark}

\section{Minimal Submanifolds of $\SO{m+n}/\SO{m}\times\SO{n}$}\label{section2realgrass}
We now show that none of the functions of the form $\Phi_a,$ as defined in Theorem \ref{thsososo}, are regular over $0\in\cn.$ We will first provide an easy example, and will then prove the general case in Theorem \ref{exsososo2}.

\begin{example}\label{exsososo1}
    Let $m=n=2,$ $a=(1,i,0,0)$ and $A=aa^t.$ We now want to find out whether $0$ is a regular value of the function $$\hat{\Phi}_a(x)=\sum_{j,\alpha=1}^{4} a_ja_\alpha\cdot\hat{\phi}_{j\alpha}(x)$$ from Theorem \ref{thsososo}.
Further, for $r=1,2$ and $s=3,4,$ an easy computation shows that $Y_{rs}(\hat{\Phi}_a(x))=0$ is equivalent to 
    \begin{eqnarray}\label{eqnphiA}
        0=(x_{1r}+ix_{2r})\cdot(x_{1s}+ix_{2s}).
    \end{eqnarray}
    If this equation holds, then at least one factor must be zero.
Since all entries of $x$ are real-valued, $x_{1t}+ix_{2t}=0$ implies that $x_{1t}=x_{2t}=0.$ 
Now assume that Equation \ref{eqnphiA} holds for all choices of $r$ and $s.$ Then either $$x_{11}=x_{21}=x_{12}=x_{22}=0\ \  \textrm{or}\ \  x_{13}=x_{23}=x_{14}=x_{24}=0.$$
We further note that $$\hat{\Phi}_a(x)=(x_{11}-x_{22}+i(x_{21}+x_{12}))\cdot(x_{11}+x_{22}+i(x_{21}-x_{12})).$$

Thus, $\nabla\hat{\Phi}_a=0$ at $x$ and $\hat{\Phi}_a(x)=0,$ if $x$ is of the form $$\begin{pmatrix}
    0&0&\cos\theta&\sin\theta\\
    0&0&-\sin\theta&\cos\theta\\
    \cos\alpha&\sin\alpha&0&0\\
    \sin\alpha&-\cos\alpha&0&0\\
\end{pmatrix} $$ for some $\alpha,\theta\in\rn.$
This shows that the set of critical points in the preimage of $\{0\}$ of $\hat{\Phi}_a$ is not discrete. Clearly $0$ is not a regular value.
\end{example}

We will now make a general statement.
\begin{theorem}\label{exsososo2}
    Let $a\in\cn\backslash\{0\}$ and assume $(a,a)=0.$ Let
$$\hat{\phi}_{j\alpha}(x)=\sum_{t=1}^m x_{jt}x_{\alpha t}$$ and let the $\SO{m}\times\SO{n}$-invariant function $\hat{\Phi}_a: \SO{m+n}\rightarrow\mathbb{C}$ be given by $$\hat{\Phi}_a(x)=\sum_{j,\alpha=1}^{m+n} a_ja_\alpha\cdot\hat{\phi}_{j\alpha}(x).$$ Define
$\Phi_a$ to be the eigenfunction on $\SO{m+n}/\SO{m}\times\SO{n}$ induced by $\hat{\Phi}_a$. Then $\Phi_a$ is not regular over $0\in\cn.$
\end{theorem}

\begin{proof}
Let $a\in\cn\backslash\{0\}$ be a vector satisfying $(a,a)=0.$ Let $a=u+iv$ be the decomposition of $a$ into its real and imaginary parts.
By Theorem \ref{thsososo} and Lemma \ref{lemmaAaat}, $$\Phi_a:\SO{m+n}/\SO{m}\times\SO{n}\rightarrow\cn$$ is a well defined eigenfunction.
In the following, we will denote by $x_t$ the $t$-th column of a point $x\in\SO{m+n}/\SO{m}\times\SO{n}.$
Through computation it can be shown that \begin{eqnarray*}
        \hat{\Phi}_a(x)&=&\sum_{j,\alpha=1}^{m+n} a_ja_\alpha\cdot\hat{\phi}_{j\alpha}(x)\\
        &=&\sum_{j,\alpha=1}^{m+n}\sum_{t=1}^m (a_ja_\alpha x_{jt}x_{\alpha t})\\
        &=&\sum_{t=1}^{m}\left(\left(\sum_{j=1}^{m+n} a_jx_{jt}\right)\cdot \left(\sum_{\alpha=1}^{m+n} a_\alpha x_{\alpha t}\right)\right)\\
        &=&\sum_{t=1}^{m}(a,x_t)^2.
    \end{eqnarray*}
    Further, for $Y_{rs}\in\mathcal{B}_{\m},$ with $1\leq r\leq m<s\leq m+n,$ we have that $Y_{rs}(\hat{\Phi}_a)=0$ if and only if $$0=\left(\sum_{j=1}^{m+n} a_jx_{jr}\right)\cdot \left(\sum_{\alpha=1}^{m+n} a_\alpha x_{\alpha s}\right)=(a,x_r)\cdot(a,x_s).$$
    We will now construct an element $x\in\SO{m+n}$ such that $\hat{\Phi}_a(x)=0$ and for every $Y_{rs}\in\mathcal{B}_\m$, we have that $Y_{rs}(\hat{\Phi}_a)=0$ at $x.$
    
       Assume first that $n\geq2.$ Then we can choose $m$ pairwise orthogonal unit vectors $x_1,\dots,x_m$ of $\rn^{m+n},$ which are also orthogonal to $u$ and $v.$ We may now choose $n$ vectors $x_{m+1},\dots x_{m+n}\in\rn^{m+n}$ that extend
       $\{x_1,\dots,x_m\}$ to an orthonormal basis of $\rn^{m+n}.$ Then the matrix $x$ with columns $$(x_1,\dots x_{m+n})$$ clearly belongs to $\O{m+n}.$ If $\det x=-1,$ multiply one of the columns by $-1.$ Hence we may assume that there exists a matrix $x\in\SO{m+n},$ such that the first $m$ columns $x_r, 1\leq r\leq m$ satisfy $$(a,x_r)=(u,x_r)+i(v,x_r)=0.$$  It follows that for all choices $1\leq r\leq m<s\leq m+n,$ we have that $$Y_{rs}(\hat{\Phi}_a)=(a,x_r)\cdot(a,x_s)=0.$$ It trivially holds that $$\hat{\Phi}_a(x)=\sum_{t=1}^{m}(a,x_t)^2=0.$$ We conclude that if $n>1,$ then $0$ is not a regular value of $\hat{\Phi}_a,$ and neither of $\Phi_a.$
       
       Let us now consider the case $m>1, n=1.$ Recall from Lemma \ref{lemmaiso} that $u,v\in\rn^{m+1}$ are linearly independent, non-zero and orthogonal. We now have to adapt our proof, since there do not exist $m$ pairwise orthogonal unit vectors $x_1,\dots,x_m$ of $\rn^{m+1}$ that satisfy %$$\sum_{j=1}^{m+n} (u_j+i\cdot v_j)x_{j,r}=0$$ for all $1\leq r\leq m.$
       $$(u+iv, x_r)=0$$ for all $1\leq r\leq m.$ 
       Hence, if $$Y_{r,m+1}(\hat{\Phi}_a)=0$$ for all such $r,$ we require that $$(a,x_{m+1})=0.$$ Let $\Tilde{x}\in\rn^{m+1}$ be any real non-zero unit vector that is orthogonal to both $u$ and $v,$ i.e. it satisfies $$(a,\Tilde{x})=0.$$ Such a vector exists since $m\geq2.$
       If $m=2,$ let $x$ be the matrix with columns $$\left(\frac{u}{|u|},\frac{v}{|v|},\Tilde{x}\right).$$ Otherwise,
       we may now choose $m-2$ vectors $x_1,\dots,x_{m-2}\in\rn^{m+1}$ that extend $\left\{\frac{u}{|u|},\frac{v}{|v|},\Tilde{x}\right\}$ to an orthonormal basis of $\rn^{m+1}.$ 
       By the same argument as above, we may assume that the matrix $x$ with columns $$\left(\frac{u}{|u|},\frac{v}{|v|},x_1,\dots,x_{m-2},\Tilde{x}\right)$$ belongs to $\SO{m+1}.$ 
       By construction, $$Y_{r,m+1}(\hat{\Phi}_a)=0$$ at $x$ for all $1\leq r\leq m.$ This means that $\hat{\Phi}_a$ is not regular at $x.$ 
       Further, we have that \begin{eqnarray*}
        \hat{\Phi}_a(x)&=&\sum_{t=1}^{m}(a,x_t)^2\\
        &=&\frac{1}{|u|^2}\cdot(u+iv,u)^2+\frac{1}{|v|^2}\cdot(u+iv,v)^2\\
        &=&\frac{1}{|u|^2}\cdot(u,u)^2+\frac{1}{|v|^2}\cdot(iv,v)^2\\
        &=&\frac{1}{|u|^2}\cdot|u|^4-\frac{1}{|v|^2}\cdot|v|^4\\
        &=&|u|^2-|v|^2.\\
        &=&0.
    \end{eqnarray*}
    Here, we again used Lemma \ref{lemmaiso}. It follows that $\Phi_a$ is not regular over 0.
    This proves the statement for $m>1,n=1.$ 

\end{proof}

As a consequence,the application of Theorem \ref{GM} to functions of the form $\Phi_a$ does not yield compact minimal submanifolds. 
For sake of completeness, we will state an Example of a compact minimal submanifold on $\SO{4}/\SO{3},$ namely the famous Clifford Torus. 

\begin{example}\label{consolationprize}
    (A consolation prize.) As we have seen in Example \ref{spheresoso}, $$S^3\cong\SO{4}/\SO{3}.$$ Since $\SO{1}=\{e\}$ is the trivial group, we may write $$S^3\cong\SO{4}/\SO{3}\times\SO{1}.$$ A well-known minimal submanifold of $S^3$ is the Clifford torus $$T_{{\pi/4}}=\left(\tfrac{1}{\sqrt2}\cdot e^{i\alpha},\tfrac{1}{\sqrt2}\cdot e^{i\beta}\right).$$ We have already discussed it in more detail in Example \ref{clifford}.
\end{example}
%%%%%%%%%%%%%%%%%%%%%%%%%%%%%%%%%%%%%%%%%%

\chapter{The Complex Grassmannians $\U{m+n}/\U{m}\times\U{n}$}
We will now turn our attention to the complex Grassmannians. 
In Section \ref{sectionefununun} we study eigenfunctions on $\U{m+n}/\U{m}\times\U{n}$.
Lemma \ref{lemma1ununun} corrects a lemma given in Gudmundsson and Ghandour's paper \cite{Gha-Gud-5}. This allows us to construct a new family of eigenfunctions on $\U{m+n}/\U{m}\times\U{n}$, which is done in Theorem \ref{theigenfunctionununun}. 
We then study minimal submanifolds of $\U{m+n}/\U{m}\times\U{n}$ in Section \ref{msununun}.
We show with Theorem \ref{thunununnotregular} that those eigenfunctions are never regular over $0\in\cn.$  
Overall, this chapter uses similar arguments as in the last chapter on the real Grassmannians. 

\section{Eigenfunctions on $\U{m+n}/\U{m}\times\U{n}$}\label{sectionefununun}
We first show in Lemma \ref{coordfunctionsununun} how the tension field and conformality operator act on the coordinate functions. This Lemma stems from Gudmundsson and Ghandour's paper \cite{Gha-Gud-5}. We then use it in Lemma \ref{lemma1ununun} to correct a similar statement made in \cite{Gha-Gud-5}. This allows us to construct a family of eigenfunctions on $\U{m+n}/\U{m}\times\U{n}$, as shown in Theorem \ref{theigenfunctionununun}.

\begin{lemma}{\rm \cite{Gha-Gud-5}}\label{coordfunctionsununun}
Let $z_{j\alpha}:\U{n}\rightarrow\cn$ be the matrix elements of the standard representation of the unitary group $\U{n}.$ Then, the tension field and conformality operators $\tau,\kappa$ satisfy the following relations:
\begin{eqnarray*}
        \tau(z_{j\alpha})&=&-n\cdot z_{j\alpha}, \quad \kappa(z_{j\alpha},z_{k\beta})=-z_{j\beta}z_{k\alpha},\\
        \tau(\Bar{z}_{j\alpha})&=&-n\cdot \Bar{z}_{j\alpha}, \quad \kappa(\Bar{z}_{j\alpha},\Bar{z}_{k\beta})=-\Bar{z}_{j\beta}\Bar{z}_{k\alpha},\\
        \kappa(z_{j\alpha},\Bar{z}_{k\beta})&=&\delta_{jk}\delta_{\alpha\beta}.\\
    \end{eqnarray*}
\end{lemma}
We offer the following correction of Lemma 8.1 of Ghandour and Gudmundsson's paper \cite{Gha-Gud-5}. We would like to point out that in his thesis \cite{Lind}, Lindström independently came to the same result as us using a different proof. 
\begin{lemma}{\rm \cite{Gha-Gud-5}}\label{lemma1ununun}
For $1\leq j,\alpha\leq m+n,$ we define the complex-valued functions $$\hat{\phi}_{j\alpha}:\U{m+n}\rightarrow\cn$$ on the unitary group by $$\hat{\phi}_{j\alpha}(z)=\sum_{t=1}^m z_{jt}\Bar{z}_{\alpha t}.$$ 
The tension field $\tau $ and the conformality operator $\kappa $ on the unitary group $\U{m+n}$ satisfy
$$\tau (\hat{\phi}_{j\alpha})=-2(m+n)\cdot\hat{\phi}_{j\alpha}+2m\cdot\delta_{j\alpha},$$
and %$$\kappa (\hat{\phi}_{jk},\hat{\phi}_{lm})=-2\cdot\hat{\phi}_{jk}\hat{\phi}_{lm}.$$
$$\kappa (\hat{\phi}_{j\alpha},\hat{\phi}_{k\beta})=-2\cdot\hat{\phi}_{j\beta}\hat{\phi}_{k\alpha}+\delta_{j\beta}\cdot\hat{\phi}_{k\alpha}+\delta_{k\alpha}\cdot\hat{\phi}_{j\beta}.$$
\end{lemma}
\begin{proof}
    Using Lemma \ref{coordfunctionsununun}, we now obtain the following:
    \begin{eqnarray*}
        \tau(\hat{\phi}_{j\alpha})&=&\sum_{t=1}^m \tau(z_{jt}\Bar{z}_{\alpha t})\\
        &=&-2(m+n)\cdot\sum_{t=1}^m z_{jt}\Bar{z}_{\alpha t}+2\cdot\sum_{t=1}^m \delta_{j\alpha}\\
        &=&-2(m+n)\hat{\phi}_{j\alpha}+2m\cdot\delta_{j\alpha}.
    \end{eqnarray*}

\begin{eqnarray*}
    \kappa(\hat{\phi}_{j\alpha},\hat{\phi}_{k\beta})&=&\sum_{s=1}^m\sum_{t=1}^m\kappa(z_{js}\Bar{z}_{\alpha s},z_{kt}\Bar{z}_{\beta t})\\
    &=&\sum_{s=1}^m\sum_{t=1}^m\Bar{z}_{\alpha s}\Bar{z}_{\beta t}\kappa(z_{js},z_{kt})+z_{js}z_{kt}\kappa(\Bar{z}_{\alpha s},\Bar{z}_{\beta t})\\
&&\quad\quad\quad+\Bar{z}_{\alpha s}z_{kt}\kappa(z_{js},\Bar{z}_{\beta t})+z_{js}\Bar{z}_{\beta t}\kappa(\Bar{z}_{\alpha s},z_{kt})\\
&=&\sum_{s=1}^m\sum_{t=1}^m(-\Bar{z}_{\alpha s}\Bar{z}_{\beta t}z_{jt}z_{ks}+z_{js}z_{kt}\Bar{z}_{\alpha t}\Bar{z}_{\beta s}+\Bar{z}_{\alpha s}z_{kt}\delta_{j\beta}\delta_{st}+z_{js}\Bar{z}_{\beta t}\delta_{\alpha k}\delta_{st})\\
&=&-2\cdot\hat{\phi}_{j\beta}\hat{\phi}_{k\alpha}+\delta_{k\alpha}\cdot\hat{\phi}_{j\beta}+\delta_{j\beta}\cdot\hat{\phi}_{k\alpha}.\\
\end{eqnarray*}
\end{proof}

\begin{remark}
    For any $z\in\U{m},$ it holds by definition that $z\cdot z^*=I_m.$ Consequently,
    \begin{equation*}
        (z\cdot z^*)_{j\alpha}=\sum_{t=1}^m z_{jt}\Bar{z}_{\alpha t}=\delta_{j\alpha}.
    \end{equation*}
    If $j<\alpha,$ it is now clear that the functions $\hat{\phi}_{j\alpha}:\U{m+n}\rightarrow\cn,$ defined by $\hat{\phi}_{j\alpha}(z)=\sum_{t=1}^m z_{jt}\Bar{z}_{\alpha t}$ are $\U{m}\times\U{n}$-invariant and hence induce functions $$\phi_{j\alpha}:\U{m+n}/\U{m}\times\U{n}\rightarrow\cn.$$
\end{remark}

\begin{theorem}{\rm \cite{Gha-Gud-5}, \cite{Lind}}
    For a fixed natural number $1\leq\alpha<m+n,$ the set
    $$\mathcal{E}_\alpha=\{\phi_{j\alpha}:\U{m+n}/\U{m}\times\U{n}\rightarrow\cn \ \rvert \ 1\leq j\leq m+n, \ j\neq\alpha\}$$
    is an eigenfamily on the complex Grassmannian $\U{m+n}/\U{m}\times\U{n}$ such that the tension field $\tau$ and the conformality operator $\kappa$ satisfy
    $$\tau(\phi_{j\alpha})=-2\cdot(m+n)\cdot\phi_{j\alpha},$$ and $$\kappa(\phi_{j\alpha},\phi_{k\alpha})=-2\cdot\phi_{j\alpha}\cdot\phi_{k\alpha}$$ for all $\phi_{j\alpha},\phi_{k\alpha}\in\mathcal{E}_\alpha.$
\end{theorem}
\begin{proof}
    The statement  follows immediately from Lemma \ref{lemma1ununun}. 
\end{proof}

Building up on the previous lemmas, we now construct a family of $\U{m}\times U{n}$-invariant eigenfunctions on $\U{m+n}$. Those constructions are very similar to the ones we have seen in the last chapter, c.f. Theorem \ref{thsososo}. 
\begin{theorem}\label{theigenfunctionununun}
    Let $a,b\in\cn^{m+n}\backslash\{0\}$ satisfy $\langle a,\Bar{b}\rangle=0.$ Then the complex-valued function $\hat{\Phi}_{a,b}:\U{m+n}\rightarrow\cn$, given by $$\hat{\Phi}_{a,b}(z)=\sum_{j,\alpha=1}^{m+n} a_jb_\alpha\hat{\phi}_{j\alpha }(z),$$ is an $\U{m}\times\U{n}$-invariant eigenfunction and thus induces an eigenfunction $\Phi_{a,b}:\U{m+n}/\U{m}\times\U{n}\rightarrow\cn.$ In particular, 
    $$\tau (\hat{\Phi}_{a,b})=-2(m+n)\cdot\hat{\Phi}_{a,b}, \quad \kappa (\hat{\Phi}_{a,b},\hat{\Phi}_{a,b})=-2\cdot\hat{\Phi}_{a,b}^2.$$
\end{theorem}
\begin{proof}
    We will make use of Lemma \ref{lemma1ununun}. The rest of the computation is simple. 
    \begin{eqnarray*}
        \tau (\hat{\Phi}_{a,b})&=&\sum_{j,\alpha =1}^{m+n} a_jb_\alpha \tau (\hat{\phi}_{j\alpha })\\
        &=&\sum_{j,\alpha =1}^{m+n} a_jb_\alpha (-2(m+n)\cdot\hat{\phi}_{j\alpha }+2m\cdot\delta_{j\alpha })\\
        &=&-2(m+n)\cdot\hat{\Phi}_{a,b}+2m\cdot\sum_{j=1}^{m+n} a_jb_j\\
        &=&-2(m+n)\cdot\hat{\Phi}_{a,b}+2m\cdot(a,b)\\
        &=&-2(m+n)\cdot\hat{\Phi}_{a,b}.
    \end{eqnarray*}
    In the last step we used that $(a,b)=0.$
    \begin{eqnarray*}
        \kappa (\hat{\Phi}_{a,b},\hat{\Phi}_{a,b})&=&\sum_{j,\alpha =1}^{m+n}\sum_{k,\beta=1}^{m+n}a_jb_\alpha a_kb_\beta\cdot\kappa (\hat{\phi}_{j\alpha },\hat{\phi}_{k\beta})\\
        &=&\sum_{j,\alpha =1}^{m+n}\sum_{k,\beta=1}^{m+n}a_jb_\alpha a_kb_\beta(-2\cdot\hat{\phi}_{j\beta}\hat{\phi}_{k\alpha }+\delta_{j\beta}\hat{\phi}_{k\alpha }+\delta_{k\alpha }\hat{\phi}_{j\beta})\\
&=&-2\cdot\left(\sum_{j,\beta=1}^{m+n} a_jb_\beta\hat{\phi}_{j\beta}\right)\cdot\left(\sum_{k,\alpha =1}^{m+n} a_kb_\alpha \hat{\phi}_{k\alpha }\right)\\
&& +2\cdot\left(\sum_{j=1}^{m+n} a_jb_j\right)\cdot\left(\sum_{k,\alpha =1}^{m+n} a_kb_\alpha \hat{\phi}_{k\alpha }\right)\\
&=&-2\cdot\hat{\Phi}_{a,b}^2+2\cdot(a,b)\cdot\hat{\Phi}_{a,b}\\
&=&-2\cdot\hat{\Phi}_{a,b}^2.
    \end{eqnarray*}
\end{proof}

\section{Minimal Submanifolds of $\U{m+n}/\U{m}\times\U{n}$}\label{msununun}
We will now show with Theorem \ref{thunununnotregular} that the functions defined in Theorem \ref{theigenfunctionununun} are not regular over $0\in\cn.$ Therefore we may not apply Theorem \ref{GM}.
We first decompose the Lie algebra $\u{m+n}$ of $\U{m+n},$ which we will use in the proof of Theorem \ref{thunununnotregular}. 
\begin{proposition}
    Consider the orthogonal decomposition $$\u{m+n}=\mathfrak{k}\oplus \mathfrak{m},$$ where $\k$ is the Lie algebra of $\U{m}\times\U{n},$ and $\m$ is the subspace corresponding to $\U{m+n}/\U{m}\times\U{n}.$
Then
$$\mathcal{B}_\mathfrak{m}=\{Y_{rs}, iX_{rs}\ |\ 1\leq r\leq m,\space m+1\leq s\leq m+n\}$$ is a basis of $\m.$
\end{proposition}
\begin{proof}
    As shown in Proposition \ref{liealgebraonun},
 $$\mathcal{B}_{\u{m+n}}=\{Y_{rs}, iX_{rs}\ \rvert \ 1\leq r<s\leq m+n\}$$ is a basis of the Lie algebra $\u{m+n}$ of $\U{m+n}.$
It is easily seen that if $1\leq r<s\leq m$ or $m+1\leq r<s\leq m+n,$ then $Y_{rs}$ and $iX_{rs}$ belong to $\k.$ This proves the claim. 
\end{proof}
We are now ready to show the main result of this chapter. The statement and proof are similar to Theorem \ref{exsososo2} from the chapter on the real Grassmannians.

\begin{theorem}\label{thunununnotregular}
Let $a,b\in\cn^{m+n}\backslash\{0\}$ with $\langle a,\Bar{b}\rangle =0.$
    Let $$\Phi_{a,b}:\U{m+n}/\U{m}\times\U{n}$$ be the eigenfunction induced by the function $\hat{\Phi}_{a,b}:\U{m+n}\rightarrow\cn$ given by $$\hat{\Phi}_{a,b}(z)=\sum_{j,\alpha=1}^{m+n} a_jb_\alpha\hat{\phi}_{j\alpha }(z).$$
Then $\Phi_{a,b}$ is not regular over $0\in\cn.$
\end{theorem}
\begin{proof}
Let $a,b\in\cn^{m+n}\backslash\{0\}$ with $\langle a,\Bar{b}\rangle =0$ be given. %By Theorem \ref{theigenfunctionununun}, the function $\hat{\Phi}_{a,b}:\U{m+n}\rightarrow\cn$ given by $$\hat{\Phi}_{a,b}(z)=\sum_{j,\alpha=1}^{m+n} a_jb_\alpha\hat{\phi}_{j\alpha }(z)$$ induces an eigenfunction $\Phi_{a,b}$ on $\U{m+n}/\U{m}\times\U{n}. $

We now wish to construct $z\in\U{m+n}$ such that for all $X\in\m, \ X(\hat{\Phi}_{a,b})=0$ at $z$ and $\hat{\Phi}_{a,b}(z)=0.$ 

Let $z_t$ denote the $t$-th column of $z.$
It can be shown that for $1\leq r\leq m<s\leq m+n,$ $Y_{rs}(\hat{\Phi}_{a,b})=0$ if and only if
$$0=\sum_{j,k=1}^{m+n}a_jb_k(z_{js}\Bar{z}_{kr}+z_{jr}\Bar{z}_{ks}),$$ and 
$iX_{rs}(\hat{\Phi}_{a,b})=0$ if and only if
$$0=\sum_{j,k=1}^{m+n}a_jb_k(-z_{js}\Bar{z}_{kr}+z_{jr}\Bar{z}_{ks}).$$
Thus, it holds that $Y_{rs}(\hat{\Phi}_{a,b})=iX_{rs}(\hat{\Phi}_{a,b})=0,$ for all $1\leq r\leq m<s\leq m+n$ if and only if \begin{eqnarray}
    0&=&\sum_{j,k=1}^{m+n}a_jb_kz_{js}\Bar{z}_{kr}=\langle z_s,\Bar{a}\rangle \langle b,z_r\rangle ,\label{unununyrs}\\ 
    0&=&\sum_{j,k=1}^{m+n}a_jb_kz_{jr}\Bar{z}_{ks}=\langle z_r,\Bar{a}\rangle \langle b,z_s\rangle .\label{unununxrs}
\end{eqnarray}

Similarly, we may write \begin{eqnarray}
\hat{\Phi}_{a,b}(z)&=&\sum_{t=1}^m \langle z_t,\Bar{a}\rangle \langle b,z_t\rangle .\label{eqnphiununun}\end{eqnarray}
Since $\langle \Bar{a},b\rangle =0$ and $a$ and $b$ are non-zero, it is clear that $\Bar{a}$ and $b$ are linearly independent. Further, we can choose $m+n-2$ linearly independent vectors $\{e_1,\dots,e_{m+n-2}\}\in\cn^{m+n}$ such that the set $$S=\{\Bar{a},b,e_1,\dots,e_{m+n-2}\}$$ consists of $m+n$ linearly independent vectors of $\cn^{m+n}.$ Then, for example with the well-known Gram-Schmidt-algorithm, we can derive $m+n$ linearly independent, mutually orthogonal unit vectors of $\cn^{m+n}$ from $S:$
$$\left\{\frac{\Bar{a}}{|\Bar{a}|},\frac{b}{|b|},v_1,\dots,v_{m+n-2}\right\}.$$
If $n\geq2,$ we let $$z=\left(v_1,\dots,v_{m+n-2},\frac{\Bar{a}}{|\Bar{a}|},\frac{b}{|b|}\right).$$ Then $z\in\U{m+n}$ by construction. Further, for all $1\leq r\leq m, $ we have that $$0=\langle z_r,\Bar{a}\rangle =\langle b,z_r\rangle .$$ It follows that Equations \ref{unununyrs},\ref{unununxrs}, and \ref{eqnphiununun} are satisfied. Hence $z\in\hat{\Phi}_{a,b}^{-1}(\{0\}),$ and $\nabla(\hat{\Phi}_{a,b})=0$ at $z.$

If $n=1,$ let $$z=\left(\frac{\Bar{a}}{|\Bar{a}|},\frac{b}{|b|},v_1,\dots,v_{m+n-2}\right).$$ Again, $z\in\U{m+n}.$ Clearly, Equations \ref{unununyrs} and \ref{unununxrs} are satisfied since $$0=\langle z_n,\Bar{a}\rangle =\langle z_n,b\rangle .$$ Further, $\hat{\Phi}_{a,b}(z)=0$ since $\langle \Bar{a},b\rangle =0.$ In this case we now also have constructed a matrix $z$ such that $z\in\hat{\Phi}_{a,b}^{-1}(\{0\}),$ and $\nabla(\hat{\Phi}_{a,b})=0$ at $z.$

We conclude that $\Phi_{a,b}$ is not regular over $0\in\cn.$
\end{proof}

%%%%%%%%%%%%%%%%%%%%%%%%%%%%%%%%%%%%%%%%%%%

\chapter{The Quaternionic Grassmannians $\Sp{m+n}/\Sp{m}\times\Sp{n}$}
In Section \ref{efspspspsection} we define eigenfunctions on the quaternionic Grassmannians $G_m(\H^{m+n})=\Sp{m+n}/\Sp{m}\times\Sp{n}$ and show in Section \ref{sectionmsspspsp} that they are not regular over $0$. Our main reference is Ghandour and Gudmundsson's paper \cite{Gha-Gud-5}.  

\section{Eigenfunctions on $\Sp{m+n}/\Sp{m}\times\Sp{n}$}\label{efspspspsection}
We now study eigenfunctions on the quaternionic Grassmannians. We first define an eigenfamily found by Gudmundsson and Ghandour in their paper \cite{Gha-Gud-5}. In Theorem \ref{thspspspnew}, we then define a new eigenfunction, which can be seen as the analogue to the functions given in Theorems \ref{thsososo} and \ref{theigenfunctionununun}.
\begin{lemma}{\rm \cite{Gha-Gud-5}}\label{quatgrasscoord}
    For $1\leq j,\alpha\leq n,$ let $z_{j\alpha},w_{j\alpha}:\Sp{n}\rightarrow\cn$ be the complex-valued matrix elements of the standard representation of the quaternionic unitary group $\Sp{n}$. Then, the tension field $\tau$ and the conformality operator $\kappa$ on $\Sp{n}$ satisfy the following relations:
        $$\tau(z_{j\alpha})=-\frac{2n+1}{2}\cdot z_{j\alpha}, \ \tau(w_{j\alpha})=-\frac{2n+1}{2}\cdot w_{j\alpha},$$
        $$\kappa(z_{j\alpha},z_{k\beta})=-\frac{1}{2}\cdot z_{j\beta}z_{k\alpha}, \ \kappa(w_{j\alpha},w_{k\beta})=-\frac{1}{2}\cdot w_{j\beta}w_{k\alpha},$$
$$\tau(\overline{z}_{j\alpha})=-\frac{2n+1}{2}\cdot \overline{z}_{j\alpha}, \ \tau(\overline{w}_{j\alpha})=-\frac{2n+1}{2}\cdot \overline{w}_{j\alpha},$$
$$\kappa(\overline{z}_{j\alpha},\overline{z}_{k\beta})=-\frac{1}{2}\cdot \overline{z}_{j\beta}\overline{z}_{k\alpha}, \ \kappa(\overline{w}_{j\alpha},\overline{w}_{k\beta})=-\frac{1}{2}\cdot \overline{w}_{j\beta}\overline{w}_{k\alpha},$$
    $$\kappa(\overline{z}_{j\alpha},\overline{w}_{k\beta})=-\frac{1}{2}\cdot \overline{w}_{j\beta}\overline{z}_{k\alpha},\  \kappa({z}_{j\alpha},\overline{z}_{k\beta})=\frac{1}{2}\cdot( w_{j\beta}\overline{w}_{k\alpha}+\delta_{jk}\delta_{\alpha\beta}),$$
   $$ \kappa({z}_{j\alpha},\overline{w}_{k\beta})=-\frac{1}{2}\cdot w_{j\beta}\overline{w}_{k\alpha},\ \kappa({w}_{j\alpha},\overline{z}_{k\beta})=-\frac{1}{2}\cdot w_{j\beta}\overline{z}_{k\alpha},$$
    $$\kappa({w}_{j\alpha},\overline{w}_{k\beta})=\frac{1}{2}\cdot( z_{j\beta}\overline{z}_{k\alpha}+\delta_{jk}\delta_{\alpha\beta}),\ \kappa(z_{j\alpha},w_{k\beta})=-\frac{1}{2}\cdot w_{j\beta}z_{k\alpha}.$$

\end{lemma}
From Lemma \ref{quatgrasscoord}, Ghandour and Gudmundsson obtained the following result.
\begin{lemma}{\rm \cite{Gha-Gud-5}}\label{efspspsp}
For $1\leq j<\alpha\leq m+n$ we define the complex-valued functions $\hat{\phi}_{j\alpha}:\Sp{m+n}\rightarrow\cn$ by $$\hat{\phi}_{j\alpha}(q)=\sum_{t=1}^m(z_{jt}\Bar{z}_{\alpha t}+w_{jt}\Bar{w}_{\alpha t}).$$ Then the tension field $\tau $ and the conformality operator $\kappa $ on $\Sp{m+n}$ satisfy $$\tau (\hat{\phi}_{j\alpha})=-2(m+n)\cdot\hat{\phi}_{j\alpha}, \quad \kappa (\hat{\phi}_{j\alpha},\hat{\phi}_{k\alpha})=-\hat{\phi}_{j\alpha}\hat{\phi}_{k\alpha}, $$ where $j,k\neq\alpha.$
\end{lemma}

\begin{remark}
    The functions $\hat{\phi}_{j\alpha}$ are all $\Sp{m}\times\Sp{n}$-invariant and hence induce functions $\phi_{j\alpha}:\Sp{m+n}/\Sp{m}\times\Sp{n}\rightarrow\cn$ on the quotient space.
\end{remark}
Lemma \ref{efspspsp} led Ghandour and Gudmundsson to an eigenfamily on the quaternionic Grassmannians.
\begin{theorem}{\rm \cite{Gha-Gud-5}}
For a fixed natural number $1\leq r<m+n,$ the set 
$$\mathcal{E}_r=\{\phi_{j\alpha}:\Sp{m+n}/\Sp{m}\times\Sp{n}\rightarrow\cn \ \rvert \ 1\leq j\leq r<\alpha\leq m+n\}$$ is an eigenfamily on the quaternionic Grassmannian $\Sp{m+n}/\Sp{m}\times\Sp{n}$ such that the tension field $\tau$ and conformality operator $\kappa$ satisfy
$$\tau(\phi)=-2(m+n)\cdot\phi,\quad \kappa(\phi,\psi)=-\phi\cdot\psi$$ for all $\phi,\psi\in\mathcal{E}_r.$
\end{theorem}
In the following, we will generalise Lemma \ref{efspspsp}. Here, we eased our restrictions on the indices $j$ and $\alpha.$
\begin{lemma}\label{quatgrassphigeneral}
    For $1\leq j,\alpha,k,\beta\leq m+n,$ we define the complex-valued functions $\hat{\phi}_{j\alpha}:\Sp{m+n}\rightarrow\cn$ by $$\hat{\phi}_{j\alpha}(q)=\sum_{t=1}^m(z_{jt}\Bar{z}_{\alpha t}+w_{jt}\Bar{w}_{\alpha t}).$$ Then, the tension field $\tau $ and the conformality operator $\kappa $ on $\Sp{m+n}$ satisfy $$\tau (\hat{\phi}_{j\alpha})=-2(m+n)\cdot\hat{\phi}_{j\alpha}+2\cdot\delta_{j\alpha}, \quad \kappa (\hat{\phi}_{j\alpha},\hat{\phi}_{k\beta})=-\hat{\phi}_{j\beta}\hat{\phi}_{k\alpha}+\frac{1}{2}\cdot\hat{\phi}_{k\alpha}\delta_{j\beta}+\frac{1}{2}\cdot\hat{\phi}_{j\beta}\delta_{k\alpha}. $$ 
\end{lemma}
\begin{proof}
    The result follows from a computation employing Lemma \ref{quatgrasscoord}.
\end{proof}
The following Theorem is inspired by Theorems \ref{thsososo} and \ref{theigenfunctionununun} in the settings of the real and complex Grassmannians respectively.
\begin{theorem}\label{thspspspnew}
Let $a\in\cn^{m+n}\backslash\{0\}$ satisfy $(a,a)=0.$ Then the complex-valued function $\hat{\Phi}_a:\Sp{m+n}\rightarrow\cn$ given by
$$\hat{\Phi}_a=\sum_{j,\alpha}a_ja_\alpha\hat{\phi}_{j\alpha}$$ satisfies $$\tau (\hat{\Phi}_a)=-2(m+n)\cdot\hat{\Phi}_a\ \textrm{and}\ \kappa (\hat{\Phi}_a,\hat{\Phi}_a)=-\hat{\Phi}_a^2.$$ Since $\hat{\Phi}_a$ is $\Sp{m}\times\Sp{n}$-invariant, it induces an eigenfunction $\Phi_a$ on the quaternionic Grassmannian $\Sp{m+n}/\Sp{m}\times\Sp{n}.$
\end{theorem}
\begin{proof}
    We will make use of Lemma \ref{quatgrassphigeneral}. For the tension field, we obtain the following computation.
    \begin{eqnarray*}
        \tau (\hat{\Phi}_a)&=&\tau \left(\sum_{j,\alpha}a_ja_\alpha\hat{\phi}_{j\alpha}\right)\\
        &=&\sum_{j,\alpha}a_ja_\alpha\tau (\hat{\phi}_{j\alpha})\\
    &=&\sum_{j,\alpha}a_ja_\alpha(-2(m+n)\cdot\hat{\phi}_{j\alpha}+2\cdot\delta_{j\alpha})\\
        &=&-2(m+n)\hat{\Phi}_a+2\cdot\sum_j a_j^2\\
        &=&-2(m+n)\hat{\Phi}_a.
    \end{eqnarray*}
    Lastly, we compute the conformality operator.
    \begin{eqnarray*}
        \kappa (\hat{\Phi}_a,\hat{\Phi}_a)&=&\kappa \left(\sum_{j,\alpha}a_ja_\alpha\hat{\phi}_{j\alpha},\sum_{k,\beta}a_ka_\beta\hat{\phi}_{k\beta}\right)\\
        &=&\sum_{j,\alpha,k,\beta}a_ja_\alpha a_ka_\beta\cdot\kappa (\hat{\phi}_{j\alpha},\hat{\phi}_{k\beta})\\
        &=&\sum_{j,\alpha,k,\beta}a_ja_\alpha a_ka_\beta\cdot\left(-\hat{\phi}_{j\beta}\hat{\phi}_{k\alpha}+\frac{1}{2}\cdot\hat{\phi}_{k\alpha}\delta_{j\beta}+\frac{1}{2}\cdot\hat{\phi}_{j\beta}\delta_{k\alpha}\right)\\
       &=&-\hat{\Phi}_a^2+\sum_{j,\alpha,k,\beta}a_ja_\alpha a_ka_\beta\cdot\left(\frac{1}{2}\cdot\hat{\phi}_{k\alpha}\delta_{j\beta}+\frac{1}{2}\cdot\hat{\phi}_{j\beta}\delta_{k\alpha}\right)\\
       &=&-\hat{\Phi}_a^2+\sum_{j}a_j^2\cdot\hat{\Phi}_a\\
       &=&-\hat{\Phi}_a^2.
    \end{eqnarray*}
\end{proof}

\section{Minimal Submanifolds of $\Sp{m+n}/\Sp{m}\times\Sp{n}$}\label{sectionmsspspsp}
We first define a basis of the tangent space of $\Sp{m+n}/\Sp{m}\times\Sp{n}$. We use this to show that the eigenfunctions given in Section \ref{efspspspsection} are not regular. 
\begin{proposition}\label{liealgspspsp}
    For $G=\Sp{m+n},$ and $K=\Sp{m}\times\Sp{n},$ we have the orthogonal decomposition $\g=\k\oplus\m,$ where $\k$ is the Lie algebra of $K$ and $\m$ is the subspace associated with $M=\Sp{m+n}/\Sp{m}\times\Sp{n}.$ It follows that an 
orthonormal basis of $\m$ is given by 
\begin{eqnarray*}
    \mathcal{B}_{\m}
    &=&\left\{Y_{rs}^{a}=\frac{1}{\sqrt{2}}\begin{pmatrix}
        Y_{rs}&0\\
        0&Y_{rs}
    \end{pmatrix}, \ X_{rs}^{a}=\frac{1}{\sqrt{2}}\begin{pmatrix}
        iX_{rs}&0\\
        0&-iX_{rs}
    \end{pmatrix},\right.\\
    &&\ \ X_{rs}^{b}=\frac{1}{\sqrt{2}}\begin{pmatrix}
        0&iX_{rs}\\
        iX_{rs}&0
    \end{pmatrix},   X_{rs}^{c}=\frac{1}{\sqrt{2}}\begin{pmatrix}
        0&X_{rs}\\
        -X_{rs}&0
    \end{pmatrix}\\
    && \ \ \rvert \ 1\leq r\leq m<s\leq m+n\Big\}.
\end{eqnarray*}
\end{proposition}
\begin{proof}
    In Proposition \ref{liealgebraspn}, we have seen that a basis of the Lie algebra $\sp{m+n}$ of $\Sp{m+n}$ is given by \begin{eqnarray*}
    \mathcal{B}_{\sp{m+n}}&=&\left\{Y_{rs}^{a}=\frac{1}{\sqrt{2}}\begin{pmatrix}
        Y_{rs}&0\\
        0&Y_{rs}
    \end{pmatrix}, \ X_{rs}^{a}=\frac{1}{\sqrt{2}}\begin{pmatrix}
        iX_{rs}&0\\
        0&-iX_{rs}
    \end{pmatrix},\right. \\
    &&\ \ X_{rs}^{b}=\frac{1}{\sqrt{2}}\begin{pmatrix}
        0&iX_{rs}\\
        iX_{rs}&0
    \end{pmatrix}, X_{rs}^{c}=\frac{1}{\sqrt{2}}\begin{pmatrix}
        0&X_{rs}\\
        -X_{rs}&0
    \end{pmatrix} \ \\ 
    &&\ \ \rvert \ 1\leq r<s\leq m+n\Big\}.
\end{eqnarray*}
It is clear that if $1\leq r<s\leq m$ or $m+1\leq r<s\leq m+n,$ then the elements $Y_{rs}^{a},X_{rs}^{a},X_{rs}^{b},X_{rs}^{c}$ belong to $\mathcal{B}_\k.$ The fact that $\mathcal{B}_{\sp{m+n}}$ is the disjoint union of $\mathcal{B}_\k$ and $\mathcal{B}_\m$ now proves the statement.
\end{proof}

In the following, we will show that the eigenfunctions defined on Lemma \ref{efspspsp} are not regular over $0\in\cn.$ In the proof, we will proceed similar to Theorems \ref{exsososo2} and \ref{thunununnotregular}.
\begin{theorem}
    Let $1\leq j<\alpha\leq m+n.$ Let $\phi_{j\alpha}:\Sp{m+n}/\Sp{m}\times\Sp{n}\rightarrow\cn$ be the eigenfunction on the quaternionic Grassmannians induced by the function $\hat{\phi}_{j\alpha}:\Sp{m+n}\rightarrow\cn$ defined by $$\hat{\phi}_{j\alpha}(q)=\sum_{t=1}^m(z_{jt}\Bar{z}_{\alpha t}+w_{jt}\Bar{w}_{\alpha t}).$$ Then $\phi_{j\alpha}$ is not regular over $0\in\cn.$
\end{theorem}
\begin{proof}
    We know from Lemma \ref{efspspsp} that $\phi_{j\alpha}:\Sp{m+n}/\Sp{m}\times\Sp{n}\rightarrow\cn$ is indeed an eigenfunction, which justifies the statement. We wish to construct an element $q=z+jw\in\Sp{m+n}$ such that $\hat{\phi}_{j\alpha}(q)=0$ and for every $X\in\mathcal{B}_\m$ as specified in Proposition \ref{liealgspspsp}, we have that $X(\hat{\phi}_{j\alpha})=0$ at $q.$ 
    A computation shows that for $1\leq r\leq m<s\leq m+n,$ the following holds.
\begin{eqnarray*}
    Y_{rs}^{a}(\hat{\phi}_{j\alpha})&=&\frac{-1}{\sqrt{2}}(z_{js}\Bar{z}_{\alpha r}+z_{jr}\Bar{z}_{\alpha s}+w_{js}\Bar{w}_{\alpha r}+w_{jr}\Bar{w}_{\alpha s}),\\
    X_{rs}^{a}(\hat{\phi}_{j\alpha})&=&\frac{i}{\sqrt{2}}(z_{js}\Bar{z}_{\alpha r}-z_{jr}\Bar{z}_{\alpha s}-w_{js}\Bar{w}_{\alpha r}+w_{jr}\Bar{w}_{\alpha s}),\\
    X_{rs}^{b}(\hat{\phi}_{j\alpha})&=&\frac{i}{\sqrt{2}}(w_{js}\Bar{z}_{\alpha r}-z_{jr}\Bar{w}_{\alpha s}+z_{js}\Bar{w}_{\alpha r}-w_{jr}\Bar{z}_{\alpha s}),\\
    X_{rs}^{c}(\hat{\phi}_{j\alpha})&=&\frac{1}{\sqrt{2}}(-w_{js}\Bar{z}_{\alpha r}-z_{jr}\Bar{w}_{\alpha s}+z_{js}\Bar{w}_{\alpha r}+w_{jr}\Bar{z}_{\alpha s}).
\end{eqnarray*}
Hence $$Y_{rs}^{a}(\hat{\phi}_{j\alpha})=0, \ X_{rs}^{a}(\hat{\phi}_{j\alpha})=0,\ X_{rs}^{b}(\hat{\phi}_{j\alpha})=0, \ \textrm{and}\ X_{rs}^{c}(\hat{\phi}_{j\alpha})=0$$ if and only if \begin{eqnarray}
    0&=&z_{js}\Bar{z}_{\alpha r}+w_{jr}\Bar{w}_{\alpha  s},\label{spspspeqns1}\\ 
    0&=&z_{jr}\Bar{z}_{\alpha s}+w_{js}\Bar{w}_{\alpha r},\label{spspspeqns2}\\
    0&=&w_{js}\Bar{z}_{\alpha r}-w_{jr}\Bar{z}_{\alpha s},\label{spspspeqns3}\\
    0&=&z_{jr}\Bar{w}_{\alpha s}-z_{js}\Bar{w}_{\alpha r}.\label{spspspeqns4}
\end{eqnarray}
Let $e_j,e_\alpha$ be the $j$-th and $\alpha$-th standard unit vectors of $\rn^{m+n}$ respectively. Let $v_1,\dots,v_{m+n-2}$ denote $m+n-2$ pairwise orthogonal unit vectors of $\cn^{m+n}$, which are also orthogonal to $e_j,e_\alpha.$ Such vectors exist, for example we may just take the remaining elements of the canonical basis above. 

If $n\geq2,$ let $z=(v_1,\dots,v_{m+n-2},e_j,e_\alpha)$ and $w=0.$ By construction, $q=z+jw$ satisfies $$I_n=z\Bar{z}^t+\Bar{w}w^t, \quad 0=w\Bar{z}^t-\Bar{z}w^t.$$ Thus $q\in\Sp{m+n}.$ For every column $z_t$ of $z,$ we have that $$z_{jt}=\langle z_t,e_j\rangle $$ and $$\Bar{z}_{\alpha t}=\langle e_\alpha,z_t\rangle .$$ Since $v_1,\dots,v_{m+n-2}$ were chosen to satisfy $$\langle v_t,e_j\rangle =\langle e_\alpha,v_t\rangle =0$$ for every $1\leq t\leq m+n-2,$ we see that \begin{eqnarray*}
    \hat{\phi}_{j\alpha}(z+jw)&=&\sum_{t=1}^m(z_{jt}\Bar{z}_{\alpha t}+w_{jt}\Bar{w}_{\alpha t})\\
    &=&\sum_{t=1}^m\langle z_t,e_j\rangle \langle e_\alpha,z_t\rangle \\
    &=&0.
\end{eqnarray*}
Thus, $\hat{\phi}_{j\alpha}(q)=0.$
Further, for every $1\leq r\leq m, $ clearly $z_{jr}=\langle z_r,e_j\rangle =0$ and $\Bar{z}_{\alpha r}=\langle e_\alpha,z_r\rangle =0.$ Since all entries of $w$ vanish, Equations \ref{spspspeqns1} through \ref{spspspeqns4} are satisfied. 
We conclude that $\hat{\phi}_{j\alpha}$ is not regular at $q$.
This concludes the proof for $n\geq2.$

In the case that $n=1,$ we let $z=(e_j,e_\alpha,v_1,\dots,v_{m+n-2})$ and $w=0.$ A similar computation as above shows that $\hat{\phi}_{j\alpha}$ is not regular at $z+jw,$ and $\hat{\phi}_{j\alpha}(z+jw)=0$.

\end{proof}

\appendix
\chapter{Notation and Useful Lemmas}

\begin{remark}
    We denote $$(\cdot,\cdot):\cn^n\times\cn^n\rightarrow\cn,$$to be the bilinear form $$(x,y)=\sum_{k=1}^n x_ky_k.$$ When restricted to $\rn^n\times\rn^n,$ this yields the standard inner product on $\rn^n.$ We call a vector $a\in\cn^n$ isotropic if it satisfies $(a,a)=0.$

    The standard inner product on $\cn^n$ $$\langle \cdot,\cdot\rangle :\cn^n\times\cn^n\rightarrow\cn$$ is given by $$\langle z,w\rangle =\sum_{k=1}^n z_k\Bar{w}_k.$$ 
\end{remark}
 The following statement is widely known and the proof is a standard exercise in Linear Algebra. 
 
\begin{lemma}\label{lemmaiso}
Let $$(\cdot,\cdot):\cn^n\times\cn^n\rightarrow\cn,$$ be the bilinear form $$(x,y)=\sum_{k=1}^n x_ky_k.$$
    Let $a=u+i\cdot v\in\cn^n$ be a non-zero vector, where $u,v\in\rn^n.$ If $(a,a)=0,$ then $u,v$ are both non-zero, orthogonal and linearly independent. Further, $$|u|^2=|v|^2.$$
\end{lemma}
\begin{proof}
   Let $a\in\cn\backslash\{0\}$ and assume $(a,a)=0.$ Note that
       $$0=(a,a)=(u,u)+2i(u,v)-(v,v).$$ It follows that $|u|^2=|v|^2$ and $(u,v)=0.$ Since $a\neq0,$ at least one of $u,v$ is non-zero. But from the previous assertion it follows that both $u,v$ are non-zero.
       If $u,v$ were linearly dependent, then $v=\lambda\cdot u$ for some $\lambda\in\rn\backslash\{0\}.$ But then $$0=(u,v)=(u,\lambda\cdot u)=\lambda\cdot(u,u)$$ implies that both $u=v=0,$ since $(\cdot,\cdot)$ defines the standard inner product on $\rn^n\times\rn^n.$
       Thus, $u$ and $v$ must be linearly independent.
\end{proof}
The following lemma was used in the proof of Theorem \ref{thmsso2nun} and is covered in basic level courses in Linear Algebra or Matrix Theory.
\begin{lemma}\label{SIinvertible}
    Suppose that $S$ is a skew-symmetric and real-valued $n\times n$ matrix. Then $S+I$ is invertible.
\end{lemma}
\begin{proof}
    Assume towards a contradiction that $S+I$ is not invertible. Then there exists a non-zero vector $x\in\rn^n$ such that $$(S+I)\cdot x=S\cdot x+x=0.$$ Equivalently, $$S\cdot x=-x.$$ Consider the following computation.
    \begin{eqnarray*}
        |x|^2&=&x^tx\\
        &=&-x^tSx\\
        &=&x^tS^tx\\
        &=&-x^tx\\
        &=&-|x|^2.
    \end{eqnarray*}
    Hence $x=0,$ which is a contradiction. We conclude that $S+I$ must be invertible.
\end{proof}

The next statement was used in the chapter on the real Grassmannians. This is well known and is partially discussed in Gudmundsson and Ghandour's work \cite{Gha-Gud-4}.

\begin{lemma}\label{lemmaAaat}
    Let $A\in\cn^{n\times n}\backslash\{0\}$ be a symmetric matrix. Then $\rank A=1$ if and only if $A=aa^t$ for some $a\in\cn^n\backslash\{0\}.$ In this case, $\trace A=0$ if and only if $a$ is isotropic, and further, $A^2=0$.
\end{lemma}
\begin{proof}
    Let $a\in\cn^n$ and $A=aa^t.$ Since any two rows of $A$ are multiples of one another, the rank of $A$ is one. 
    
    Conversely, assume $A$ has rank one. Then the columns are all multiples of some $v\in\cn^n.$ Hence for the standard basis $e_1,\dots,e_n,$ $$A\cdot e_i=\lambda_i\cdot v.$$ But this means that if we write $w=(\lambda_1,\dots\lambda_n)^t,$ $A=vw^t.$ It can now be shown that due to symmetry, $v=w.$
    
    The second statement follows immediately from the fact that $$(a,a)=a_1^2+\dots+a_{m+n}^2=\trace(A).$$
    If $(a,a)=0$, then for all $1\leq i,j\leq n,$ $$(A^2)_{ij}=\sum_{t=1}^n A_{it}A_{tj}=a_ia_j\sum_{t=1}^n a_t^2=0.$$ This shows that $A^2=0.$
 \end{proof}

%\chapter{Computations in Maple}
%\newpage

%\pagenumbering{gobble}
%\begin{figure}[htbp]
%\hspace*{-2cm}   
%    \adjustimage{scale=1,left}{Maple-1.eps}
%\end{figure}

%%%%%%%%%%%%%%%%%%%%%%%%%%%%%%%  %%%%%%%%%%%%%%%%%%%%%%%

\backcover

\end{document}